\documentclass[aos]{imsart}

\usepackage{amsmath, amsfonts, amssymb, amsthm,  graphicx, enumerate, dsfont}

\usepackage[numbers, square]{natbib}

\usepackage[%
            CJKbookmarks=true,
            bookmarksnumbered=true,
            bookmarksopen=true,
            colorlinks=true,
            citecolor=red,
            linkcolor=blue,
            anchorcolor=red,
            urlcolor=blue,
            ]{hyperref}

\usepackage[usenames,dvipsnames]{color}

\usepackage[ruled, vlined, lined, commentsnumbered]{algorithm2e}

\usepackage{prettyref,soul,xspace}

\usepackage{tikz}
\usepackage{pgfplots,makecell}

\usepackage{bm}

\theoremstyle{plain}

\newtheorem{lemma}{Lemma}[section]
\newtheorem{theorem}{Theorem}[section]
\newtheorem{proposition}{Proposition}[section]
\newtheorem{corollary}{Corollary}[section]

\newcommand{\p}{\mathbb{P}}
\newcommand{\E}{\mathbb{E}}

\newcommand{\iid}{\stackrel{iid}{\sim}}

\newcommand{\br}[1]{\left( #1 \right)}

\newcommand{\cbr}[1]{\left\{ #1 \right\}}
\newcommand{\pbr}[1]{\p\left( #1 \right)}
\newcommand{\ebr}[1]{\exp\left( #1 \right)}
\newcommand{\abs}[1]{\left| #1 \right|}

\newcommand{\tr}[1]{\text{tr}\br{#1}}

\newcommand{\mathr}{\mathbb{R}}
\newcommand{\matho}{\mathbb{O}}
\newcommand{\mathn}{\mathcal{N}}
\newcommand{\mathf}{\mathcal{F}}

\newcommand{\diff}{{\rm d}}

\newcommand{\indic}[1]{{\mathbb{I}\left\{{#1}\right\}}}
\newcommand{\iprod}[2]{\left \langle #1, #2 \right\rangle}
\newcommand{\norm}[1]{\left\|{#1} \right\|}

\newcommand{\fnorm}[1]{\norm{#1}_{\rm F}}
\newcommand{\normf}[1]{\|#1\|_{\rm F}}
\newcommand{\normop}[1]{\|#1\|}
\newcommand{\argmin}{\mathop{\rm argmin}}

\newcommand{\tspan}{\text{span}}
\newcommand{\cd}[1]{{\cdot , #1}}

\newcommand{\tpo}{{\br{1}}}
\newcommand{\tpt}{{\br{2}}}

\newcommand{\one}{\mathds{1}}

\newcommand{\lrt}{\text{LRT}}
\newcommand{\kr}{{\kappa}}
\newcommand{\rr}{{r}}
\newcommand{\fc}{{\mathcal{I}}}

\begin{document}
\begin{frontmatter}
\title{Leave-one-out Singular Subspace Perturbation Analysis for Spectral Clustering}
\runtitle{Singular Subspace Perturbation and Spectral Clustering}

\begin{aug}
\author[A]{\fnms{Anderson Y.} \snm{Zhang}\thanksref{t1}\ead[label=e1]{ayz@wharton.upenn.edu}}
\and
\author[B]{\fnms{Harrison Y.} \snm{Zhou}\thanksref{t2}\ead[label=e2]{huibin.zhou@yale.edu}}
\thankstext{t1}{Research supported in part by NSF grant DMS-2112988.}
\thankstext{t2}{Research supported in part by NSF grant DMS-2112918.}

\address[A]{Department of Statistics and Data Science,
University of Pennsylvania,
\printead{e1}}

\address[B]{Department of Statistics and Data Science,
Yale University,
\printead{e2}}
\end{aug}

\begin{abstract}
The singular subspaces perturbation theory is of fundamental importance in probability and statistics. It has various applications across different fields. We consider two arbitrary matrices where one is a leave-one-column-out submatrix of the other one and establish a novel perturbation upper bound for the distance between the two corresponding singular subspaces. It is well-suited for mixture models and results in a sharper and finer statistical analysis than classical perturbation bounds such as Wedin's Theorem. Empowered by this leave-one-out perturbation theory, we provide a deterministic entrywise analysis for the performance of spectral clustering under mixture models. Our analysis leads to an explicit exponential error rate for spectral clustering of sub-Gaussian mixture models. For the mixture of isotropic Gaussians, the rate is optimal under a weaker signal-to-noise condition than that of L{\"o}ffler et al. (2021).

\end{abstract}



\end{frontmatter}

\section{Introduction}

The matrix  perturbation theory \cite{stewart1990matrix, bai2010spectral} is a central topic in  probability and statistics. It plays a fundamental role in
spectral methods \cite{chen2021spectral, kannan2009spectral}, an umbrella term for algorithms involving eigendecomposition or singular value decomposition. It has a wide range of applications including principal component analysis \cite{abbe2020ell_p, cai2021optimal}, covariance matrix estimation \cite{fan2018eigenvector},  clustering \cite{von2008consistency, rohe2011spectral, schiebinger2015geometry, ndaoud2018sharp}, and matrix completion \cite{ma2018implicit, ding2020leave}, 
 throughout different fields such as machine learning \cite{belabbas2009spectral},  network science \cite{newman2013spectral, abbe2020entrywise}, and genomics \cite{kiselev2019challenges}.

Perturbation analysis for eigenspaces and singular subspaces dates back to seminal works of Davis and Kahan \cite{davis1970rotation} and Wedin \cite{wedin1972perturbation}. Davis-Kahan Theorem provides a clean bound for eigenspaces in terms of operator norm and Frobenius norm, and Wedin further  extends it to singular subspaces. In recent years, there has been growing literature in developing fine-grained $\ell_{\infty}$ analysis for singular vectors \cite{abbe2020entrywise, fan2018eigenvector} and $\ell_{2,\infty}$ analysis for singular subspaces \cite{lei2019unified, cape2019two, cai2021subspace, agterberg2021entrywise}, which often lead to sharp upper bounds. For clustering problems, they can be used to establish the exact recovery of spectral methods, but are usually not suitable for low signal-to-noise ratio regimes where only partial recovery is possible.

In this paper, we consider a special matrix perturbation case where one matrix differs from the other one by having one less column and investigate the difference between two corresponding left singular subspaces. Consider two matrices
\begin{align}\label{eqn:Y_hat_Y}
Y=(y_1,\ldots, y_{n-1})\in\mathr^{p\times (n-1)} \text{ and }\hat Y = (y_1,\ldots, y_{n-1}, y_n)\in\mathr^{p\times n},
\end{align}
where $Y$ is a leave-one-column-out submatrix of $\hat Y$ with the last column removed.
Let $U_\rr $ and $\hat U_\rr $
 include the leading $\rr $ left singular vectors of $Y$ and $\hat Y$, respectively. The two corresponding left singular subspaces are $\tspan(U_\rr )$ and $\tspan(\hat U_\rr )$, where  the former one can be interpreted as a leave-one-out counterpart of the latter.

We establish a novel upper bound for the Frobenius norm of  $\hat U_\rr \hat U_\rr ^T-U_\rr U_\rr ^T$ to quantify the distance between the two singular subspaces   $\tspan(U_\rr )$ and $\tspan(\hat U_\rr )$. A direct application of the generic Wedin's Theorem leads to a ratio of the magnitude of perturbation $(I-U_\rr U_\rr ^T)y_n$ to the corresponding spectral gap $\sigma_\rr -\sigma_{\rr +1}$.  We go beyond Wedin's Theorem and  reveal that the  interplay between $U_\rr U_\rr ^T y_n$ and $(I- U_\rr U_\rr ^T )y_n$ plays a crucial role. Our new upper bound is  a product of the aforementioned ratio and a factor determined $U_\rr ^T y_n$. That is, informally (see Theorem \ref{thm:general} for a precise statement),
\begin{align*}
\fnorm{\hat U_\rr \hat U_\rr ^T-U_\rr U_\rr ^T}\lesssim \frac{\norm{(I-U_\rr U_\rr ^T)y_n}}{\sigma_\rr -\sigma_{\rr +1}} \times \text{ a factor from }U_\rr ^T y_n.
\end{align*}
When this factor is smaller than some constant,  it results in a sharper upper bound than Wedin's Theorem. The derived upper bound is particularly suitable for mixture models where the contributions of $U_\rr ^T y_n$ are well-controlled, and consequently provides a key toolkit for the follow-up statistical analysis on  spectral clustering.

Spectral clustering is one of the most popular  approaches to group high-dimensional data. It first reduces the dimensionality of data by only using a few of its singular components
and then applies a classical clustering method, such as $k$-means, 
to the data of reduced dimension. It is computationally appealing and often delivers  remarkably good performance, and has been widely used in various problems. In recent years there has been growing interest in theoretical properties of spectral clustering, noticeably in community detection \cite{abbe2020entrywise, lei2015consistency, Jin15, qin2013regularized, rohe2011spectral, zhou2019analysis, han2020exact, ndaoud2021improved, lei2020bias}. In spite of various polynomial-form upper bounds in terms of signal-to-noise ratios for the performance of spectral clustering, sharper exponential error rates are established in literature only for a few special scenarios, such as  Stochastic Block Models with two equal-size communities \cite{abbe2020entrywise}. Spectral clustering is also investigated  in mixture models \cite{ndaoud2018sharp, loffler2019optimality, abbe2020ell_p, davis2021clustering, wang2020efficient, srivastava2019robust, blum2007separating}.
For isotropic Gaussian mixture models, 
\cite{loffler2019optimality} 
shows spectral clustering achieves the optimal minimax rate. However, the proof technique used in \cite{loffler2019optimality} is very limited to the isotropic Gaussian noise and it is unclear whether it is possible to be extended to either sub-Gaussian distributed errors or unknown covariance matrices. Spectral clustering for sub-Gaussian mixture models is studied in \cite{abbe2020ell_p}, but only under  special assumptions on the spectrum and geometry of the centers. It requires  eigenvalues of the Gram matrix of centers to be all of the same order and sufficiently large, which rules out many interesting cases.

We study the theoretical performance of the spectral clustering under general mixture models where each observation $X_i$ is equal to one of $k$ centers plus some noise $\epsilon_i$. The spectral clustering first projects $X_i$ onto $\hat U_{1:r}^TX_i$ where $\hat U_{1:r}$ includes the leading $r$ left singular vectors of the data matrix, and then performs $k$-means on this low-dimensional space. Building upon our leave-one-out perturbation theory, we provide a deterministic entrywise analysis for the spectral clustering.
We demonstrate that the correctness of $X_i$'s clustering is determined by $\hat U_{-i,1:r}^T\epsilon_i$, where $\hat U_{-i,1:r}$ is the leave-one-out counterpart of $\hat U_{1:r}$ that uses all the observations except $X_i$.
The independence between $\hat U_{-i,1:r}$  and $\epsilon_i$ enables us to derive explicit error risks when the noises are randomly generated from certain distributions. Specifically:

\begin{enumerate}
\item For sub-Gaussian mixture models, we establish an exponential error rate for the performance of the spectral clustering, assuming the centers are separated from each other and the smallest non-zero singular value is away from zero.  Compared to \citep{abbe2020ell_p}, our assumptions on the spectrum and geometric distribution of the centers are weaker. In addition, we obtain an explicit constant $1/8$ in the exponent, which is sharp when the noises are further assumed to be isotropic Gaussian.
To  remove the spectral gap condition, we propose a variant of the spectral clustering where the number of singular vectors used is selected adaptively. 
\item For Gaussian mixture models with isotropic covariance matrix, we fully recover the results of \cite{loffler2019optimality}.
Empowered by the leave-one-out perturbation theory, our proof adopts a completely different approach  and is much shorter compared to that of \cite{loffler2019optimality}. In addition, the signal-to-noise ratio condition of \cite{loffler2019optimality} is improved.
\item For a two-cluster symmetric mixture model where coordinates of the noise $\epsilon_i$ are independently and identically distributed, we provide a matching upper and lower bound for the performance of the spectral clustering. This sharp analysis provides an answer to the optimality of the spectral clustering in this setting: it is in general sub-optimal and is optimal only if each coordinate  of $\epsilon_i$ is normally distributed.
\end{enumerate}

~\\
\indent\emph{Organization.} The structure of this paper is as follows. In Section \ref{sec:perturbation}, we first establish a general leave-one-out perturbation theory for singular subspaces, followed by its application in mixture models. In Section \ref{sec:spectral_clustering_mixture_model}, we use our leave-one-out perturbation theory to provide theoretical guarantees for the spectral clustering under mixture models. We discuss extensions and potential caveats of our analysis in Section \ref{sec:discuss}.
The proofs of main results in Section \ref{sec:perturbation} and Section \ref{sec:spectral_clustering_mixture_model}  are given in Section \ref{sec:proof_perturbation_sec_2} and  in Section \ref{sec:proof_main_thm}, respectively. 
All other proofs can be found in the  supplement \cite{supplement}.

~\\
\indent\emph{Notation.} For any positive integer $r$, let $[r]=\{1,2,\ldots,r\}$. For two scalars $a,b\in\mathr$, denote $a\wedge b= \min\{a,b\}$.  For two matrices $A=(A_{i,j})$ and $B=(B_{i,j})$, we denote $\iprod{A}{B}=\sum_{i,j}A_{i,j}B_{i,j}$ to be the trace product, $\norm{A}$ to be its operator norm,  $\fnorm{A}$ to be its Frobenius norm, and $\tspan(A)$ to be the linear space spanned by columns of $A$. If both $A,B$ are symmetric, we write $A\prec B$ if $B-A$ is positive semidefinite. For scalars $x_1,\ldots,x_d$, we denote $\text{diag}(x_1,\ldots,x_d)$ to be a $d\times d$ diagonal matrix with diagonal entries being $x_1,\ldots, x_d$. For any integers $d,p\geq 0$, we denote $0_{d}\in\mathr^d$ to be a vector with all coordinates being 0, $\one_d\in\mathr^d$ to be a vector with all coordinates being 1, and  $O_{d\times p}\in\mathr^{d\times p}$ to be a matrix with all entries being 0. We denote $I_{d\times d}$ and $I_d$ to be the $d\times d$ identity matrix and we use $I$ for short when the dimension of clear according to context.  Let $\matho^{d\times p}=\cbr{V\in\mathr^{d\times p}: V^TV=I}$ be the set of matrices in $\mathr^{d\times p}$ with orthonormal columns.  We denote $\indic{\cdot}$ to be the indicator function.  For two positive sequences $\{a_n\}$ and $\{b_n\}$, $a_n\lesssim b_n$, $a_n=O(b_n)$, $b_n\gtrsim a_n$ all mean $a_n\leq Cb_n$ for some constant $C>0$ independent of $n$. We also write $a_n=o(b_n)$ when $\limsup_{n\rightarrow\infty}\frac{a_n}{b_n}=0$.
For a random variable $X$, we say $X$ is sub-Gaussian with variance proxy $\sigma^2$ (denoted as $X\sim\text{SG}(\sigma^2)$) if $\E e^{tX}\leq \ebr{\sigma^2t^2/2}$ for any $t\in \mathr$. For a random vector $X\in\mathr^d$, we say $X$ is sub-Gaussian with variance proxy $\sigma^2$ (denoted as $X\sim\text{SG}_d(\sigma^2)$) if $u^TX\sim\text{SG}(\sigma^2)$ for any unit vector $u\in\mathr^d$.

\section{Leave-one-out  Singular  Subspace Perturbation Analysis}\label{sec:perturbation}

Classical singular subspace perturbation theory examines the relationship between the singular spaces of two matrices of the same dimension. However, prevailing upper bounds, such as those given by Wedin's Theorem, often achieve tightness only in worst-case scenarios. They can be sub-optimal, especially in situations like the one considered in this paper where one matrix is short of one column relative to the other.

In the domains of statistics and data science, it's common to work with data matrices wherein columns represent independent and identically distributed observations. Intuitively, when the number of observations is large, omitting a single observation should have minimal impact on the singular subspace. This intuition can guide entrywise perturbation analyses for spectral methods. As a case in point, the efficacy of spectral clustering under mixture models can largely be attributed to the perturbation of $\hat U_{1:r}^TX_i$, where $X_i$ represents the $i$th observation and $\hat U_{1:r}$ encompasses the leading $r$ left singular vectors of the data matrix. Directly analyzing $\hat U_{1:r}^TX_i$ is cumbersome due to the inherent dependence between $\hat U_{1:r}$ and $X_i$. To disentangle this dependence, a logical strategy is to substitute $\hat U_{1:r}$ with its leave-one-out counterpart, $\hat U_{-i,1:r}$, which is formed using all observations except $X_i$. The resulting independence between $\hat U_{-i,1:r}$ and $\epsilon_i$ facilitates a more precise characterization of the tail probabilities of $\hat U_{-i,1:r}^T X_i$. This, in turn, yields a more defined bound on spectral clustering's performance. Such an analytical approach presumes that $\hat U_{1:r}$ and its leave-one-out version $\hat U_{-i,1:r}$ are sufficiently similar.

With this foundation laid, in this section, we focus on establishing a comprehensive leave-one-out perturbation theory for singular subspaces.

\subsection{General Results}
Consider two matrices as in (\ref{eqn:Y_hat_Y})
 such that they are equal to each other except that $\hat Y$ has an extra last column. Let the Singular Value Decomposition (SVD) of these two matrices be
\begin{align*}
Y =  \sum_{i\in [p\wedge (n-1)]} \sigma_i u_i v_i^T \text{ and }\hat Y = \sum_{i\in [p\wedge n]} \hat \sigma_i \hat u_i\hat v_i^T,
\end{align*}
where $\sigma_1\geq \ldots \geq \sigma_{p\wedge (n-1)}$ and  $\hat \sigma_1\geq \ldots \geq \hat \sigma_{p\wedge n}$.  Consider any $\rr\in[p\wedge (n-1)]$. Define  
\begin{align*}
U_\rr  := (u_1,\ldots, u_\rr )\in\matho^{p\times \rr } \text{ and }\hat U_\rr  := (\hat u_1,\ldots, \hat u_\rr )\in\matho^{p\times \rr }
\end{align*}
to include the leading $\rr $ left singular vectors of $Y$ and $\hat Y$, respectively.
 Since $Y$ can be viewed as a leave-one-out submatrix of $\hat Y$ without the last column $y_n$, $U_\rr $ can be  interpreted as a leave-one-out counterpart of $\hat U_\rr $.
 
The two matrices $U_\rr ,\hat U_\rr $ correspond to two singular subspaces $\tspan(U_\rr ),\tspan(\hat U_\rr )$, respectively. The difference between these two subspaces 
can be captured by sin $\Theta$ distances,  $\normop{\text{sin}\; \Theta(\hat U_\rr ,U_\rr )}$ or $\normf{\text{sin}\; \Theta(\hat U_\rr ,U_\rr )}$, where $$\Theta(\hat U_\rr ,U_\rr ):=\text{diag}(\cos^{-1}(\alpha_1),\cos^{-1}(\alpha_2),\ldots, \cos^{-1}(\alpha_\rr ))$$ with $\alpha_1\geq \alpha_2\geq \ldots \geq \alpha_\rr \geq 0$ being the $\rr $ singular values of $\hat U_\rr ^T U_\rr $. It is known (see Lemma 1 of \cite{cai2018rate}) that $\normf{\hat U_\rr  \hat U_\rr ^T -   U_\rr U_\rr ^T}  = \sqrt{2}\normf{\text{sin}\; \Theta(\hat U_\rr ,U_\rr )}$. Throughout this section, we will focus on establishing sharp upper bounds for $\normf{\hat U_\rr  \hat U_\rr ^T -   U_\rr U_\rr ^T}$, i.e., the Frobenius norm of the difference between  two corresponding projection matrices  $U_\rr U_\rr ^T$ and $\hat U_\rr  \hat U_\rr ^T$.

Since the augmented matrix $Y':=(Y,U_\rr U_\rr ^T y_n)\in\mathr^{p\times n}$  concatenated by $Y$ and $U_\rr U_\rr ^Ty_n$ has the same leading $\rr$ left singular subspace and projection matrix as $Y$, a natural idea is to relate $\normf{\hat U_\rr  \hat U_\rr ^T -   U_\rr U_\rr ^T}$ with the difference $\hat Y - Y'$. The classical spectral perturbation theory such as  Wedin's Theorem \cite{yu2015useful, cai2018rate}  leads to that if $\sigma_\rr-\sigma_{\rr + 1}>2\norm{(I-U_\rr U_\rr ^T)y_n}$, then
\begin{align}\label{eqn:wedin}
\fnorm{\hat U_\rr  \hat U_\rr ^T -   U_\rr U_\rr ^T} \leq \frac{2\sqrt{2}\norm{(I-U_\rr U_\rr ^T)y_n}}{\sigma_\rr - \sigma_{\rr +1}}.
\end{align}
See Proposition \ref{prop:wedin} in the supplement for its proof.
The upper bound in (\ref{eqn:wedin}) requires the spectral gap $\sigma_\rr - \sigma_{\rr +1}$ is away from zero.
It also indicates the magnitude of the difference $\normop{\hat Y-Y'} = \normop{(I-U_\rr U_\rr ^T)y_n}$ plays a crucial role.
In spite of its simple form,
(\ref{eqn:wedin})  comes from generic spectral perturbation theories not specifically designed for  the setting (\ref{eqn:Y_hat_Y}).

In the following Theroem \ref{thm:general}, we provide a deeper and finer analysis for  $\normf{\hat U_\rr  \hat U_\rr ^T -   U_\rr U_\rr ^T}$, utilizing the fact that $\hat Y$ and $Y$  differ by only one column and exploiting the interplay between $U_\rr U_\rr ^Ty_n$ and $(I-U_\rr U_\rr ^T)y_n$.

\begin{theorem}\label{thm:general}
If
\begin{align}\label{eqn:general_condition3}
\rho:= \frac{\sigma_\rr -\sigma_{\rr +1}}{\normop{(I-U_\rr U_\rr ^T)y_n}} >2,
\end{align}
we have
\begin{align}\label{eqn:general_upper}
\fnorm{\hat U_\rr  \hat U_\rr ^T -   U_\rr  U_\rr ^T} \leq \frac{4\sqrt{2}}{\rho}\sqrt{\sum_{i=1}^\rr\br{ \frac{u_i^T y_n}{\sigma_i}}^2}.
\end{align}
\end{theorem}

Theorem \ref{thm:general} gives an upper bound on $\normf{\hat U_\rr  \hat U_\rr ^T -   U_\rr U_\rr ^T} $ that is essentially a product of $\rho^{-1}$ and some quantity determined by $\{\sigma_i^{-1}u_i^Ty_n\}_{i\in[\rr ]}$. 
Since $(\sigma_i^{-1}u_i^Ty_n)^2\leq \sigma_\rr^{-2}(u_i^Ty_n)^2$ for each $i\in[\rr ]$, (\ref{eqn:general_upper}) leads to a simpler upper bound 
\begin{align}\label{eqn:general_upper_simplified}
\normf{\hat U_\rr  \hat U_\rr ^T -   U_\rr U_\rr ^T}  \leq \frac{4\sqrt{2}}{\rho}\frac{\normop{U_\rr U_\rr ^T y_n}}{\sigma_\rr}.
\end{align}

The condition  (\ref{eqn:general_condition3}) in Theorem \ref{thm:general} can be understood as a spectral gap assumption as it needs the  gap $\sigma_\rr -\sigma_{\rr +1}$ to be larger than twice the magnitude of the perturbation $\normop{(I-U_\rr U_\rr ^T)y_n}$. This condition can be slightly weakened into $\sigma_\rr ^2 - \sigma_{\rr +1}^2 - \normop{(I-U_\rr U_\rr ^T)y_n}^2>0$, though resulting in a more involved upper bound. See Theorem \ref{thm:general_more} in Section \ref{sec:proof_general} for details.

We are ready to have a comparison of our result (\ref{eqn:general_upper}) and (\ref{eqn:wedin}) that is from  Wedin's  Theorem. Under the assumption (\ref{eqn:general_condition3}),  the upper bound  in (\ref{eqn:wedin}) can be written equivalently as $2\sqrt{2}\rho^{-1}$.
As a result, the comparison is about the magnitude of  
$(\sum_{i\in[\rr]}(\sigma_i^{-1} u_i^T y_n)^2)^{1/2}$.
If it is smaller than $1/2$, then  (\ref{eqn:general_upper}) gives a sharper upper bound than (\ref{eqn:wedin}). To further compare these two bounds, consider the following examples.
\begin{itemize}
\item \emph{Example 1. } When $U_\rr^Ty_n=0$ and (\ref{eqn:general_condition3}) is satisfied, (\ref{eqn:general_upper}) gives the correct upper bound 0. That is, $\hat U_\rr  \hat U_\rr ^T =   U_\rr U_\rr ^T$. On the contrary, (\ref{eqn:wedin}) gives a non-zero bound $2\sqrt{2}/\rho^{-1}$. To be more concrete, let $Y=\sigma_1 (p^{-1/2}\one_p)((n-1)^{-1/2}\one_{n-1})^T$ be a rank-one matrix and $y_n$ be some vector that is orthogonal to $\one_p$. Then if $\sigma_1>2\norm{y_n}$, we have $\hat u_1=u_1=p^{-1/2}\one_p$ up to sign. (\ref{eqn:general_upper})  gives the correct answer $\normf{\hat u_1 \hat u_1^T -   u_1u_1^T}=0$ as $u_1^Ty_n=0$, while (\ref{eqn:wedin}) leads to a loose upper bound  $2\sqrt{2}\norm{y_n}/\sigma_1$.
\item\emph{Example 2. } Let $Y$ be a matrix with two unique columns such that $y_j$ is equal to either $\theta$ or $-\theta$ for all $j\in[n-1]$ and for some vector $\theta\in\mathr^p$. Then $Y$ is a rank-one matrix with $\sigma_1 = \norm{\theta}\sqrt{n-1}$. Let $y_n=\theta + \epsilon$. As long as $\norm{\theta}\sqrt{n-1}>2\norm{\epsilon}$, we have $\normf{\hat u_1 \hat u_1^T -   u_1u_1^T}\leq 4\sqrt{2}\rho^{-1}(\norm{\theta} + \norm{\epsilon})/\sigma_1$ from (\ref{eqn:general_upper}). If we further assume $\norm{\theta}=1$ and $\epsilon\sim\mathn(0,I_p)$ with $p\ll n$, we have $\normf{\hat u_1 \hat u_1^T -   u_1u_1^T}\lesssim \sqrt{p/n}\rho^{-1}=o(\rho^{-1})$ with high probability. In contrast,     (\ref{eqn:wedin}) only gives $2\sqrt{2}\rho^{-1}$.
\end{itemize}
In the next section, we consider mixture models where the magnitude of  $(\sum_{i\in[\rr]}(\sigma_i^{-1} u_i^T y_n)^2)^{1/2}$ is well-controlled and (\ref{eqn:general_upper}) leads to a much sharper upper bound compared to (\ref{eqn:wedin}). 

Regarding the sharpness of the bound in Theorem \ref{thm:general}, it's worth noting that in Example 1 above, our theorem accurately derives an upper bound of 0, showcasing its optimality in that specific context. To further demonstrate the optimality of our theorem, consider a more intricate example.
\begin{itemize}
\item \emph{Example 3. } Consider a rank-one matrix $Y= \one_p\one_{n-1}^T$ where $\sigma_1= \sqrt{(n-1)p}$ and $u_1 = p^{-1/2}\one_p$. Now, define $y_n = \one_p + sw$, wherein $s$ represents a scalar and $w$ is a unit vector orthogonal to $\one_p$. This means that $y_n$ matches each column of $Y$ for $s=0$ and introduces an orthogonal perturbation for $s\neq 0$.
Given that $\rho = \sigma_1/s = \sqrt{(n-1)p}/s$ and $u_1^Ty_n=\sqrt{p}$, it follows from Theorem \ref{thm:general} that $\normf{\hat u_1 \hat u_1^T -   u_1u_1^T} \leq 4\sqrt{2}s/((n-1)\sqrt{p})$. Since $\hat Y$ is of rank-two, we can express $\hat u_1$ as $\hat u_1 = \sqrt{1-\alpha^2}u_1 + \alpha w$ where $|\alpha|\leq 1$. Note that $\hat u_1^T \hat Y=(\sqrt{(1-\alpha^2)p} \one_{n-1}^T,\sqrt{(1-\alpha^2)p} + \alpha s)$ and $\|{\hat u_1^T \hat Y}\|^2 =(1-\alpha^2)np + \alpha^2 s^2 + 2\sqrt{(1-\alpha^2)p}\alpha s$. For small $s$, we can approximate $\alpha$ (by maximizing $\|{\hat u_1^T \hat Y}\|^2$ over $\alpha$) as $s/(n\sqrt{p})$. Since $\alpha$ is also small, we have $\normf{\hat u_1 \hat u_1^T -   u_1u_1^T} \approx \alpha\sqrt{1-\alpha^2}\normf{u_1w^T+w^Tu_1}=\sqrt{2}\alpha\sqrt{1-\alpha^2} \approx\sqrt{2}s/(n\sqrt{p})$. A comparison with the upper bound deduced from Theorem \ref{thm:general} underscores that the theorem captures the correct rate $s/(n\sqrt{p})$, albeit with a multiplicative constant.
\end{itemize}
However, the sharpness of   Theorem \ref{thm:general} in diverse settings or under different conditions remains an area needing further investigation.

The leave-one-out singular subspace perturbation analysis established in this paper shares conceptual similarities with the leave-one-out technique grounded in random matrix theory and used in the $\ell_\infty$ or $\ell_{2,\infty}$ perturbation analysis \cite{abbe2020entrywise, chen2021spectral}. On a high level, for a matrix \(X\) with an eigenvector \(u\), the goal of  the $\ell_\infty$ analysis is to derive an upper bound for \(\|u\|_\infty =\max_i |u_i|\), where \(\{u_i\}\) represents the coordinates of \(u\). To aid in this task, the leave-one-out technique introduces an auxiliary matrix, formed by excluding the \(i\)th column, \(X_i\), of \(X\), and the corresponding eigenvectors \(u_{-i}\). It approximates \(u_i\) by a quantity involving both \(X_i\) and \(u_{-i}\), leveraging the independence between them.  Our approach aligns with this principle but subsequent analysis distinctly sets it apart. While both methods involve the difference between \(u\) and \(u_{-i}\), the $\ell_\infty$ analysis predominantly uses it as a stepping stone towards \(\|u\|_\infty\),  dealing with it by a direct application of Wedin's theorem. In contrast, our methodology focuses on establishing a sharp bound for this difference. This distinction enables us to characterize the tail probabilities of \(u_i\) rather than just a general $\ell_\infty$ bound and paves the way for a more fine-grained investigation into the performance of spectral methods.

We conclude this section by mentioning that our current analytical framework might extend to scenarios wherein a matrix has multiple columns left out relative to another. Intuitively, as columns can be removed sequentially, Theorem \ref{thm:general} (or its more concise variant, (\ref{eqn:general_upper_simplified})) can be invoked in a successive manner. This iterative application can provide an upper bound on the discrepancy between the two singular subspaces in question. A more intricate way to consider would be a direct extension of the proof of Theorem \ref{thm:general}. Given that this theorem fundamentally revolves around the dynamics between \( U_\rr U_\rr ^Ty_n \) and \( (I-U_\rr U_\rr ^T)y_n \), its generalization is likely to encompass similar, yet more expansive, interactions.

\subsection{Singular Subspace Perturbation in Mixture Models}\label{sec:perturbation_mixture}
The general perturbation theory presented in Theorem \ref{thm:general} is particularly suitable for analyzing singular subspaces of mixture models.

~\\
\emph{Mixture Models. }We consider a mixture model with $k$ centers $\theta_1^*,\theta_2^*,\ldots,\theta_k^*\in\mathr^p$ and a cluster assignment vector $z^*\in[k]^n$.
The observations $X_1,X_2,\ldots, X_n\in\mathr^p$ are generated from	
\begin{align}\label{eqn:mixture_model}
X_i  = \theta^*_{z^*_i} + \epsilon_i,
\end{align}
where $\epsilon_1,\ldots,\epsilon_n\in\mathr^p$ are noises.  The data matrix $X:=(X_1,\ldots, X_n)\in\mathr^{p\times n}$ can be written equivalently in a matrix form
\begin{align}\label{eqn:matrix}
X = P+E,
\end{align}
where $P:=(\theta^*_{z^*_1},\theta^*_{z^*_2},\ldots, \theta^*_{z^*_n})$ is the signal matrix and $E := (\epsilon_1,\ldots,\epsilon_n)$ is the noise matrix. Define $\beta:=\frac{1}{n/k}\min_{a\in[k]} |\{i:z^*_i = a\}| $ such that $\beta n/k$ is the smallest cluster size.

~\\
\indent We are interested in the left singular subspaces of $X$ and its leave-one-out counterparts. For each $i\in[n]$, define $X_{-i}$ to be a submatrix of  $X$ with its $i$th column removed. That is,
\begin{align}\label{eqn:X_minus_i}
X_{-i} := (X_1,\ldots,X_{i-1},X_{i+1},\ldots,X_n)\in\mathr^{p\times (n-1)}.
\end{align}
Let their SVDs be $X = \sum_{j\in[p\wedge n]} \hat\lambda_j\hat u_j\hat v_j^T$  and $X_{-i}=\sum_{j\in[p\wedge(n-1)]} \hat \lambda_{-i,j}\hat u_{-i,j}\hat v_{-i,j}^T$, where  $\hat \lambda_1\geq \hat \lambda_2\geq \ldots \geq \hat \lambda_{p\wedge n}$ and $\hat \lambda_{-i,1}\geq \hat \lambda_{-i,2} \geq \ldots \geq \hat \lambda_{-i,p\wedge(n-1)}$.
Note that the signal matrix $P$ is at most rank-$k$. Then for any $r\in[k]$, define
\begin{align*}
\hat U_{1:r}:=(\hat u_1,\hat u_2,\ldots,\hat u_r)\in\matho^{p\times r}\text{ and }\hat U_{-i,1:r}=(\hat u_{-i,1},\ldots,\hat u_{-i,r})\in\matho^{p\times r}
\end{align*}
to include the  leading $r$ left singular vectors of $X$ and $X_{-i}$, respectively.
We are interested in  controlling the  quantity $ \normf{\hat U_{1:r}\hat U_{1:r}^T - \hat U_{-i,1:r}\hat U_{-i,1:r}^T}$ for each $i\in[n]$.

In Theorem \ref{thm:perturbation}, we provide upper bounds for $\normf{{\hat U_{1:\kr} \hat U^T_{1:\kr} - \hat U_{-i,1:\kr}\hat U_{-i,1:\kr}^T}}$ for all $i\in[n]$ where $\kr\in[k]$ is the rank of the signal matrix $P$. In order to have such a uniform control across all $i\in[n]$, we consider the spectrum of the signal matrix $P$. Let $\lambda_1\geq \lambda_2\geq \ldots\geq \lambda_{p\wedge n}$ be the singular values of $P$ and $\kr$ be the rank of $P$ such that $\kr\in[k]$, $\lambda_{\kr}>0$, and $\lambda_{\kr+1}=0$.

\begin{theorem}\label{thm:perturbation}
Assume   $\beta n/k^2 \geq 10$. 
Assume
\begin{align}\label{eqn:eigengap}
\rho_0 := \frac{\lambda_{\kr}}{\norm{E}} >16.
\end{align}
For any $i\in[n]$, we have
\begin{align}\label{eqn:perturbation_mixture_bounds}
\fnorm{{\hat U_{1:\kr} \hat U^T_{1:\kr} - \hat U_{-i,1:\kr}\hat U_{-i,1:\kr}^T}} \leq  \frac{128}{\rho_0}\br{\sqrt{\frac{k\kr}{\beta n}} +\frac{\norm{\hat U_{-i,1:\kr}\hat U_{-i,1:\kr}^T\epsilon_i}}{\lambda_{\kr}}}.
\end{align}
\end{theorem}
Theorem \ref{thm:perturbation} leverages the mixture model structure (\ref{eqn:mixture_model}) that the signal matrix $P$ has only $k$ unique columns with each appearing at least $\beta n/k$ times. The assumption $\beta n/k^2\geq 10$ helps ensure that spectrum and singular vectors of $P$ do not change significantly if any column of $P$ is removed. We require the condition (\ref{eqn:eigengap})   so that $\hat \lambda_{-i,\kr}-\hat \lambda_{-i,\kr+1}>2\normop{\hat U_{-i,1:\kr}\hat U_{-i,1:\kr}^TX_i}$ holds for each $i\in[n]$, and hence Theorem \ref{thm:general} can be applied uniformly for all $i\in[n]$.  The upper bound (\ref{eqn:perturbation_mixture_bounds}) is a product of $\rho_0^{-1}$ and a sum of two terms.  The second term $\normop{\hat U_{-i,1:\kr}\hat U_{-i,1:\kr}^T\epsilon_i}/\lambda_{\kr}$ can be trivially upper bounded by $\norm{E}/\lambda_{\kr}\leq\rho_0^{-1}$. The first term $\sqrt{k\kr/(\beta n)}=o(1)$ if $\beta n/k^2\gg 1$, for example, when $\beta$ is a constant  and $k\ll \sqrt{n}$. Then (\ref{eqn:perturbation_mixture_bounds}) leads to $\normf{{\hat U_{1:\kr} \hat U^T_{1:\kr} - \hat U_{-i,1:\kr}\hat U_{-i,1:\kr}^T}}\lesssim o(1)\rho_0^{-1} + \rho_0^{-2}$,  superior to the upper bound (\ref{eqn:wedin}) obtained from the direct application of  Wedin's Theorem that is  of  order  $\rho_0^{-1}$.

Theorem \ref{thm:perturbation} studies the perturbation for the leading $\kr$ singular subspaces where $\kr$ is the rank of $P$. In the following Theorem \ref{thm:perturbation_r}, we consider an extension to $\normf{{\hat U_{1:r} \hat U^T_{1:r} - \hat U_{-i,1:r}\hat U_{-i,1:r}^T}}$ where $r$ is not necessarily $\kr$.

\begin{theorem}\label{thm:perturbation_r}
Assume   $\beta n/k^2 \geq 10$.   Assume there exists some $r\in[k]$ such that
\begin{align}\label{eqn:perturbation_r}
\tilde\rho_0 := \frac{\lambda_r -\lambda_{r+1}}{\max\cbr{\norm{E},\sqrt{\frac{k^2}{\beta n}}\lambda_{r+1}}} >16.
\end{align}
For any $i\in[n]$, we have
\begin{align}\label{eqn:perturbation_r_upper_bound}
\fnorm{{\hat U_{1:r} \hat U_{1:r}^T - \hat U_{-i,1:r}\hat U_{-i,1:r}^T}} \leq   \frac{128}{\tilde\rho_0}\br{\frac{\sqrt{kr}}{\sqrt{\beta n}} +\frac{\norm{\hat U_{-i,1:r}\hat U_{-i,1:r}^T\epsilon_i}}{\lambda_{r}}}.
\end{align}
\end{theorem}

In Theorem \ref{thm:perturbation_r}, $r\in[k]$ is any number such that (\ref{eqn:perturbation_r}) is satisfied. When $r$ is chosen to be $\kr$,  (\ref{eqn:perturbation_r}) is reduced to (\ref{eqn:eigengap}), and (\ref{eqn:perturbation_r_upper_bound}) leads to the same upper bound as (\ref{eqn:perturbation_mixture_bounds}).
When $r<\kr$, $\lambda_{r+1}$ is non-zero and in (\ref{eqn:perturbation_r}) it needs to be smaller than the spectral gap $\lambda_r-\lambda_{r+1}$ after some scaling factor. 
To provide some intuition on the condition  (\ref{eqn:perturbation_r}) when $r<\kr$, let the SVD of the signal matrix $P$ be $P=\sum_{j\in[p\wedge n]}\lambda_j u_jv_j^T$ and define $U_{1:r}:=(u_1,u_2,\ldots, u_r)\in\matho^{p\times r}$ and $U_{(r+1):\kr}:=(u_{r+1},u_{r+2},\ldots, u_{\kr})\in\matho^{p\times (\kr-r)}$. Then the data matrix (\ref{eqn:matrix}) can be written equivalently as 
\begin{align}\label{eqn:mixture_r}
X=P'+E', \text{ where }P':=U_{1:r}U_{1:r}^TP\text{ and }E':=E+U_{(r+1):\kr}U_{(r+1):\kr}^TP.
\end{align}
Since it is still a mixture model, Theorem \ref{thm:perturbation} can be applied. Nevertheless, the condition (\ref{eqn:eigengap}) essentially requires $\lambda_{r}/(\norm{E}+\lambda_{r+1})>16$ as $\norm{E'}\leq \norm{E} + \normop{U_{(r+1):\kr}U_{(r+1):\kr}^TP}=\norm{E} + \lambda_{r+1}$, which is  stronger than the condition (\ref{eqn:perturbation_r}). In order to weaken the requirement on the spectral gap  into (\ref{eqn:perturbation_r}), we  study the contribution of $U_{(r+1):\kr}U_{(r+1):\kr}^TP$ towards to the leading $r$ singular subspaces perturbation of $E$. It turns out that its contribution is roughly $\sqrt{k^2/(\beta n)}\lambda_{r+1}$ instead of $\lambda_{r+1}$, due to the fact that $U_{(r+1):\kr}U_{(r+1):\kr}^TP$ has at most $k$ unique columns with each one appearing at least $\beta n/k$ times.

Theorem \ref{thm:perturbation} and Theorem \ref{thm:perturbation_r} require $\beta n/k^2$ be sufficiently large. Further in the paper, results such as Lemma \ref{lem:thresholding} need an even stronger condition wherein $\beta n/k^4$ should be large. We acknowledge that these dependencies on $k$ appear non-optimal. The current formulations stem from challenges faced during our analysis, resulting in these inherent dependencies. We hope to explore more optimal dependency in future research.

\section{Spectral Clustering for Mixture Models}\label{sec:spectral_clustering_mixture_model}

\subsection{Spectral Clustering and Polynomial Error Rate}

Recall the definition of the mixture model in (\ref{eqn:mixture_model}) and also in (\ref{eqn:matrix}). The goal of clustering is to estimate the cluster assignment vector  $z^*$ from the observations $X_1,X_2,\ldots, X_n$.   Since the signal matrix $P$ is of low rank, a natural idea  is to project the observations $\{X_i\}_{i\in[n]}$ onto a low dimensional space before applying classical clustering methods such as variants of $k$-means. This leads to the spectral clustering presented in Algorithm \ref{alg:main}.

\begin{algorithm}[h]
\SetAlgoLined
\KwIn{Data matrix $X=(X_1,\ldots,X_n)\in\mathr^{p\times n}$, number of clusters $k$, number of singular vectors $r$}
\KwOut{Cluster assignment vector $\hat z\in [k]^n$}
 \nl Perform SVD on $X$ to have $$X = \sum_{i=1}^{p \wedge n}   \hat \lambda_i   \hat u_i   \hat v_i^T,$$ where $  \hat \lambda_1 \geq   \hat \lambda_2 \geq \ldots \geq   \hat \lambda_{p\wedge n} \geq 0$ and $\cbr{  \hat u_i}_{i=1}^{p \wedge n}\in\mathr^p, \cbr{  \hat v_i}_{i=1}^{p \wedge n}\in\mathr^n$. Let $ \hat U_{1:r} := \br{ \hat u_1,\ldots,  \hat u_r}\in\mathr^{p\times r}$.

 \nl Perform $k$-means on the columns of $\hat U_{1:r}^T X$. That is, 
 \begin{align}
 \br{\hat z,\cbr{\hat c_j}_{j\in[k]}} = \argmin_{ z\in [k]^n , \cbr{ c_j}_{j\in[k]} \in \mathr^{r}} \sum_{i\in[n]} \norm{\hat U_{1:r}^T X_i -  c_{z_i}}^2.\label{eqn:spectral_low}
 \end{align}
\caption{Spectral Clustering}\label{alg:main}
\end{algorithm}

 In (\ref{eqn:spectral_low}), the dimensionality of each data point $\hat U_{1:r}^T X_i $ is $r$, reduced from  original dimensionality $p$.
 This is computationally appealing as $r$ can be much smaller than $p$.
 The second step of Algorithm \ref{alg:main} is  the $k$-means on the columns of $\hat U_{1:r}^TX$, which is equivalent to performing $k$-means onto the columns of $\hat U_{1:r}\hat U_{1:r}^TX\in\mathr^{p\times n}$. That is, define $\hat \theta_a = \hat U_{1:r} \hat c_a$ for each $a\in[k]$. It can be shown that (see Lemma 4.1 of \citep{loffler2019optimality})
 \begin{align}
 \br{\hat z,\cbr{\hat \theta_j}_{j\in[k]}} = \argmin_{ z\in [k]^n , \cbr{ \theta_j}_{j\in[k]} \in \mathr^{p}} \sum_{i\in[n]} \norm{\hat U_{1:r}\hat U_{1:r}^T X_i - \theta_{z_i}}^2,\label{eqn:spectral}
\end{align}
due to the fact that $\hat U_{1:r}$ has orthonormal columns. As a result,  in the rest of the paper, we  carry out our analysis on $\hat z$ using (\ref{eqn:spectral}).

Before characterizing the theoretical performance of the spectral clustering $\hat z$, we give the definition of the misclustering error which quantifies the distance between an estimator and the ground truth $z^*$.
For any $z\in[k]^n$, its misclustering error is defined as 
\begin{align*}
\ell(z,z^*) := \min_{\phi \in \Phi} \frac{1}{n} \sum_{i\in[n]}\indic{z_i = \phi(z_i^*)},
\end{align*}
where $\Phi := \{\phi: \phi\text{ is a bijection from }[k]\text{ to }[k]\}$. The minimization of $\Phi$ is due to that the cluster assignment vector $z^*$ is identifiable only up to a permutation of the labels $[k]$. In addition to $\beta$ that controls the smallest cluster size, another important quantity in this clustering task is the separation among the  centers. Define $\Delta$ to be the minimum distance among centers, i.e.,
\begin{align*}
\Delta := \min_{a,b\in[k]:a\neq b}\norm{\theta^*_a - \theta^*_b}.
\end{align*}
As we will see later, $\Delta$ determines the difficulty of the clustering task and plays a pivotal role.

In Proposition \ref{prop:poly}, a rough upper bound is provided on the misclustering error $\ell(\hat z,z^*)$ that takes a polynomial expression (\ref{eqn:l_poly}). Notably, Proposition \ref{prop:poly} is deterministic with no assumption on the distribution or the independence of the noises $\{\epsilon_i\}_{i\in[n]}$. In fact, the noise matrix $E$ can be an arbitrary matrix as long as the data matrix has the decomposition (\ref{eqn:matrix}) and the separation condition $(\ref{eqn:Delta_poly})$ is satisfied. In addition, it requires no spectral gap condition. Proposition \ref{prop:poly} is essentially an extension of  Lemma 4.2 in \cite{loffler2019optimality} which is only for the Gaussian mixture model and needs $r=k$.  We include its proof in Appendix \ref{sec:auxiliary} for completeness. Recall $\kr$ denotes the rank of the signal matrix $P$.

\begin{proposition}\label{prop:poly}
Consider the spectral clustering $\hat z$ of Algorithm \ref{alg:main} with $\kr\leq r\leq k$. Assume
\begin{align}\label{eqn:Delta_poly}
\psi_0:=\frac{\Delta}{\beta^{-0.5}kn^{-0.5}\norm{E}} \geq 16.
\end{align}
Then $\ell(\hat z,z^*)\leq \beta /(2k)$. Furthermore, there exists one $\phi\in \Phi$ such that $\hat z$ satisfies
\begin{align}\label{eqn:l_poly}
\ell(\hat z,z^*)= \frac{1}{n}|\{i\in[n]:\hat z_i \neq \phi(z^*_i)\}| \leq \frac{C_0k\norm{E}^2}{n\Delta^2},
\end{align}
and
\begin{align}\label{eqn:hat_theta_difference}
\max_{a\in[k]} \norm{\hat\theta_{\phi(a)} - \theta_a^*}\leq C_0\beta^{-0.5}kn^{-0.5}\norm{E},
\end{align}
where $C_0= 128$.
\end{proposition}

Proposition \ref{prop:poly} provides a starting point for our further theoretical analysis. In the following sections, we are going to provide a sharper analysis for the spectral clustering $\hat z$ that is beyond the polynomial rate stated in (\ref{eqn:l_poly}), with the help of singular subspaces perturbation established in Section \ref{sec:perturbation}.

\subsection{Entrywise Error Decompositions}

In this section, we are going to develop a fine-grained and entrywise analysis on the performance of $\hat z$. Proposition \ref{prop:poly} points out that there exists a permutation $\phi\in\Phi$ such that $n\ell(\hat z,z^*)= |\{i\in[n]:\hat z_i \neq \phi(z^*_i)\}| \leq n\beta/(2k)$. Since the smallest cluster size in $z^*$ is at least $\beta n/k$, such permutation $\phi$ is unique. With $\phi$ identified, $\hat z_i \neq \phi(z_i^*)$ means that the $i$th data point $X_i$ is incorrectly clustered in $\hat z$, for each $i\in[n]$. The following Lemma \ref{lem:decomposition_simple} studies the event $\hat z_i \neq \phi(z_i^*)$ and shows that it is determined by the magnitude of $\normop{\hat U_{1:r}\hat U_{1:r}^T\epsilon_i  }$.

\begin{lemma}\label{lem:decomposition_simple}
Consider the spectral clustering $\hat z$  of Algorithm \ref{alg:main}   with $\kr\leq r\leq k$. Assume  (\ref{eqn:Delta_poly}) holds. Let $\phi\in\Phi$ be the permutation such that $\ell(\hat z,z^*)=\frac{1}{n}|\{i\in[n]:\hat z_i \neq \phi(z^*_i)\}| $. Then there exists a constant $C>0$ such that for any $i\in[n]$,
\begin{align}\label{eqn:decomposition1}
\indic{\hat z_i \neq \phi(z_i^*)}  \leq \indic{\br{1-C\psi_0^{-1}}\Delta\leq 2\norm{\hat U_{1:r}\hat U_{1:r}^T\epsilon_i  } }.
\end{align}
\end{lemma}

To understand Lemma \ref{lem:decomposition_simple}, recall that  in  (\ref{eqn:spectral}) $\hat z$ is obtained by $k$-means on $\{\hat U_{1:r}\hat U_{1:r}^TX_i\}_{i\in[n]}$. Since we have the decomposition $\hat U_{1:r}\hat U_{1:r}^TX_i = \hat U_{1:r}\hat U_{1:r}^T\theta^*_{z^*_i} + \hat U_{1:r}\hat U_{1:r}^T\epsilon_i$ for each $i\in[n]$, the data points $\{\hat U_{1:r}\hat U_{1:r}^TX_i\}_{i\in[n]}$ follow a mixture model with centers $\{\hat U_{1:r}\hat U_{1:r}^T\theta^*_{a}\}_{a\in[k]}$ and noises $\{\hat U_{1:r}\hat U_{1:r}^T\epsilon_i\}_{i\in[n]}$. In the proof of Lemma \ref{lem:decomposition_simple}, we can show these $k$  centers preserve the geometric structure of $\{\theta^*_a\}_{a\in[k]}$ with minimum distance around $\Delta$. Intuitively, if $\normop{\hat U_{1:r}\hat U_{1:r}^T\epsilon_i  } $ is smaller than half of the minimum distance,  $\hat U_{1:r}\hat U_{1:r}^TX_i$ is closer to $ \hat U_{1:r}\hat U_{1:r}^T\theta^*_{z^*_i} $ than any other centers, and thus $z_i^*$ can be correctly recovered.

While Lemma \ref{lem:decomposition_simple} lays foundational understanding, it alone is not sufficient for deriving explicit expressions
 for the performance of spectral clustering when the noises $\{\epsilon_i\}_{i\in[n]}$ are assumed to be random.  The entrywise upper bound (\ref{eqn:decomposition1}) shows that  the event $\hat z_i \neq \phi(z_i^*)$ is determined by the $\normop{\hat U_{1:r}\hat U_{1:r}^T\epsilon_i  }$, but the fact that $\hat U_{1:r}\hat U_{1:r}^T$ depends on $\epsilon_i$ makes any follow-up probability calculations challenging.  The key to make use of Lemma \ref{lem:decomposition_simple} is our leave-one-out singular subspace perturbation theory, particularly, Theorem \ref{thm:perturbation}. To decouple the dependence between $\hat U_{1:r}\hat U_{1:r}^T$ and $\epsilon_i$, we replace the former quantity by its leave-one-out counterpart $\hat U_{-i,1:r}\hat U_{-i,1:r}^T$. Take $r$ to be $\kr$.  Note that
 \begin{align}\label{eqn:decomposition_100}
 \norm{\hat U_{1:\kr}\hat U_{1:\kr}^T\epsilon_i  }\leq \norm{\hat U_{-i,1:\kr}\hat U_{-i,1:\kr}^T\epsilon_i  } + \normf{\hat U_{1:\kr}\hat U_{1:\kr}^T - \hat U_{-i,1:\kr}\hat U_{-i,1:\kr}^T}\norm{\epsilon_i  }.
 \end{align}
The perturbation $\normf{\hat U_{1:\kr}\hat U_{1:\kr}^T - \hat U_{-i,1:\kr}\hat U_{-i,1:\kr}^T}$ is well-controlled by Theorem \ref{thm:perturbation}, which shows the second term on the RHS of the above display is  essentially $O(\rho_0^{-2}) \normop{\hat U_{-i,1:\kr}\hat U_{-i,1:\kr}^T\epsilon_i  }$. This leads to the following Lemma \ref{lem:decomposition} on the entrywise clustering errors.

\begin{lemma}\label{lem:decomposition}
Consider the spectral clustering $\hat z$  of Algorithm \ref{alg:main}   with $r=\kr$. Assume $\beta n/k^2 \geq10$, (\ref{eqn:eigengap}), and (\ref{eqn:Delta_poly}) hold.  Let $\phi\in\Phi$ be the permutation such that $\ell(\hat z,z^*)=\frac{1}{n}|\{i\in[n]:\hat z_i \neq \phi(z^*_i)\}| $. Then there exists a constant $C$ such that  for any $i\in[n]$,
\begin{align*}
\indic{\hat z_i \neq \phi(z_i^*)}\leq  \indic{\br{1-C\br{\psi_0^{-1}+\rho_0^{-2}}}\Delta \leq 2\norm{\hat U_{-i,1:\kr}\hat U_{-i,1:\kr}^T\epsilon_i}}.
\end{align*}
Consequently, if the noises $\{\epsilon_i\}_{i\in[n]}$ are random, the risk of $\hat z$ satisfies
\begin{align*}
\E\ell(\hat z,z^*)\leq n^{-1}\sum_{i\in[n]} \E\indic{\br{1-C\br{\psi_0^{-1}+\rho_0^{-2}}}\Delta \leq 2\norm{\hat U_{-i,1:r}\hat U_{-i,1:r}^T\epsilon_i}}.
\end{align*}
\end{lemma}

Lemma \ref{lem:decomposition} needs three conditions. The first one $\beta n/k^2 \geq10$ is on the smallest cluster sizes and can be easily satisfied if both $\beta,k$ are constants. The second condition (\ref{eqn:eigengap}) is a spectral gap condition on  the smallest non-zero singular value $\lambda_{\kr}$. The third one is for the separation of the centers $\Delta$.
With all the three conditions satisfied, Lemma \ref{lem:decomposition} shows that the entrywise clustering error for $X_i$ boils down to $\normop{\hat U_{-i,1:\kr}\hat U_{-i,1:\kr}^T\epsilon_i}$. 
When the noises $\{\epsilon_j\}_{j\in[n]}$ are assumed to be random and independent of each other, the projection matrix $\hat U_{-i,1:\kr}\hat U_{-i,1:\kr}^T$ is independent of $\epsilon_i$ for each $i\in[n]$, a desired property crucial to our follow-up investigation on the risk $\E\ell(\hat z,z^*)$. When $\{X_i\}_{i\in[n]}$ are generated randomly, as discussed in subsequent sections, Lemma \ref{lem:decomposition}  leads to explicit expressions for the performance of the spectral clustering.

The key towards establishing Lemma \ref{lem:decomposition}  is 
Theorem \ref{thm:perturbation}.  Without Theorem \ref{thm:perturbation}, if the classical perturbation theory such as Wedin's theorem is used instead, then in order to obtain similar upper bounds in Lemma \ref{lem:decomposition},  the second term on the RHS of (\ref{eqn:decomposition_100}) needs to be much smaller than $\Delta$. This essentially requires $\max_{i\in[n]} \normop{\epsilon_i}^2\lesssim \lambda_\kr\Delta$, in addition to (\ref{eqn:eigengap}) and (\ref{eqn:Delta_poly}). As we will show in the next section, for sub-Gaussian noises, this additional condition requires $p \log n\lesssim \sqrt{n}$ in regimes where Lemma \ref{lem:decomposition} only needs $p\lesssim n$.

\subsection{Sub-Gaussian Mixture Models}\label{sec:sub_gaussian}

In this section, we investigate the performance of the spectral clustering $\hat z$ for mixture models with sub-Gaussian noises. Theorem \ref{thm:subg} assumes that each noise $\epsilon_i$ is an independent sub-Gaussian random vector with zero mean and variance proxy $\sigma^2$  and establishes an exponential rate for the risk  $\E \ell(\hat z,z^*)$.

\begin{theorem}\label{thm:subg}
Consider the spectral clustering $\hat z$  of Algorithm \ref{alg:main}   with $r=\kr$. Assume $\epsilon_i\sim\text{SG}_p(\sigma^2)$ independently  with zero mean for each $i\in[n]$.  Assume   $\beta n/k^2 \geq10$. There exist constants $C,C'>0$ such that under the assumption that
\begin{align}\label{eqn:Delta_subg}
\psi_1 := \frac{\Delta}{\beta^{-0.5}k\br{1+ \sqrt{\frac{p}{n}}}\sigma} >C
\end{align}
and
\begin{align}\label{eqn:gap_subg}
\rho_1 :=\frac{\lambda_\kr}{\br{\sqrt{n} + \sqrt{p}}\sigma} >C,
\end{align}
we have
\begin{align*}
\E \ell(\hat z,z^*)\leq \ebr{-\br{1-C'\br{\psi_1^{-1}+\rho_1^{-2}}}\frac{\Delta^2}{8\sigma^2}}+ \ebr{-\frac{n}{2}}.
\end{align*}
\end{theorem}

Under this sub-Gaussian setting,  standard concentration theory shows that the noise matrix $E$ has its operator norm $\norm{E}\lesssim \sigma(\sqrt{n}+\sqrt{p})$ with high probability (see Lemma \ref{lem:sub_gaussian_operator}). Under this event,  (\ref{eqn:Delta_subg}) and (\ref{eqn:gap_subg}) are sufficient conditions for (\ref{eqn:eigengap}) and (\ref{eqn:Delta_poly}), respectively. The risk in Theorem \ref{thm:GMM} has two terms, where the first term takes an exponential form of $\Delta^2/(8\sigma^2)$ and  the second term $\exp(-n/2)$ comes from the aforementioned event of $\norm{E}$. The first term is the dominating one, as long as $\Delta^2/\sigma^2$, which can be interpreted as the signal-to-noise ratio, is smaller than $n/2$. In fact,  $\Delta^2/\sigma^2\lesssim \log n$ is the most interesting regime as otherwise $\hat z$ already achieves the exact recovery (i.e., $\hat z=z^*$) with high probability, since $\E \{\ell(\hat z,z^*)=0\}=o(1)$.

Theorem \ref{thm:subg} makes a substantial improvement over Proposition \ref{prop:poly}. Using the aforementioned high-probability event on $\norm{E}$, (\ref{eqn:l_poly}) only leads to $\E \ell(\hat z,z^*)\lesssim (1+\sqrt{p/n})^2\sigma^2/\Delta^2 + \ebr{-n/2}$ which takes a polynomial form of the $\Delta^2/\sigma^2$. On the contrary, Theorem \ref{thm:subg} provides a much sharper exponential rate. 

Our leave-one-out singular subspace perturbation theory and  its consequence Lemma \ref{lem:decomposition} provide the key toolkit towards Theorem \ref{thm:subg}. Since $\hat U_{-i,1:\kr}^T$ is independent of $\epsilon_i$, we have $\hat U_{-i,1:\kr}^T\epsilon_i \sim \text{SG}_\kr(\sigma^2)$ being another sub-Gaussian random vector. This makes it possible to control the tail probabilities of $\normop{\hat U_{-i,1:\kr}\hat U_{-i,1:\kr}^T\epsilon_i }^2 = \normop{\hat U_{-i,1:\kr}^T\epsilon_i }^2 $ which is a  quadratic form of sub-Gaussian random vectors. Without using our perturbation theory, if  the classical  perturbation bounds such as Wedin’s Theorem is used instead, the previous section shows that $\max_{i\in[n]} \normop{\epsilon_i}^2\lesssim \lambda_\kr\Delta$ is additionally needed to obtain results similar to Lemma \ref{lem:decomposition}. This equivalently requires $\lambda_\kr\Delta/(\sigma^2 p\log n)\gtrsim 1$. When $\Delta/\sigma,k,\beta$ are constants, this additional   condition essentially requires $p\log n\lesssim \sqrt{n}$. In contrast, Theorem \ref{thm:subg} only needs $p\lesssim n$.

Theorem \ref{thm:subg} gives a finite-sample result for the performance of spectral clustering in sub-Gaussian mixture models. In the following Corollary \ref{cor:subg}, by slightly strengthening conditions (\ref{eqn:Delta_subg}) and (\ref{eqn:gap_subg}), we immediately obtain an asymptotic error bound with the exponent being $(1-o(1))\Delta^2/(8\sigma^2)$.

\begin{corollary}\label{cor:subg}
Under the same setting as in Theorem \ref{thm:subg}, if $\psi_1,\rho_1\rightarrow\infty$ is further assumed,
we have
\begin{align*}
\E \ell(\hat z,z^*)\leq \ebr{-\br{1-o(1)}\frac{\Delta^2}{8\sigma^2}}+ \ebr{-\frac{n}{2}}.
\end{align*}
If $\Delta /\sigma\geq (1+c)2\sqrt{2\log n}$ is further assumed where $c>0$ is any constant, $\hat z$ achieves the exact recovery, i.e.,  $\E \indic{\ell(\hat z,z^*) \neq 0}=o(1)$.
\end{corollary}

In the exponents of Theorem \ref{thm:subg} and Corollary \ref{cor:subg}, we are able to obtain an explicit constant $1/8$. In addition, we obtain an explicit constant $2\sqrt{2}$ for the exact recovery in Corollary \ref{cor:subg}. These constants are sharp when the noises are further assumed to be isotropic Gaussian, as we will show in Section \ref{sec:isotropic}. 

The recent related paper by \citep{abbe2020ell_p} develops a $\ell_p$ perturbation theory and applies it to the spectral clustering for sub-Gaussian mixture models. It  obtains exponential error rates but with unspecified constants in the exponents and  under  special assumptions on the spectrum and geometric distribution of the centers. It first assumes both $\beta$ and $k$ are constants. Let $G\in\mathr^{k\times k}$ be the Gram matrix of the centers such that $G_{i,j}=\theta^{*T}_i\theta^*_j$ for each $i,j\in[k]$. It further requires $\bar \lambda I \prec G \prec c\bar \lambda I$ for some constant $c>1$, i.e., all $k$ eigenvalues of $G$ are of the same order. It implies that the maximum  and minimum distances among centers are comparable. This rules out many interesting cases such as all the centers are on one single line. In addition, \citep{abbe2020ell_p} needs $\bar \lambda /\sigma \rightarrow\infty$. Equivalently it means that the leading $k$ singular values $\lambda_1,\lambda_2,\ldots,\lambda_k$ of the signal matrix $P$  not only are all of the same order, but also $\lambda_k/(\sqrt{n}\sigma)\gg \max\{1,\sqrt{p/n}\}$. As a comparison, we allow collinearity of the centers such that the rank of $G$ (and $P$) can be smaller than $k$. We allow the singular values $\lambda_1,\lambda_2,\ldots,\lambda_\kr$ not of the same order as long as the smallest one satisfies (\ref{eqn:gap_subg}), which can be equivalently written as $\lambda_{\kr}/(\sqrt{n}\sigma) \gtrsim \max\{1,\sqrt{p/n}\}$. The distances among the centers are also not necessarily of the same order as long as the smallest distance satisfies (\ref{eqn:Delta_subg}). Hence, our conditions are more general than those in \citep{abbe2020ell_p}.

The spectral gap condition (\ref{eqn:gap_subg}) ensures that singular vectors corresponding to small non-zero singular values are well-behaved. It is not needed in Section \ref{sec:adaptive} where we propose a variant of spectral clustering with adaptive dimension reduction. It can also be dropped in Section \ref{sec:isotropic} when the noise is isotropic Gaussian. When the mixture model is symmetric with two components (for example, the model considered in Section \ref{sec:lower}), the signal matrix $P$ is rank-one. Hence, (\ref{eqn:gap_subg}) is also no longer needed as it can be directly implied from (\ref{eqn:Delta_subg}).

\subsection{Spectral Clustering with Adaptive Dimension Reduction}\label{sec:adaptive}

The theoretical analysis for the spectral clustering $\hat z$ of Algorithm \ref{alg:main} that is presented in Lemma \ref{lem:decomposition} and Theorem \ref{thm:subg} requires the use of all the $\kr$ singular vectors where $\kr$ is the rank of the signal matrix $P$. Nevertheless, not all singular components are equally useful towards the clustering task and the importance of an individual singular vector can be characterized by its corresponding singular value. This motivates us to propose the following algorithm where the number of singular vectors used is carefully picked.

\begin{algorithm}[h]
\SetAlgoLined
\KwIn{Data matrix $X=(X_1,\ldots,X_n)\in\mathr^{p\times n}$, number of clusters $k$, threshold $T$}
\KwOut{Clustering label vector $\tilde z\in [k]^n$}
 \nl Perform SVD on $X$ same as Step 1 of Algorithm \ref{alg:main}.\\
\nl Let $\hat r$ be the largest index in $[k]$ such that the difference between two neighboring  singular values is greater than $T$, i.e.,
\begin{align}\label{eqn:hat_r_def}
\hat r=\max\{a\in[k]:\hat\lambda_a -\hat\lambda_{a+1}\geq T\}.
\end{align} 
 Let $  \hat U_{1:\hat r} := \br{ \hat u_1,\ldots,  \hat u_{\hat r}}\in\mathr^{p\times \hat r}$.

 \nl Perform $k$-means on the columns of $\hat U_{1:\hat r}^T X$. That is, 
 \begin{align}
 \br{\tilde z,\cbr{\tilde c_j}_{j=1}^k} = \argmin_{ z\in [k]^n , \cbr{ c_j}_{j=1}^{k} \in \mathr^{\hat r}} \sum_{i\in[n]} \norm{\hat U_{1:\hat r}^T X_i -  c_{z_i}}^2.\label{eqn:thresholding_kmeans}
 \end{align}
\caption{Spectral Clustering with Adaptive Dimension Reduction}\label{alg:thresholding}
\end{algorithm}

Algorithm \ref{alg:thresholding} is a  variant of Algorithm \ref{alg:main} with the number of singular vectors selected by (\ref{eqn:hat_r_def}), where $\hat r$ is the largest integer such that the empirical spectral gap $\hat \lambda_{\hat r}-\hat \lambda_{\hat r+1}$ is greater or equal to some threshold $T$. 
The criterion in (\ref{eqn:hat_r_def}) for choosing \(\hat r\) has two purposes. Firstly, it ensures the presence of a desirable spectral gap. More crucially, it is intended to encompass important singular vectors while disregarding those that are noisy or of lesser relevance. This is illuminated by an implication from (\ref{eqn:hat_r_def}) that \(\hat \lambda_{\hat r+1}\leq \hat \lambda_{k+1} + kT\) and that the significance of a singular vector can be characterized by the magnitude of its associated singular value.
To illustrate this further, let us compare our approach with an alternative selection mechanism that simply choose an arbitrary index from \(\{a\in[k]:\hat\lambda_a -\hat\lambda_{a+1}\geq T\}\) instead of the largest one. While such a criterion would indeed ensure a spectral gap, it is  possible that \(\hat \lambda_{\hat r+1}\) and subsequent singular values remain large, suggesting that the corresponding singular vectors are of importance. Omission of these pivotal vectors from the clustering algorithm would result in a decline in its performance.

The choice of the threshold $T$ is crucial. When $T$ is small, $\hat r$ might be even bigger than the rank $\kr$. When $T\gtrsim \norm{E}$, it guarantees that the singular values of the signal matrix $P$ satisfy $\lambda_{\hat r}-\lambda_{\hat r+1}\gtrsim T$ and $\lambda_{\hat r+1}\lesssim T$. When $T$ is too large, the singular subspace $\hat U_{1:\hat r}$ misses singular vectors such as $\hat u_{\hat r+1}$ whose importance scales with $\lambda_{\hat r+1}$ that can not be ignored. This in turn deteriorates the clustering performance of $\tilde z$.
A rule of thumb for the threshold $T$ is that $T/\norm{E}$ is at least of constant order. It is allowed to grow but not faster than $\tilde \phi_0$ defined in (\ref{eqn:thresholding_tilde_psi}). The precise description of the choices of $T$ needed is given below in Lemma \ref{lem:thresholding}, which provides an entrywise analysis of $\tilde z$ that is analogous to Lemma \ref{lem:decomposition}.

\begin{lemma}\label{lem:thresholding}
Consider the estimator $\tilde z$ from Algorithm \ref{alg:thresholding}. Assume $\beta n/k^4\geq 400$. Let $\phi\in\Phi$ be the permutation such that $\ell(\hat z,z^*)=\frac{1}{n}|\{i\in[n]:\hat z_i \neq \phi(z^*_i)\}| $. Define 
\begin{align}\label{eqn:thresholding_tilde_psi}
\tilde \psi_0 := \frac{\Delta}{\beta^{-0.5}k^{2}n^{-0.5}\norm{E}}
\end{align}
and $\tilde \rho := T/\norm{E}$. Assume $256<\tilde \rho <\tilde\psi_0/64 $. There exist constants $C,C'$ such that if $\tilde\psi_0 >C$, then
\begin{align*}
\indic{\hat z_i \neq\phi(z^*_i)}\leq \indic{\br{1-C'\br{\tilde \rho \tilde \psi_0^{-1} + \tilde \rho^{-1}}}\Delta\leq 2 \norm{\hat U_{-i,1:r}\hat U_{-i,1:r}^T\epsilon_i}}.
\end{align*}
Consequently, we have
\begin{align*}
\E \ell(\hat z,z^*)\leq n^{-1} \sum_{i\in[n]} \E \indic{\br{1-C'\br{\tilde \rho \tilde \psi_0^{-1} + \tilde \rho^{-1}}}\Delta\leq 2 \norm{\hat U_{-i,1:r}\hat U_{-i,1:r}^T\epsilon_i}}.
\end{align*}
\end{lemma}

With a proper choice of the threshold $T$, Lemma \ref{lem:thresholding} only poses requirements on the smallest cluster size $\beta n/k$ and the  minimum separation among the centers $\Delta$. Compared to Lemma \ref{lem:decomposition} and Theorem \ref{thm:subg}, it removes any condition on the smallest non-zero singular value such as (\ref{eqn:eigengap}) or (\ref{eqn:gap_subg}). In addition, it requires no knowledge on the rank $\kr$. Note that under the conditions of Lemma \ref{lem:thresholding}, $\hat r$ defined in (\ref{eqn:hat_r_def}) always exists  (See Lemma \ref{lem:hat_r_exist}).

With Lemma \ref{lem:thresholding}, we have the following exponential error bound on the performance of $\tilde z$ on sub-Gaussian mixture models, analogous to Theorem \ref{thm:subg} and Corollary \ref{cor:subg} for $\hat z.$
\begin{theorem}\label{thm:GMM_2}
Consider the estimator $\tilde z$ from Algorithm \ref{alg:thresholding}.  Assume $\epsilon_i\sim\text{SG}_p(\sigma^2)$ independently with zero mean for each $i\in[n]$.  Assume   $\beta n/k^4\geq 400$. There exist constants $C,C',C_1,C_2>0$ such that under the assumption that
\begin{align*}
\psi_2 := \frac{\Delta}{\beta^{-0.5}k^2\br{1+ \sqrt{\frac{p}{n}}}\sigma} >C
\end{align*}
and $\rho_2:=T/(\sigma(\sqrt{n}+\sqrt{p}))$ satisfies $C_1\leq \rho_2\leq \psi_2/C_2$, we have
\begin{align*}
\E \ell(\tilde z,z^*)\leq \ebr{-\br{1-C'\br{\rho_2\psi_2^{-1}+\rho_2^{-1}}}\frac{\Delta^2}{8\sigma^2}}+ \ebr{-\frac{n}{2}}.
\end{align*}
If $\psi_2,\rho_2\rightarrow\infty$ and $\rho_2/\psi_2=o(1)$ are further assumed, we have
\begin{align*}
\E \ell(\tilde z,z^*)\leq \ebr{-\br{1-o(1)}\frac{\Delta^2}{8\sigma^2}}+ \ebr{-\frac{n}{2}}.
\end{align*}
\end{theorem}

\subsection{Isotropic Gaussian Mixture Models}\label{sec:isotropic}
In this section, we consider the isotropic Gaussian mixture models where the noises are sampled from $\mathn(0,\sigma^2I_p)$ independently. As a special case of the sub-Gaussian mixture models, Theorem \ref{thm:subg} can be directly applied. Nevertheless, the isotropic Gaussian noises make it possible to remove the spectral gap condition (\ref{eqn:gap_subg}). In addition, we study the performance of the spectral clustering $\hat z$ from Algorithm \ref{alg:main} with exactly the leading $k$ singular vectors, regardless of $\kr$, the rank of matrix $P$. As a result, it requires no knowledge on $\kr$ and needs no adaptive dimension reduction such as Algorithm \ref{alg:thresholding}. We have the following theorem on its performance.

\begin{theorem}\label{thm:GMM}
Consider the spectral clustering $\hat z$  of Algorithm \ref{alg:main}   with $r=k$. Assume $\epsilon_i\iid\mathn(0,\sigma^2I_p)$ for each $i\in[n]$. Assume $\beta n/k^4\geq 100$ and
\begin{align}\label{eqn:Delta_GMM}
\frac{\Delta}{k^{3.5}\beta^{-0.5}\br{1+{\frac{p}{n}}}\sigma} \rightarrow\infty.
\end{align}
We have
\begin{align}\label{eqn:GMM_loss}
\E \ell(\hat z,z^*)\leq  \ebr{-\br{1-C\br{\frac{\Delta}{k^{3.5}\beta^{-0.5}\br{1+{\frac{p}{n}}}\sigma}}^{-0.25}}\frac{\Delta^2}{8\sigma^2}} + 2e^{-0.08n},
\end{align}
where $C>0$ is some constant.
\end{theorem}

Theorem \ref{thm:GMM} shows that asymptotically $\E \ell(\hat z,z^*)\leq \exp(-(1-o(1))\Delta^2/(8\sigma^2))+  2\ebr{-0.08n}$
where the first term dominates when $\Delta^2/\sigma^2=o(n)$. 
The minmax lower bound for recovering $z^*$ under the given model is established in \cite{lu2016statistical}: $\inf_{\hat z}\sup_{(\theta_1^*,\ldots,\theta_k^*),z^*}\E \ell(\hat z,z^*) \geq \exp(-(1+o(1))\Delta^2/(8\sigma^2))$ as long as $\Delta^2/\sigma^2\gg \log(k\beta^{-1})$. This immediately implies that the considered estimator is minimax optimal. Theorem \ref{thm:GMM} also implies $\hat z$ achieves the exact recovery $\E\{\ell(\hat z,z^*)\neq 0\}=o(1)$ when  $\Delta /\sigma\geq (1+c)2\sqrt{2\log n}$ for any small constant $c>0$. When $\Delta /\sigma\leq  (1-c)2\sqrt{2\log n}$, no algorithm is able to recover $z^*$ exactly with high probability according to the minimax lower bound.

It is worth mentioning that Theorem \ref{thm:GMM}   requires no spectral gap condition such as (\ref{eqn:eigengap}) or (\ref{eqn:gap_subg}). The purpose of such conditions is to ensure that  singular vectors of $X$  are well controlled, especially those corresponding to small non-zero singular values of the signal matrix $P$. When the noises are isotropic Gaussian,  the distribution of each right singular vector $\hat v_j$ is well-behaved for any $j\in[p\wedge n]$. Lemma 4.4 of \cite{loffler2019optimality} shows that each $(I-V_{1:\kr}V_{1:\kr}^T)\hat v_j$ is Haar distributed on the sphere spanned by $(I-V_{1:\kr}V_{1:\kr}^T)$, where $V_{1:\kr}:=(v_1,v_2,\ldots,v_\kr)\in\matho^{n\times \kr}$ is the right singular subspace of the signal matrix $P$. Theorem \ref{thm:GMM} is about the singular subspace $\hat U_{1:k}$. In its proof, we decompose it into $\hat U_{1:r}$ and $\hat U_{(r+1):k}$, for some index $r\in[\kr]$ with sufficiently large spectral gap $\lambda_{r}-\lambda_{r+1}$ so that the contribution of $\hat U_{1:r}$ can be precisely quantified following similar arguments used to establish Lemma \ref{lem:thresholding} and Theorem \ref{thm:subg}. The contribution of each $\hat u_j$ where $j\in \{r+1,\ldots,k\}$ is eventually connected with properties of the corresponding right singular vector $\hat v_j$, particularly, the distribution of $(I-V_{1:\kr}V_{1:\kr}^T)\hat v_j$. These two sources of errors together lead to the upper bound (\ref{eqn:GMM_loss}).

The performance of  Algorithm \ref{alg:main} with $r=k$ under the same isotropic Gaussian mixture model is the main topic of \cite{loffler2019optimality} which derives a similar upper bound for $\E\ell(\hat z,z^*)$ assuming $\Delta/(\beta^{-0.5}k^{10.5}(1+p/n))\rightarrow\infty$. The key technical tool used in \cite{loffler2019optimality} is spectral operator perturbation theory of \cite{koltchinskii2016asymptotics, koltchinskii2016perturbation} on the difference between empirical singular subspaces and population ones, which  works for the Gaussian noise case and  it is not clear whether it is possible to be extended to other distributions including sub-Gaussian distributions. In this paper, 
the proof of Theorem \ref{thm:GMM} is completely different, using Theorem \ref{thm:perturbation_r} on the  difference between empirical singular subspaces and their leave-one-out counterparts. We not only recover the main result of  \cite{loffler2019optimality} with a much shorter proof, but also improve the dependence of $k$. Despite that Theorem \ref{thm:GMM} needs an extra condition  $\beta n/k^4\geq 100$, it only requires $k^{3.5}$ to satisfy (\ref{eqn:Delta_GMM}), while \cite{loffler2019optimality} needs $k^{10.5}$ instead which is a stronger condition.

\subsection{Lower Bounds and Sub-optimality of Spectral Clustering}\label{sec:lower}

In the above sections, we focus on quantifying the performance of  spectral clustering under mixture models. An interesting question is whether the spectral clustering is optimal. When the noise is the isotropic Gaussian, Theorem \ref{thm:GMM} matches with the minimax  rate assuming (\ref{eqn:Delta_GMM}) holds, showing that the spectral clustering is indeed optimal in this case. It remains unclear whether the spectral clustering is optimal or not when the noise is beyond the isotropic Gaussian model.

To answer this question, in this section we consider a  two-cluster symmetric mixture model where the centers are proportional to $\one_p$ and the noises have i.i.d. entries. This setup makes it possible to apply the central limit theorem to characterize the performance of the spectral clustering with sharp upper and lower bounds, as $\one_p^T\epsilon_i$ is asymptotically normal for each $i\in[n]$ when $p$ is large.

~\\
\emph{A Two-cluster Symmetric Mixture Model.} Consider a mixture model (\ref{eqn:mixture_model})  with two clusters such that

\begin{align}\label{eqn:simple}
\theta^*_1= -\theta^*_2 = \delta \one_p, \text{ and }\{\epsilon_{i,j}\}_{i\in[n],j\in[p]}\iid F,
\end{align}
for some $\delta \in\mathr$
and some distribution $F$, where $\{\epsilon_{i,j}\}_{j\in[p]}$ are entries of $\epsilon_i$ for each $i\in[n]$.

~\\
\indent 
Under the above model (\ref{eqn:simple}), we have  $k=2$, $\Delta =2\sqrt{p}\delta$ and the largest singular value $\lambda_1 = \delta \sqrt{np}$. Since the signal matrix $P$ is rank-one (i.e., $\kr=1$) with $u_1= (1/\sqrt{p}) \one_p$, a natural idea is to cluster using the first singular vector only. Define
\begin{align}\label{eqn:check_z}
 \br{\check z,\cbr{\check c_j}_{j=1}^2} = \argmin_{ z\in [2]^n , \cbr{ c_j}_{j=1}^{2} \in \mathr} \sum_{i\in[n]} \br{\hat u_1^T X_i -  c_{z_i}}^2.
\end{align}
The performance of the spectral estimator $\check z$ will be the focus in this section. Note that $\hat u_1^TX =\hat\lambda_1 \hat v_1^T$ where $\hat v_1$ is the leading right singular vector of $X$, so $\check z$ equivalently  performs clustering on  $\{\hat v_{1,i}\}_{i\in[n]}$, the entries of $\hat v_1$.
 This is closely related to the sign estimator $\{\text{sign}(\hat v_{1,i})\}_{i\in[n]}$, which estimates the cluster assignment by the signs of $\{\hat v_{1,i}\}_{i\in[n]}$.

Since $\check z$ is exactly the spectral clustering $\hat z$ of Algorithm \ref{alg:main} with $r=1$, Theorem \ref{thm:subg} can be directly applied when noises are sub-Gaussian and yields the following result. 
Under the model (\ref{eqn:simple}), assume that $F$ is a $\text{SG}(\sigma^2)$ distribution with zero mean and $\beta n>40$. There exist constants $C,C'>0$ such that under the assumption that
\begin{align*}
\psi_3:=\frac{\Delta}{\beta^{-0.5}\br{1+\sqrt{\frac{p}{n}}}\sigma} >C,
\end{align*}
we have $\E \ell(\check z,z^*)\leq \exp(-(1-C'\psi_3^{-1}) \Delta^2/(8\sigma^2)) + \exp(-n/2).$

The special structure of (\ref{eqn:simple}) makes it possible to derive a sharper upper bound than the  one above and a matching lower bound on the performance of $\check z$ with some additional assumption on the distribution $F$. Instead of directly using Lemma \ref{lem:decomposition} (which leads to Theorem \ref{thm:subg} and then the above upper bound),
we can further connect the clustering error with $u_1^T \epsilon_i$
where $u_1^T \epsilon_i = p^{-1/2} \sum_{j=1}^p \epsilon_{i,j}$ is approximately normally distributed when $p$ is large.  On the other hand, the structure of (\ref{eqn:simple}) enables us to have a lower bound for $\indic{\hat z_i \neq \phi(z^*_i)}$ that is in an opposite direction of Lemma \ref{lem:decomposition}. See Lemma \ref{lem:lower_bound_enrywise}  for details. The key technical tool used is Theorem \ref{thm:perturbation} on the perturbation $|\hat u_1\hat u_1^T-\hat u_{-i,1}\hat u_{-i,1}^T|$ for all $i\in[n]$.
These together  give a sharp and matching  lower bound for $\E \ell(\check z,z^*)$ where the clustering error is essentially determined by  $\Delta$ and the variance $\bar\sigma^2$.

\begin{theorem}\label{thm:lower_bound_spectral}
Consider the model (\ref{eqn:simple}). For any $\xi\sim F$, assume $\E \xi=0, \text{Var}(\xi)=\bar \sigma^2$, 
and $\xi\sim \text{SG}(\sigma^2)$ where $\sigma\leq  C\bar\sigma$ for some constant $C>0$. Assume $\beta n>40$. Then there exist constants $C',C'',C'''>0$ such that if $\psi_3\geq C'$,
we have
\begin{align*}
& \E \ell(\check z,z^*)\leq \ebr{-\frac{\br{1-C''\psi_3^{-1}}^2\Delta^2}{8\bar \sigma^2}}+  \ebr{-C''\sqrt{p}}+ \ebr{-\frac{n}{2}},\\
\text{and }\quad& \E \ell(\check z,z^*) \geq\ebr{-\frac{\br{1+C'''\psi_3^{-1}}^2\Delta^2}{8\bar \sigma^2}}-  \ebr{-C'''\sqrt{p}}- \ebr{-\frac{n}{2}}.
\end{align*}
\end{theorem}

In Theorem \ref{thm:lower_bound_spectral}, the term $\exp(-C''\sqrt{p})$ is due to the normal approximation of $u_1^T\epsilon_i$ and decays when the dimensionality $p$ increases. The term $\exp(-n/2)$ is due to a high-probability event on $\norm{E}$. If 
additionally $\Delta/\bar \sigma \ll \max\{p^{1/4}, n^{1/2}\}$ is assumed, Theorem \ref{thm:lower_bound_spectral} concludes asymptotically
\begin{align}\label{eqn:spectral_sharp}
\E\ell(\check z,z^*) = \ebr{-\frac{(1+c)\Delta^2}{8\bar \sigma^2}},
\end{align}
for some small constant $c$.

The upper and lower bounds in Theorem \ref{thm:lower_bound_spectral} give a sharp characterization of the performance of $\check z$. To answer the question of whether it is optimal or not, we need to establish the minimax rate for the clustering task under the model (\ref{eqn:simple}). Since the model (\ref{eqn:simple}) is essentially about a testing between two parametric distributions, the optimal procedure is the likelihood ratio test. According to the classical asymptotics theory \cite{van2000asymptotic}, the likelihood ratio behaves like a normal random variable as $p\rightarrow\infty$ under some regularity conditions. This leads to an error rate determined by $\Delta$ and the Fisher information.

\begin{lemma}\label{lem:minimax_spectral}
Consider the model (\ref{eqn:simple}). Assume the distribution $F$ has a positive, continuously differentiable density $f$ with mean zero and finite Fisher information $\fc :=\int \br{f'/f}^2f\diff x$.
Assume $\Delta$ is a constant. 
We have
\begin{align}\label{eqn:minimax}
C_1\ebr{-\frac{\Delta^2}{8\fc ^{-1}}} \leq \lim_{p\rightarrow\infty}\inf_{z}\sup_{z^*\in[2]^n}\E \ell(z,z^*) \leq  C_2\ebr{-\frac{\Delta^2}{8\fc ^{-1}}},
\end{align}
for some constants $C_1,C_2>0$.
\end{lemma}

With Lemma \ref{lem:minimax_spectral}, the question of  whether $\check z$ is optimal or not boils down to a comparison of the variance $\bar \sigma^2$ and the inverse of the Fisher information $\fc^{-1}$. Due to the fact that $\fc^{-1}\leq \bar\sigma^2$ and the equation is true if and only if $F$ is a normal distribution, we have the following conclusion.
\begin{theorem}\label{thm:optimal}
Consider the model (\ref{eqn:simple}).  Assume all the assumptions needed in  Theorem \ref{thm:lower_bound_spectral} and Lemma \ref{lem:minimax_spectral} hold. Then the spectral clustering $\check z$ is in general suboptimal, i.e., it fails to achieve  the minimax rate (\ref{eqn:minimax}). It is optimal if and only if the noise distribution $F$ is $N(0,\bar \sigma^2)$.
\end{theorem}

Theorem \ref{thm:optimal} establishes the sub-optimality of the spectral clustering $\check z$ under the model (\ref{eqn:simple}). 
Though $\check z$ achieves an exponential error rate, it has a fundamentally sub-optimal exponent involving $\bar \sigma^2$ instead of $\fc^{-1}$. This is due to the fact $\check z$ clusters data points based on Euclidean distances, whereas the optimal procedure uses the likelihood ratio test.
Only when the noise is normally distributed, the likelihood ratio test is equivalent to a comparison of two Euclidean distances, leading to the optimality of $\check z$ in the Gaussian case.  Even though that Theorem \ref{thm:optimal} is only limited to the model (\ref{eqn:simple}), the above reasoning suggests the spectral clustering is generally sub-optimal under mixture models  beyond (\ref{eqn:simple}) unless the noise follows a Gaussian distribution.

\section{Discussion}\label{sec:discuss}

\subsection{Potential Applications of Leave-One-Out Singular Subspace Perturbation Analysis}

In this paper, we have primarily applied the developed leave-one-out singular subspace perturbation toolkit to study the performance of spectral clustering in the context of mixture models. However, it is important to highlight that this toolkit holds promise for various other applications that exhibit low-rank structures and require entrywise analysis. Examples of such applications include low-rank matrix denoising, matrix completion, factor analysis, biclustering, and more.

To illustrate the versatility of our approach, consider a simple scenario where the data matrix $W$ is approximately rank-one and  can be expressed as $W = \lambda uv^T + E$. Here, $\lambda$ is a scalar, and $u$ and $v$ are unit vectors. Let $\hat\lambda, \hat u, \hat v$ be the leading singular value, left singular vector,  and right singular vector of $W$. Specifically, $\hat v_i$, the $i$-th coordinate of $\hat v$, can be expressed as $\hat v_i = \hat u^T W_i /\hat \lambda =  (\lambda \hat u^T u /\hat \lambda)v_i +  \hat u^T \epsilon_i /\hat \lambda$,
where $X_i$ and $\epsilon_i$ represent the $i$-th column of $X$ and $E$, respectively. Under suitable regularity conditions, we can observe that the first term, $(\lambda \hat u^T u / \hat \lambda) v_i$, is well-controlled, leaving the perturbation of $\hat v_i$ to be predominantly determined by the second term, $\hat u^T \epsilon_i / \hat \lambda$, which can be approximated as $\hat u^T \epsilon_i / \lambda$. 
Since $| \hat u^T \epsilon_i | = \| \hat u \hat u^T \epsilon_i \|$, we can leverage Theorem \ref{thm:general} to establish a connection between $| \hat u^T \epsilon_i |$ and $\|\hat u_{-i}\hat u_{-i}^T \epsilon_i\|=|\hat u_{-i}^T \epsilon_i|$, where $\hat u_{-i}$ represents the leading left singular vector of the data matrix with the $i$-th column removed. Importantly, the independence between $\hat u_{-i}$ and $\epsilon_i$ can be exploited to analyze the magnitude of $|\hat u_{-i}^T \epsilon_i|$, facilitating an entrywise perturbation analysis for $\hat v_i$.
This demonstrates the potential broader applicability of our leave-one-out singular subspace perturbation analysis beyond spectral clustering and mixture models.

\subsection{Extension to Eigenspace Perturbation}
In this paper, we primarily focus on the analysis of singular subspace perturbations. However, it is worth considering the potential extension of our findings to eigenspace perturbation scenarios.
Let us consider two symmetric matrices, $Y\in\mathbb{R}^{(n-1)\times (n-1)}$ and $\hat{Y}\in\mathbb{R}^{n\times n}$. Here, $\hat{Y}$ is obtained from $Y$ by removing the last row and column of $\hat{Y}$. For simplicity, we assume that $\hat{Y}_{n,n}=0$. We introduce a vector $y_n\in\mathbb{R}^{n-1}$ such that the last row and column of $\hat{Y}$ can be represented as $(y_n^T,0)$ and $(y_n^T,0)^T$, respectively. Let the leading eigenspaces of $Y$ and $\hat{Y}$ be denoted as $U_r\in\mathbb{R}^{(n-1)\times r}$ and $\hat{U}_r\in\mathbb{R}^{n\times r}$, respectively.
In contrast to the singular subspace analysis, we note that $Y$ and $\hat{Y}$ have different dimensionalities. To address this, we consider an augmented matrix $\tilde{U}_r=(U_r^T,0)^T\in\mathbb{R}^{n\times r}$. Analyzing $\|\tilde{U}_r\tilde{U}_r^T - U_rU_r^T\|_{\text{F}}$ leads us to follow a similar proof strategy as employed in Theorem \ref{thm:general}. However, extending the proof from Theorem \ref{thm:general} to cover $\|\tilde{U}_r\tilde{U}_r^T - U_rU_r^T\|_{\text{F}}$ appears to be non-trivial and potentially challenging.

The reason for this challenge lies in the perturbation between $\tilde{U}_r\tilde{U}_r^T$ and $U_rU_r^T$ that not only involves the last column but also the last row of $\hat{Y}$. In particular, the contribution of the last row $(y_n^T,0)$ to the upper bound of $\|\tilde{U}_r\tilde{U}_r^T - U_rU_r^T\|_{\text{F}}$ remains unclear, as it is not accounted for in the current analysis presented in Theorem \ref{thm:general}.
Hence, we defer the analysis of eigenspace perturbations to future research endeavors, recognizing the need for a more comprehensive and specialized treatment of this aspect.

\subsection{Approximated Solution to $k$-means}

Solving the $k$-means problem exactly, as detailed in (\ref{eqn:spectral_low}), can be computationally challenging, particularly for large datasets. To enhance practicality, one might opt for an approximate solution to $k$-means, where the solution's objective value remains within a factor of $(1+\varepsilon)$ of the global minimum. It's worth noting, however, that such an approximate solution may lack a property intrinsic to the global minimizer in (\ref{eqn:spectral_low}): $\hat z_i = \argmin_{a\in[k]} \|\hat U_{1:r}^T X_i -  \hat c_a\|^2$ for every $i\in[n]$, which is  critical to our theoretical analysis.
To circumvent this issue, we can use a strategy delineated in Section 2.5 of \cite{loffler2019optimality}. This approach, devised for addressing a similar problem for spectral clustering under Gaussian mixture models, executes an additional step of Lloyd’s algorithm after obtaining the $(1+\varepsilon)$ solution. As evidenced by Theorem 2.2 in \cite{loffler2019optimality}, the theoretical analysis for this augmented method closely mirrors that of the original. The cost of having the approximate solution is the need for a slightly more stronger signal-to-noise condition. In our context, this means Theorem \ref{thm:subg} would remain valid, albeit with $\psi_1$ carrying an extra $\sqrt{1+\varepsilon}$ factor in its denominator.

\subsection{High-dimensional regime \( p \gg n \)}

In the context where \( k, \beta, \sigma \) are constants, Corollary \ref{cor:subg} and Theorem \ref{thm:GMM} demand the conditions $\Delta/(1+\sqrt{p/n})\rightarrow\infty$ and $\Delta/(1+p/n)\rightarrow\infty$ respectively. In the low-dimensional scenario, where \( p \lesssim n \), these conditions can be equivalently expressed as \( \Delta \rightarrow \infty \) that is recognized as optimal. Nevertheless, in the high-dimensional case \( p \gg n \), these conditions are deemed sub-optimal.
For a two-component symmetric isotropic Gaussian mixture model, \cite{cai2018rate} demonstrates that spectral clustering remains consistent as long as $\Delta/(p/n)^{1/4}\rightarrow\infty$. More recently, for sub-Gaussian mixture models, under this condition, exponential misclustering errors are obtained in \cite{giraud2019partial} through semi-definite programming (SDP) and in \cite{abbe2020ell_p, ndaoud2018sharp} through a variant of spectral clustering that employs the leading eigenvectors of a hollowed gram matrix $\mathcal{H}(X^TX)\in\mathr^{n\times n}$, where $\mathcal{H}(\cdot)$ is the hollowing operator that zeros out all diagonal entries of a square matrix.
In addition, it is suggested in \cite{abbe2020ell_p} that hollowing is crucial for spectral clustering in high-dimensional and heteroscedastic scenarios. It provides counterexamples showing that the leading eigenvectors of $X^TX$ can be asymptotically orthogonal to their population counterparts. In contrast, those of the hollowed matrix $\mathcal{H}(X^TX)$ remain consistent. Our more stringent conditions, as compared to $\Delta/(p/n)^{1/4}\rightarrow\infty$, stem from challenges inherent in our analysis, possibly related to our use of the gram matrix, as opposed to $\mathcal{H}(X^TX)$. 

\subsection{Explicit Error Rate of Spectral Clustering under Other Mixture Models}
As our analysis in this paper establishes an explicit error rate under sub-Gaussian mixture models, a natural question is whether our analysis framework can be extended to other mixture models. A key observation is that the clustering error bound in Lemma \ref{lem:decomposition} imposes no specific assumptions on the noise distribution $\{\epsilon_i\}$, allowing for potential applicability to a wide range of mixture models. However, this flexibility comes with challenges. Lemma \ref{lem:decomposition} highlights that the clustering error is intimately tied to the tail probabilities of $\normop{\hat U_{-i,1:\kr}^T\epsilon_i}$. While the independence between $\hat U_{-i,1:\kr}$ and $\epsilon_i$ is advantageous, the lack of explicit expressions for $\hat U_{-i,1:\kr}$ poses difficulties when dealing with other noise distributions.

When $\epsilon_i$ follows a sub-Gaussian distribution, existing concentration inequalities can be applied to analyze the norm of $\hat U_{-i,1:\kr}^T\epsilon_i$, providing a sharp upper bound as in Theorem \ref{thm:subg}. However, in scenarios where $\epsilon_i$ is assumed to follow a specific distribution, such as a centered Bernoulli random vector with success probability $q$ decreasing as $n$ grows (as encountered in community detection tasks), issues arise. Despite modeling $\epsilon_i$ as $\text{SG}_p(1)$, the correct variance is $q$, leading to a loose upper bound for spectral clustering performance.
Directly analyzing $\normop{\hat U_{-i,1:\kr}^T\epsilon_i}$ becomes challenging in such cases due to the lack of explicit expressions for $\hat U_{-i,1:\kr}$ and uncertainties about the behavior of its entries. It is important to acknowledge that our current analysis framework has limitations when confronted with these complexities.
Future research in this direction may involve exploring novel techniques or adapting existing methodologies to handle non-sub-Gaussian noise distributions more effectively, thereby establishing sharp analysis for spectral clustering under diverse mixture models.

\subsection{Unknown $k$ or $\sigma$} In this paper, we assume $k$, the number of clusters, is known. If $k$ is unknown, one can employ existing methodologies, as found in the literature \cite{Wang10, TibshiraniWaltherHastie01, MontiTamayoMesirovGolub03, von2007tutorial}, to estimate its value prior to applying our spectral clustering method. Our theoretical results maintain their validity, given that $k$ is accurately estimated, albeit with an added term accounting for the estimation error of $k$.
However, while such methods have empirically demonstrated decent performance, their theoretical performances are not fully understood, especially  in contexts where both $p,n$ are large. Regarding $\sigma$, the noise level in sub-Gaussian mixture models, both Algorithm \ref{alg:main} and Algorithm \ref{alg:thresholding} require no prior knowledge of $\sigma$.
However, in Theorem \ref{thm:GMM_2}, the threshold $T$ is needed to satisfy a condition involving $\sigma$. More generally, in Lemma \ref{lem:thresholding}, $T/\|E\|$ needs to be bounded away from 0. To endow the algorithm with enhanced adaptability, one possible approach is to consider $\hat\lambda_{k+1}$,  the $(k+1)$th largest singular value of the data matrix, as a surrogate of $\|E\|$.   The intuition is that when entries of the noise matrix $E$ are independent and identically distributed,  asymptotic behavior of its singular values can be characterized using random matrix theory, building a connection between $\|E\|$ and its leading singular values.   Further investigation is beyond the scope of this paper.

\section{Proof of Main Results in Section \ref{sec:perturbation}}\label{sec:proof_perturbation_sec_2}
In this section, we give the proofs of Theorem \ref{thm:general} and Theorem \ref{thm:perturbation}. The proof of Theorem \ref{thm:perturbation_r} is included in the supplement \cite{supplement} due to page limit.

\subsection{Proof of Theorem \ref{thm:general}}\label{sec:proof_general}

Before giving the proof of Theorem \ref{thm:general}, we first present and prove a slightly more general perturbation result, Theorem \ref{thm:general_more}, which  only requires $\sigma_\rr^2 - \sigma_{\rr+1}^2 - \normop{(I-U_\rr U_\rr^T)y_n}^2>0$ instead of assuming $\rho >2$. We defer the proof of Theorem \ref{thm:general} to the end of this section, which is an immediate consequence of  Theorem \ref{thm:general_more}.

\begin{theorem}\label{thm:general_more}
If $\sigma_\rr ^2 - \sigma_{\rr +1}^2 - \normop{(I-U_\rr U_\rr ^T)y_n}^2>0,$
we have
\begin{align*}
\fnorm{\hat U_\rr  \hat U_\rr ^T-   U_\rr U_\rr ^T} \leq  \frac{2\sqrt{2}{  \sigma_\rr  }\norm{(I-U_\rr U_\rr ^T)y_n}}{\sigma_\rr ^2 - \sigma_{\rr +1}^2 -  \norm{(I-U_\rr U_\rr ^T)y_n}^2}\sqrt{\sum_{i=1}^\rr\br{ \frac{u_i^T y_n}{\sigma_i}}^2}.
\end{align*}
\end{theorem}
\begin{proof}
Decompose $y_n$ into $y_n=\theta + \epsilon$ with $\theta := U_\rr U_\rr ^Ty_n$ and $\epsilon:=(I-U_\rr U_\rr ^T)y_n$. Then 
we have $u_i^T\theta = u_i^Ty_n$  for each $i\in[\rr ]$.

Throughout the proof, we denote
\begin{align*}
\alpha^2 = \fnorm{\hat U_\rr  \hat U_\rr ^T-   U_\rr U_\rr ^T}^2.
\end{align*}
Denote $d=p\wedge (n-1)$. If $p\leq n-1$, we have $d=p$ and denote $U:=(u_1,\ldots,u_p)\in\mathr^{p\times p}$ which is an orthogonal matrix. If $p> n-1$, we let $U\in\mathr^{p\times p}$ be an orthogonal matrix with the first $p\wedge (n-1)$ columns being $u_1,\ldots,u_{p\wedge(n-1)}.$ In both cases, we have $U$ being an orthogonal matrix. Then $\hat U_\rr $ can be written as $\hat U_\rr  = U\hat B$ for some $\hat B=(\hat B_{i,j})\in\mathr^{p\times \rr }$. Let $\hat B_{i,\cdot}$ be the $i$th row of $\hat B$ for each $i\in[p]$. Define $b_i^2 = 1-\|\hat B_{i,\cdot}\|^2$ for each $i\in[\rr ]$ and $b_i^2 = \|\hat B_{i,\cdot}\|^2$ for each $i>\rr $. Then we have
\begin{align}
\alpha^2 &= \fnorm{\hat U_\rr  \hat U_\rr ^T}^2 + \fnorm{U_\rr U_\rr ^T}^2 - 2\iprod{\hat U_\rr  \hat U_\rr ^T}{U_\rr U_\rr ^T} \nonumber  \\
&=2k -2 \fnorm{U_\rr ^T\hat U_\rr }^2  = 2k -2 \sum_{i\in[\rr ]}\sum_{j\in[\rr ]}\hat B_{i,j}^2\nonumber \\
 & = 2\sum_{i\in[\rr ]} b_i^2 = 2\sum_{i=\rr +1}^p b_i^2,\label{eqn:b_alpha}
\end{align}
where in the last equation we use the fact that $\|\hat B\|_\text{F}^2=\rr $.

Note that $\hat U_\rr \hat U_\rr^T  \hat Y$ is the best rank-$\rr $ approximation of $\hat Y$. We have
\begin{align*}
\fnorm{\br{I - \hat U_\rr \hat U_\rr ^T} \hat Y}^2 \leq \fnorm{\br{I -  U_\rr   U_\rr ^T} \hat Y}^2.
\end{align*}
Due to the fact $\hat Y = \br{ Y, y_n}$, we have
\begin{align*}%
\fnorm{\br{I - \hat U_\rr \hat U_\rr ^T}  Y}^2 + \norm{\br{I - \hat U_\rr \hat U_\rr ^T} y_n}^2 \leq \fnorm{\br{I -  U_\rr   U_\rr ^T}   Y}^2 + \norm{\br{I -  U_\rr   U_\rr ^T} y_n}^2,
\end{align*}
which implies
\begin{align}\label{eqn:inequality_2}
\fnorm{\br{I - \hat U_\rr \hat U_\rr ^T}  Y}^2 - \fnorm{\br{I -  U_\rr   U_\rr ^T}   Y}^2 \leq \norm{\br{I -  U_\rr   U_\rr ^T} y_n}^2 - \norm{\br{I - \hat U_\rr \hat U_\rr ^T} y_n}^2.
\end{align}
We are going to simplify terms in (\ref{eqn:inequality_2}).

~\\
\emph{(Simplification of the LHS of (\ref{eqn:inequality_2})).} Recall the decomposition $Y = \sum_{i\in[d]}\sigma_iu_iv_i^T$. Since $\br{I -  U_\rr   U_\rr ^T}   Y = \sum_{i>\rr }^d \sigma_i u_i v_i^T$, we have $ \fnorm{\br{I -  U_\rr   U_\rr ^T}   Y}^2 = \sum_{i>\rr }^d \sigma_i^2$. Since
\begin{align*}
U^TY = U^T\br{\sum_{i\in[d]}\sigma_i u_iv_i^T} =\begin{pmatrix}
\sigma_1 v_1^T\\
\ldots\\
\sigma_d v_d^T\\
0_{p-d}
\end{pmatrix}
=
\text{diag}(\sigma_1,\ldots,\sigma_d,0_{p-d})
\begin{pmatrix}
 v_1^T\\
\ldots\\
 v_d^T\\
O_{(p-d)\times n}
\end{pmatrix},
\end{align*}
we have
\begin{align*}
\fnorm{\br{I - \hat U_\rr \hat U_\rr ^T}  Y}^2 &= \fnorm{U\br{I - U^T\hat U_\rr \hat U_\rr ^TU}U^T Y}^2  \\
&= \fnorm{\br{I - \hat B\hat B^T}\text{diag}(\sigma_1,\ldots,\sigma_d,0_{p-d})
\begin{pmatrix}
 v_1^T\\
\ldots\\
 v_d^T\\
O_{(p-d)\times n}
\end{pmatrix}}^2  \\
&= \tr{\text{diag}(\sigma_1,\ldots,\sigma_d,0_{p-d})\br{I-\hat B\hat B^T}
\text{diag}(\sigma_1,\ldots,\sigma_d,0_{p-d})
\begin{pmatrix}
I_{d\times d} &\\
& O_{(p-d)\times (p-d)}
\end{pmatrix}
},
\end{align*}
where in the last equation we use the following facts: (1)  for any two square matrices of the same size $A,D$, we have $\fnorm{AD}^2 = \text{tr}(D^TA^TAD) =  \text{tr}(A^TADD^T) $; (2) $\hat B$ has orthogonal columns such that $(I-\hat B\hat B^T)^2= I-\hat B\hat B^T$; and (3) $\{v_1,\ldots,v_d\}\in\mathr^{n-1}$ are orthogonal vectors. Since the diagonal entries of $\hat B\hat B^T$ are $\{\|\hat B_{i,\cdot}\|^2\}_{i\in[p]}$, we have 
\begin{align*}
\fnorm{\br{I - \hat U_\rr \hat U_\rr ^T}  Y}^2 & = \tr{\text{diag}(\sigma_1,\ldots,\sigma_d,0_{p-d})\br{I-\hat B\hat B^T}
\text{diag}(\sigma_1,\ldots,\sigma_d,0_{p-d})}\\
&= \sum_{i=1}^d\sigma_i^2 \br{1-\fnorm{\hat B_{i,\cdot}}^2}.
\end{align*}
Then we have
\begin{align*}
\text{LHS of (\ref{eqn:inequality_2})} &= \sum_{i=1}^\rr \sigma_i^2\br{1- \fnorm{\hat B_{i,\cdot}}^2 } - \sum_{i>\rr }^d \sigma_i^2\fnorm{\hat B_{i,\cdot}}^2 = \sum_{i=1}^\rr  \sigma_i^2b_i^2- \sum_{i>\rr }^d\sigma_i^2b_i^2\geq \sum_{i=1}^\rr  \sigma_i^2b_i^2 - \sigma_{\rr +1}^2 \frac{\alpha^2}{2},
\end{align*}
where we use $\sum_{i>\rr }^d b_i^2\leq \sum_{i>\rr }^p b_i^2=\alpha^2/2$ from (\ref{eqn:b_alpha}) in the last inequality .

~\\
\emph{(Simplification of the RHS of (\ref{eqn:inequality_2})).} Recall that $\hat U_\rr  = U\hat B$. We decompose it into  $\hat B = (\hat B_1^T, \hat B_2^T)^T$ where $\hat B_1\in\mathr^{\rr \times \rr }$ are the first $\rr $ rows and $\hat B_2 \in\mathr^{(p-\rr )\times \rr }$.
We have
\begin{align*}
\text{RHS of (\ref{eqn:inequality_2})}  &= y_n^T\br{I -  U_\rr   U_\rr ^T} y_n - y_n^T\br{I - \hat U_\rr \hat U_\rr ^T}  y_n \\
 & = y_n^T\br{\hat U_\rr \hat U_\rr ^T -  U_\rr   U_\rr ^T} y_n \\
 &= y_n^T U\begin{pmatrix}
 \hat B_1 \hat B_1^T -I_{\rr\times \rr} &  \hat B_1 \hat B_2^T \\
  \hat B_2 \hat B_1^T &  \hat B_2 \hat B_2^T
 \end{pmatrix} U^Ty_n.
\end{align*}
Define $\hat B^\perp\in\mathr^{p\times (p-\rr )}$ to be the matrix such that $(\hat B, \hat B^\perp )\in\mathr^{p\times p}$ is an orthonormal matrix. We can further decompose it into  $\hat B^\perp = (\hat B_1^{\perp T},\hat B_2^{\perp Y})^T$ where $\hat B_1^\perp \in\mathr^{\rr \times (p-\rr )}$ including the first $\rr $ rows and $\hat B_2^\perp \in\mathr^{(p-\rr )\times (p-\rr )}$. Since $(\hat B, \hat B^\perp )$ has orthogonal columns, we have
\begin{align*}
(\hat B_1,\hat B_1^\perp) (\hat B_1,\hat B_1^\perp)^T =\hat B_1 \hat B_1^T + \hat B_1^\perp \hat B_1^{\perp T} = I_{\rr\times \rr},
\end{align*}
and $(\hat B_1,\hat B_1^\perp)(\hat B_2,\hat B_2^\perp) ^T=O_{r\times (p-r)}$, which implies
\begin{align*}
\hat B_1\hat B_2^T=-\hat B_1^\perp\hat B_2^{\perp T}.
\end{align*}
We also decompose the matrix $U =: (U_\rr , U_\perp)$. Then
\begin{align*}
\text{RHS of (\ref{eqn:inequality_2})}  &= y_n^T (U_\rr ,U_\perp)\begin{pmatrix}
- \hat B_1^\perp \hat B_1^{\perp T} &  -\hat B_1^\perp\hat B_2^{\perp T} \\
 -\hat B_2^\perp\hat B_1^{\perp T}  &  \hat B_2 \hat B_2^T
 \end{pmatrix} (U_\rr ,U_\perp)^Ty_n\\
 & =  - y_n^T U_\rr \hat B_1^\perp \hat B_1^{\perp T}U_\rr ^T y_n -2y_n^T U_\rr  \hat B^\perp_1 \hat B_2^{\perp T}U_\perp ^T y_n + y_n^T U_\perp \hat B_2 \hat B_2^T U_\perp ^T y_n \\
 &\leq   - \norm{ \hat B_1^{\perp T}U_\rr ^T y_n}^2 +  2\norm{ \hat B_1^{\perp T}U_\rr ^T y_n} \norm{\hat B_2^{\perp T}}\norm{U_\perp ^T y_n} +  \norm{ \hat B_2^T}^2  \norm{U_\perp ^T y_n}^2.
\end{align*}
Note that $ \normop{\hat B_2^{\perp T}}\leq 1$ and $\normop{ \hat B_2^T}^2\leq \normf{\hat B_2^T}^2 = \sum_{i>\rr }^p\normop{\hat B_{i,\cdot}}^2= \alpha^2/2$ which is by (\ref{eqn:b_alpha}). We also have 
\begin{align*}
\norm{U_\perp ^T y_n} = \norm{\epsilon}.
\end{align*}
Since $\normf{\hat B_1^{\perp}}^2 = \sum_{i=1}^\rr \br{1-\normop{\hat B_{i,\cdot}}^2} =\alpha^2/2$ according to (\ref{eqn:b_alpha}), we have $\normop{\hat B_1^{\perp}}\leq \alpha/\sqrt{2}$. Thus, using $U_\rr ^T\epsilon=0$, we have
\begin{align*}
\norm{ \hat B_1^{\perp T}U_\rr ^T y_n} &= \norm{ \hat B_1^{\perp T}U_\rr ^T \theta}.
\end{align*}
Then,
\begin{align*}
\text{RHS of (\ref{eqn:inequality_2})} &\leq 2  \norm{ \hat B_1^{\perp T}U_\rr ^T \theta}  \norm{\epsilon}  + \frac{\alpha^2}{2}\norm{\epsilon}^2.
\end{align*}
To simplify $\normop{ \hat B_1^{\perp T}U_\rr ^T \theta}$, denote $w_i = u_i^T\theta$ and $s_i = \abs{w_i}/\sigma_i$ for each $i\in[\rr]$. Recall that $u_i^T\theta = u_i^Ty_n$  for each $i\in[\rr ]$. We have
\begin{align*}
s_i = \abs{\frac{u_i^Ty_n}{\sigma_i}},\forall i\in[\rr].
\end{align*}
We then have
\begin{align*}
\norm{ \hat B_1^{\perp T}U_\rr ^T \theta} = \norm{\sum_{i=1}^\rr  w_i \hat B^\perp_{i,\cdot} } \leq \sum_{i=1}^\rr  \abs{w_i} \norm{\hat B_{i,\cdot}^\perp} = \sum_{i=1}^\rr  s_i\sigma_i \abs{b_i}\leq \norm{s}\sqrt{\sum_{i=1}^\rr \sigma_i^2b_i^2},
\end{align*}
where we denote the $i$th row of $\hat B_1^{\perp }$ as $\hat B_{i,\cdot}^{\perp} $ and we use the fact that $\normop{\hat B_{i,\cdot}^{\perp}}^2 = 1- \normop{\hat B_{i,\cdot}}^2 = b_i^2$ for each $i\in[\rr ]$.
As a result,
\begin{align*}
\text{RHS of (\ref{eqn:inequality_2})} &\leq 2 {\norm{s}\sqrt{\sum_{i=1}^\rr \sigma_i^2 b_i^2}
}\norm{\epsilon}+ \frac{\alpha^2}{2}\norm{\epsilon}^2.
\end{align*}

~\\
\emph{(Combining the above simplifications for  (\ref{eqn:inequality_2})).} From the above simplifications on the LHS and RHS of (\ref{eqn:inequality_2}), we have
\begin{align*}
\sum_{i=1}^\rr  \sigma_i^2b_i^2 - \sigma_{\rr +1}^2 \frac{\alpha^2}{2} &\leq  2\norm{s} \sqrt{ \sum_{i=1}^\rr \sigma_i^2 b_i^2}  \norm{\epsilon} 
 + \frac{\alpha^2}{2}\norm{\epsilon}^2.
\end{align*}
Define $t= \sqrt{ \sum_{i=1}^\rr \sigma_i^2 b_i^2} $. 
Then after arrangement, the above display becomes
\begin{align*}
t^2 - 2 \norm{s}\norm{\epsilon}t &\leq  \sigma_{\rr +1}^2 \frac{\alpha^2}{2}  
+ \frac{\alpha^2}{2}\norm{\epsilon}^2.
\end{align*}
Note that the function $t^2 - 2\norm{s}\norm{\epsilon}t $ is increasing as long as $t\geq t_0$ where we define $t_0 := \norm{s}\norm{\epsilon}$. On the other hand, from (\ref{eqn:b_alpha}), we have the domain $t\geq \alpha \sigma_\rr /\sqrt{2}$. We consider the following two scenarios.

If $\alpha \sigma_\rr /\sqrt{2} \leq t_0$, we have
\begin{align}\label{eqn:alpha_1}
\alpha \leq \frac{\sqrt{2}t_0}{\sigma_\rr } =\frac{\sqrt{2}\norm{s}\norm{\epsilon}}{\sigma_\rr }.
\end{align}

If $\alpha \sigma_\rr /\sqrt{2} > t_0$, we have
\begin{align*}
t^2 - 2 \norm{s}t &\geq \frac{\alpha^2\sigma_\rr ^2}{2} -\sqrt{2}\norm{s}\norm{\epsilon} \alpha \sigma_\rr .
\end{align*}
Hence, we have an inequality of $\alpha$:
\begin{align*}
\frac{\alpha^2\sigma_\rr ^2}{2} -\sqrt{2}\norm{s}\norm{\epsilon} \alpha \sigma_\rr  &\leq  \sigma_{\rr +1}^2 \frac{\alpha^2}{2}  
+ \frac{\alpha^2}{2}\norm{\epsilon}^2,
\end{align*}
which can be arranged into
\begin{align*}
\frac{\alpha}{2}\br{\sigma_\rr ^2 - \sigma_{\rr +1}^2 - \norm{\epsilon}^2} \leq  { \sqrt{2}\norm{s}\sigma_\rr  
 }\norm{\epsilon}.
\end{align*}
Hence, under the assumption $\sigma_\rr ^2 - \sigma_{\rr +1}^2 - \norm{\epsilon}^2>0$, we have
\begin{align}\label{eqn:alpha_2}
\alpha \leq  \frac{2\sqrt{2}{ \sigma_\rr\norm{s}   }\norm{\epsilon}}{\sigma_\rr ^2 - \sigma_{\rr +1}^2 - \norm{\epsilon}^2}.
\end{align}
Since $2\sigma_\rr ^2>\sigma_\rr ^2 - \sigma_{\rr +1}^2 - \norm{\epsilon}^2$, the upper bound in (\ref{eqn:alpha_1}) is strictly below that in (\ref{eqn:alpha_2}). Hence, (\ref{eqn:alpha_2}) holds for both scenarios. The proof is complete.
\end{proof}

\begin{proof}[Proof of Theorem \ref{thm:general}]
Since we assume $\rho>2$, we have 
\begin{align*}
\sigma_\rr ^2 - \sigma_{\rr +1}^2 - \norm{(I-U_\rr U_\rr ^T)\epsilon}^2 &\geq \sigma_\rr (\sigma_\rr  -\sigma_{\rr +1}) - (\sigma_\rr  -\sigma_{\rr +1})^2/4 \\
&\geq  \sigma_\rr (\sigma_\rr  -\sigma_{\rr +1})/2=\rho\sigma_\rr  \norm{(I-U_\rr U_\rr ^T)\epsilon}/2.
\end{align*}
Together with Theorem \ref{thm:general_more}, we obtain the desired bound.
\end{proof}

\subsection{Proof of Theorem \ref{thm:perturbation}}\label{sec:proof_perturbation}

\begin{proof}[Proof of Theorem \ref{thm:perturbation}]
Consider any $i\in[n]$.  In order to apply Theorem \ref{thm:general}, we need to verify that the spectral gap assumption (\ref{eqn:general_condition3}) is satisfied. That is, define
\begin{align*}
\rho_{-i}:=\frac{\hat\lambda_{-i,\kr} - \hat\lambda_{-i,\kr+1}}{\norm{\br{I- \hat U_{-i,1:\kr} \hat U_{-i,1:\kr}^T} X_i}}.
\end{align*}
We need to show $\rho_{-i}>2$. In the following, we provide a lower bound for the numerator $\hat\lambda_{-i,\kr} - \hat\lambda_{-i,\kr+1}$.

Define $\lambda_{-i,1}\geq \lambda_{-i,2}\geq \ldots\geq \lambda_{-i,p\wedge(n-1)}$ to be singular values of  $P_{-i}$, the leave-one-out counterpart of the signal matrix  $P$ where
\begin{align}\label{eqn:P_i_def}
P_{-i}:=(\theta^*_{z^*_1},\ldots,\theta^*_{z^*_{i-1}},\theta^*_{z^*_{i+1}},\ldots, \theta^*_{z^*_n})\in\mathr^{p\times (n-1)}.
\end{align}
We are interested in the value of $\lambda_{-i,\kr}$.
Recall that $\lambda_{\kr}$ is the $\kr$th largest singular value of $P$ which is rank-$\kr$. Since $P$ has $k$ unique columns $\{\theta^*_a\}_{a\in[k]}$, its left singular vectors $u_j\in \Theta$ for each $j\in[k]$ where $\Theta:=\text{span}(\{\theta^*_a\}_{a\in[k]})$.  Note that each $\theta^*_a$ appears  at least $\beta n/k$ times in the columns of $P$. Then $P_{-i}$ also has these $k$ unique columns with each appearing at least $\beta n/k-1$ times. This concludes that $P_{-i}$ has the same leading left singular vector space as $P$. We then have
\begin{align}
\nonumber\lambda_{-i,\kr}^{2} & = \min_{w\in \Theta:\norm{w}=1}\norm{w^T P_{-i}}^2  = \min_{w\in \Theta :\norm{w}=1} \sum_{j\in[n]:j\neq i} (w^T\theta^*_{z^*_j})^2 \\
\nonumber&\geq \frac{\frac{\beta n}{k}-1}{\frac{\beta n}{k}}   \min_{w\in \Theta :\norm{w}=1} \sum_{j\in[n]} (w^T\theta^*_{z^*_j})^2 = \br{1-\frac{k}{\beta n}} \min_{w\in \Theta :\norm{w}=1} \norm{w^T P}^2 \\
\label{eqn:thm_2_proof_1}&\geq  \br{1-\frac{k}{\beta n}} \lambda_{\kr}^2.
\end{align}
We also have $\lambda_{-i,\kr+1}=0$ as $P_{-i}$ is rank-$\kr$.

Next, we are going to analyze $\hat \lambda_{-i,\kr}$ and $\hat \lambda_{-i,\kr+1}$, the $\kr$th and $(\kr+1)$th largest singular values of $X_{-i}$. Recall the SVD of $X_{-i}$ in Section \ref{sec:perturbation_mixture}. Define 
\begin{align}\label{eqn:E_i_def}
E_{-i}:=(\epsilon_1,\ldots,\epsilon_{i-1},\epsilon_{i+1},\ldots,\epsilon_n)\in\mathr^{p\times (n-1)},
\end{align}
 so that $X_{-i}=P_{-i} + E_{-i}$.  By Weyl's inequality, we have $|{\lambda_{-i,\kr} - \hat\lambda_{-i,\kr}}|,|{\lambda_{-i,\kr+1}-\hat\lambda_{-i,\kr+1}}|\leq \norm{E_{-i}}\leq \norm{E}$. Then we have
\begin{align}\label{eqn:perturbation_proof_3}
\hat\lambda_{-i,\kr} \geq \lambda_{-i,\kr}-\norm{E} \geq \sqrt{1-\frac{k}{\beta n}}\lambda_\kr-\norm{E} 
\end{align}
and
\begin{align}\label{eqn:perturbation_proof_1}
\hat\lambda_{-i,\kr}-\hat\lambda_{-i,\kr+1}\geq \lambda_{-i,\kr}-\lambda_{-i,\kr +1}-2\norm{E}\geq \sqrt{1-\frac{k}{\beta n}}\lambda_{\kr} - 2\norm{E}.
\end{align}

Next, we study $\normop{(I- \hat U_{-i,1:\kr} \hat U_{-i,1:\kr}^T)X_i}$.
Since $\hat U_{-i,1:\kr}\hat U_{-i,1:\kr}^T X_{-i}$ is the best rank-$\kr$ approximation of $X_{-i}$, we have
\begin{align*}
\norm{ \hat U_{-i,1:\kr} \hat U_{-i,1:\kr}^TX_{-i} - X_{-i}} \leq \norm{P_{-i} - X_{-i}} = \norm{E_{-i}},
\end{align*}
where we use the fact that $P_{-i}$ is rank-$\kr$. Then by the triangle inequality, we have 
\begin{align*}
&\norm{\br{I- \hat U_{-i,1:\kr} \hat U_{-i,1:\kr}^T} P_{-i} } \\
&= \norm{ \hat U_{-i,1:\kr} \hat U_{-i,1:\kr}^T P_{-i} - P_{-i}} \\
&\leq \norm{ \hat U_{-i,1:\kr} \hat U_{-i,1:\kr}^T (P_{-i} - X_{-i} )} +\norm{ \hat U_{-i,1:\kr} \hat U_{-i,1:\kr}^T X_{-i} - X_{-i}} + \norm{X_{-i} -  P_{-i}}\\
&\leq  3\norm{E_{-i}}.
\end{align*}
Using the fact $P_{-i}$ is rank-$\kr$ again, we have
\begin{align*}%
\fnorm{\br{I- \hat U_{-i,1:\kr} \hat U_{-i,1:\kr}^T} P_{-i} } \leq \sqrt{\kr} \norm{\br{I- \hat U_{-i,1:\kr} \hat U_{-i,1:\kr}^T} P_{-i} }  \leq 3\sqrt{\kr}\norm{E_{-i}}\leq 3\sqrt{\kr}\norm{E}.
\end{align*}
Since $P_{-i}$ has at least $\beta n/k-1$ columns being exactly $\theta^*_{z^*_i}$, we have
\begin{align}\label{eqn:perturbation_proof_2}
\norm{\br{I- \hat U_{-i,1:\kr} \hat U_{-i,1:\kr}^T} \theta^*_{z^*_i}}\leq \frac{\fnorm{\br{I- \hat U_{-i,1:\kr} \hat U_{-i,1:\kr}^T} P_{-i} } }{\sqrt{\frac{\beta n}{k} - 1}} \leq \frac{3\sqrt{\kr}\norm{E}}{\sqrt{\frac{\beta n}{k} - 1}},
\end{align}
and consequently,
\begin{align}
\nonumber\norm{\br{I- \hat U_{-i,1:\kr} \hat U_{-i,1:\kr}^T} X_i}&\leq \norm{\br{I- \hat U_{-i,1:\kr} \hat U_{-i,1:\kr}^T} \theta^*_{z^*_i}} + \norm{\br{I- \hat U_{-i,1:\kr} \hat U_{-i,1:\kr}^T} \epsilon_i}\\
&\leq \frac{3\sqrt{\kr}\norm{E}}{\sqrt{\frac{\beta n}{k} - 1}} + \norm{E}.\label{eqn:perturbation_proof_2_1}
\end{align}

From (\ref{eqn:perturbation_proof_1}) and (\ref{eqn:perturbation_proof_2_1}), we have
\begin{align}\label{eqn:rho_prime_minus_i_lower_bound}
\rho_{-i}\geq \frac{\sqrt{1-\frac{k}{\beta n}}\lambda_{\kr} - 2\norm{E}}{\norm{E}+\frac{3\sqrt{\kr}\norm{E}}{\sqrt{\frac{\beta n}{k} - 1}}} \geq \frac{\rho_0}{8}>2,
\end{align} 
where the last inequality is due to the assumption $\rho_0>16$ and  $\beta n/k^2 \geq 10$. 

The next thing to do is to study $\{\hat u_{-i,a}^T X_i\}_{a\in[\kr]}$.
Denote the columns of $P_{-i}$ and $E_{-i}$ as $\{(P_{-i})_{\cdot,j}\}_{j\in[n-1]}$ and $\{(E_{-i})_{\cdot,j}\}_{j\in[n-1]}$, respectively. Define $S:=\{{j\in[n-1]: (P_{-i})_{\cdot,j} =\theta^*_{z^*_i}}\}$. Then for any $a\in[\kr]$, by the SVD of $X_{-i}$, we have 
\begin{align*}
\hat u_{-i,a}^T\theta^*_{z^*_i} & = \frac{1}{\abs{S}}\sum_{j\in S} \hat u_{-i,a}^T (P_{-i})_{\cdot,j} = \frac{1}{\abs{S}}\sum_{j\in S} \hat u_{-i,a}^T (X_{-i})_{\cdot,j}  +   \frac{1}{\abs{S}}\sum_{j\in S} \hat u_{-i,a}^T (E_{-i})_{\cdot,j}\\
& = \frac{1}{\abs{S}}\sum_{j\in S}\hat\lambda_{-i,a} (v_{-i,a})_j+   \frac{1}{\abs{S}} \hat u_{-i,a}^T \br{\sum_{j\in S}(E_{-i})_{\cdot,j}}.
\end{align*}
Hence, by Cauchy-Schwarz inequality and the fact that $\norm{v_{-i,a}}=1$, we have
\begin{align}\label{eqn:perturbation_proof_5}
\abs{\hat u_{-i,a}^T\theta^*_{z^*_i}}&\leq  \hat\lambda_{-i,a} \frac{\sqrt{\abs{S}}}{\abs{S}} +  \frac{\sqrt{\abs{S}}\norm{E_{-i}}}{\abs{S}} \leq \frac{\hat\lambda_{-i,a}}{\sqrt{\frac{\beta n}{k}-1}} + \frac{\norm{E}}{\sqrt{\frac{\beta n}{k}-1}}.
\end{align}
Since $|\hat u_{-i,a}^TX_i|\leq |\hat u_{-i,a}^T\theta^*_{z^*_i}| + |\hat u_{-i,a}^T\epsilon_i|$,
we have
\begin{align*}
\frac{|\hat u_{-i,a}^TX_i|}{\hat\lambda_{-i,a}} &\leq  \frac{1}{\sqrt{\frac{\beta n}{k}-1}} +\frac{1}{\hat\lambda_{-i,a}} \br{ \frac{\norm{E}}{\sqrt{\frac{\beta n}{k}-1}} + |\hat u_{-i,a}^T\epsilon_i|}\\
&\leq  \frac{1}{\sqrt{\frac{\beta n}{k}-1}} +\frac{1}{\hat\lambda_{-i,\kr}} \frac{\norm{E}}{\sqrt{\frac{\beta n}{k}-1}} +\frac{1}{\hat\lambda_{-i,\kr}} |\hat u_{-i,a}^T\epsilon_i|.
\end{align*}
Consequently,
\begin{align*}
\sqrt{\sum_{a\in\kr} \br{\frac{\hat u_{-i,a}^TX_i}{\hat\lambda_{-i,a}}}^2} &\leq \frac{\sqrt{\kr}}{\sqrt{\frac{\beta n}{k}-1}} +  \frac{1}{\hat\lambda_{-i,\kr}}  \frac{\norm{E}\sqrt{\kr}}{\sqrt{\frac{\beta n}{k}-1}}  +\frac{1}{\hat\lambda_{-i,\kr}}\norm{\hat U_{-i,1:\kr}\hat U_{-i,1:\kr}^T\epsilon_i},
\end{align*}
where we use the fact $\normop{\hat U_{-i,1:\kr}\hat U_{-i,1:\kr}^T\epsilon_i} = \normop{\hat U_{-i,1:\kr}^T\epsilon_i} = (\sum_{i\in[\kr]}(\hat u_{-i,a}^T\epsilon_i)^2)^{1/2}$.

Lastly, by Theorem \ref{thm:general}, we have
\begin{align*}
\fnorm{{\hat U_{1:\kr} \hat U_{1:\kr}^T - \hat U_{-i,1:\kr}\hat U_{-i,1:\kr}^T}} &\leq \frac{4\sqrt{2}}{\rho_{-i}}\br{\frac{\sqrt{\kr}}{\sqrt{\beta n/k-1}} +\frac{1}{\hat\lambda_{-i,\kr}}\br{\frac{\sqrt{\kr}\norm{E}}{\sqrt{\beta n/k-1}}+ \norm{\hat U_{-i,1:\kr}\hat U_{-i,1:\kr}^T\epsilon_i}}}.
\end{align*}
Since $\beta n/k^2\geq 10$ and $\rho_0>16$ are assumed, we have $\hat\lambda_{-i,\kr} \geq \lambda_{\kr}/2$ by (\ref{eqn:perturbation_proof_3}). Then together with (\ref{eqn:rho_prime_minus_i_lower_bound}), the above display can be simplified into
\begin{align*}
\fnorm{{\hat U_{1:\kr} \hat U_{1:\kr}^T - \hat U_{-i,1:\kr}\hat U_{-i,1:\kr}^T}}&\leq \frac{32\sqrt{2}}{\rho_0} \br{\frac{2\sqrt{k\kr}}{\sqrt{\beta n}} +\frac{2\norm{\hat U_{-i,1:\kr}\hat U_{-i,1:\kr}^T\epsilon_i}}{\lambda_{\kr}}}\\
&\leq \frac{128}{\rho_0}\br{\frac{\sqrt{k\kr}}{\sqrt{\beta n}} +\frac{\norm{\hat U_{-i,1:\kr}\hat U_{-i,1:\kr}^T\epsilon_i}}{\lambda_{\kr}}}.
\end{align*}
This concludes the proof of Theorem \ref{thm:perturbation}.
\end{proof}

\section{Proof of Main Results in Section \ref{sec:spectral_clustering_mixture_model}}\label{sec:proof_main_thm}
In this section, we include proofs of Lemma \ref{lem:decomposition_simple}, Lemma \ref{lem:decomposition}, and Theorem \ref{thm:subg}.
The proofs of all other results of Section \ref{sec:spectral_clustering_mixture_model}  are included in the supplement \cite{supplement} due to page limit.

\subsection{Proof of Lemma \ref{lem:decomposition_simple} and Lemma \ref{lem:decomposition}}

\begin{proof}[Proof of Lemma \ref{lem:decomposition_simple}]
For simplicity, we denote $\hat U$ to be short for $\hat U_{1:r}$ throughout the proof.
From (\ref{eqn:spectral}), we know $\hat z_i$ must satisfy
\begin{align*}
\hat z_i = \argmin_{a\in[k]}\norm{\hat U\hat U^TX_i - \hat \theta_{a}},
\end{align*}
where $\{\hat \theta_a\}_{a\in[k]}$ satisfies (\ref{eqn:hat_theta_difference}) according to Proposition \ref{prop:poly}.
Hence, we have
\begin{align*}
\indic{\hat z_i \neq \phi(z_i^*)} &= \indic{\min_{a\in[k]:a\neq \phi(z^*_i)} \norm{\hat U\hat U^TX_i - \hat \theta_{ a }} \leq \norm{\hat U\hat U^TX_i - \hat \theta_{\phi(z^*_i)}}}.
\end{align*}
Consider a fixed $a\in[k]$ such that $a\neq \phi(z^*_i)$. Note that for any vectors $x,y,w$ of same dimension, if $\norm{x-y}\leq \norm{x-w}$, then we must have  $\norm{y-w}/2\leq \norm{x-w}$.
Hence, we have
\begin{align*}
 &\indic{ \norm{\hat U\hat U^TX_i - \hat \theta_{ a }} \leq \norm{\hat U\hat U^TX_i - \hat \theta_{\phi(z^*_i)}}} \\
 & = \indic{ \frac{1}{2}\norm{ \hat \theta_{\phi(z^*_i)}- \hat \theta_{ a }} \leq \norm{\hat U\hat U^TX_i - \hat \theta_{\phi(z^*_i)}}}\\
 &\leq  \indic{ \frac{1}{2}\norm{ \hat \theta_{\phi(z^*_i)}- \hat \theta_{ a }}   \leq \norm{\hat U\hat U^T\epsilon_i - \hat \theta_{\phi(z^*_i)}} +\norm{\hat U\hat U^T\theta^*_{z^*_i} - \hat \theta_{\phi(z^*_i)}}}\\
 &\leq  \indic{\norm{ \hat \theta_{\phi(z^*_i)}- \hat \theta_{ a }} -2\norm{\theta^*_{z^*_i} - \hat \theta_{\phi(z^*_i)}}  \leq 2\norm{\hat U\hat U^T\epsilon_i - \hat \theta_{\phi(z^*_i)}}},
\end{align*}
where we use the fact that $X_i = \theta^*_{z^*_i}+\epsilon_i$ and $\|  \hat U\hat U^T \theta^*_{z^*_i} - \hat \theta_{\phi(z^*_i)}\|\leq \|\theta^*_{z^*_i} - \hat \theta_{\phi(z^*_i)}\|$. 
Since $\hat \theta_{\phi(z^*_i)}  -\hat \theta_{ a } = \hat \theta_{\phi(z^*_i)} -\theta^*_{z^*_i} + \theta^*_{z^*_i} - \theta^*_{\phi^{-1}(a)} + \theta^*_{\phi^{-1}(a)}  -\hat \theta_{ a }$, we have
\begin{align}
\nonumber&\indic{ \norm{\hat U\hat U^TX_i - \hat \theta_{ a }} \leq \norm{\hat U\hat U^TX_i - \hat \theta_{\phi(z^*_i)}}} \\
\nonumber&\leq   \mathbb{I}\Big\{\norm{ \theta^*_{z^*_i} - \theta^*_{\phi^{-1}(a)} } - \norm{\hat \theta_{\phi(z^*_i)} -\theta^*_{z^*_i}} - \norm{ \theta^*_{\phi^{-1}(a)}  -\hat \theta_{ a }} \\
\nonumber &\quad\quad- 2 \norm{\theta^*_{z^*_i} - \hat \theta_{\phi(z^*_i)}} \leq 2\norm{\hat U\hat U^T\epsilon_i  } \Big\}\\
\nonumber&\leq   \indic{\norm{ \theta^*_{z^*_i} - \theta^*_{\phi^{-1}(a)} } -4\max_{b\in[k]} \norm{\theta^*_{b} - \hat \theta_{\phi(b)}} \leq 2\norm{\hat U\hat U^T\epsilon_i  } }\\
\label{eqn:lemma_1_proof_2}&\leq \indic{\br{1-\frac{4C_0\beta^{-0.5}kn^{-0.5}\norm{E}}{\Delta}}\Delta\leq 2\norm{\hat U\hat U^T\epsilon_i  } },
\end{align}
where in the last inequality, we use the fact that $\max_{b\in[k]} \normop{\theta^*_{b} - \hat \theta_{\phi(b)}} \leq C_0\beta^{-0.5}kn^{-0.5}\norm{E}$ from Proposition \ref{prop:poly} and $\min_{b,b'\in[k]:b\neq b'}\norm{\theta^*_b -\theta^*_{b'}}=\Delta$. Since the above display holds for each $a\in[k]$ that is not $\phi(z_i^*)$, we have
\begin{align*}
\indic{\hat z_i \neq \phi(z_i^*)} & \leq \indic{\br{1-\frac{4C_0\beta^{-0.5}kn^{-0.5}\norm{E}}{\Delta}}\Delta\leq 2\norm{\hat U\hat U^T\epsilon_i  } } \\
&= \indic{\br{1-4C_0\psi_0^{-1}}\Delta\leq 2\norm{\hat U\hat U^T\epsilon_i  } },
\end{align*}
where in the last inequality we use the definition of $\psi_0$ in (\ref{eqn:Delta_poly}).
\end{proof}

\begin{proof}[Proof of Lemma \ref{lem:decomposition}]
For simplicity,  throughout the proof we denote $\hat U$ and $\hat U_{-i}$ to be short for $\hat U_{1:\kr}$ and $\hat U_{-i,1:\kr}$, respectively.
We have the following decomposition for $\hat U\hat U^T\epsilon_i$,
\begin{align*}
\norm{\hat U\hat U^T\epsilon_i} 
&\leq  \norm{\hat U_{-i}\hat U_{-i}^T\epsilon_i}  + \fnorm{\hat U\hat U^T - \hat U_{-i}\hat U_{-i}^T} \norm{\epsilon_i}.
\end{align*}
Using the fact that $\norm{\epsilon_i}\leq \norm{E}$ and Theorem \ref{thm:perturbation}, after rearrangement, we have
\begin{align*}
\norm{\hat U\hat U^T\epsilon_i} &\leq\frac{128k\norm{E}}{\sqrt{n\beta}\rho_0} + \br{1+ \frac{128\norm{E}}{\rho_0\lambda_k}}\norm{\hat U_{-i}\hat U_{-i}^T\epsilon_i} \\
&=128\psi_0^{-1}\rho_0^{-1}\Delta +  \br{1+ \frac{128}{\rho_0^2}}\norm{\hat U_{-i}\hat U_{-i}^T\epsilon_i}.
\end{align*}
In Lemma \ref{lem:decomposition_simple} we establish (\ref{eqn:decomposition1}).  From there we have
\begin{align*}
\indic{\hat z_i \neq \phi(z_i^*)}  &\leq \indic{\br{1-C\psi_0^{-1}}\Delta\leq 256\psi_0^{-1}\rho_0^{-1}\Delta +  2\br{1+ \frac{128}{\rho_0^2}}\norm{\hat U_{-i}\hat U_{-i}^T\epsilon_i} }\\
&\leq \indic{\br{1-C'\br{\psi_0^{-1}+\rho_0^{-2}}}\Delta \leq 2\norm{\hat U_{-i}\hat U_{-i}^T\epsilon_i}},
\end{align*}
for some constant $C'>0$, where in the last inequality we use the assumption $\rho_0>16$ from (\ref{eqn:eigengap}). The upper bound on $\E \ell(\hat z,z^*)$ is an immediate consequence as $\E \ell(\hat z,z^*) = n^{-1} \sum_{i\in[n]} \E \indic{\hat z_i \neq \phi(z_i^*)} $.
\end{proof}

\subsection{Proofs of Theorem \ref{thm:subg}}
\begin{proof}[Proof of Theorem \ref{thm:subg}]
For simplicity,   we denote $\hat U_{-i}$ to be short for  $\hat U_{-i,1:\kr}$ throughout the proof.
Define $\psi:=\psi_1^{-1}+\rho_1^{-2}$. Then $\psi <2/C$.

Since $E$ is a random matrix with independent sub-Gaussian columns, we have
\begin{align}\label{eqn:gmm_f}
\pbr{\norm{E}\leq 8\sigma(\sqrt{n}+\sqrt{p})}\geq 1-e^{-n/2},
\end{align}
 by Lemma \ref{lem:sub_gaussian_operator}.
Denote $\mathf$ to be this event.  Under $\mathf$, 
as long as  $\psi_1,\rho_1\geq 128$,
we have both
 (\ref{eqn:Delta_poly}) and (\ref{eqn:eigengap}) hold. Let $\phi\in\Phi$ satisfy $\ell(\hat z,z^*) =n^{-1}\sum_{i\in[n]}\indic{\hat z_i \neq \phi(z^*_i)}$. Consider a fixed $i\in[n]$.
Then from Lemma \ref{lem:decomposition}, we have  
\begin{align*}
\indic{\hat z_i \neq \phi(z_i^*)}\indic{\mathf} &\leq\indic{\br{1-C_1\psi}\Delta \leq 2\norm{\hat U_{-i}\hat U_{-i}^T\epsilon_i}}\indic{\mathf}\\
&\leq  \indic{\br{1-C_1\psi}\Delta \leq 2\norm{\hat U_{-i}\hat U_{-i}^T\epsilon_i}},
\end{align*}
where $C_1>0$ is some constant that does not depend on $C$.
Then,
\begin{align}
\E \ell(\hat z,z^*)&\leq \E \indic{\mathf^\complement} + \E \ell(\hat z,z^*) \indic{\mathf}\nonumber\\
&\leq e^{-n/2} + n^{-1}\sum_{i\in[n]}\E  \indic{\br{1-C_1\psi}\Delta \leq 2\norm{\hat U_{-i}\hat U_{-i}^T\epsilon_i}}.\label{eqn:subg_proof_1}
\end{align}
Since $\epsilon_i\sim\text{SG}_p(\sigma^2)$ and it is independent of $\hat U_{-i}\hat U_{-i}^T$, we can apply concentration inequalities for $\normop{\hat U_{-i}\hat U_{-i}^T\epsilon_i}$ from Lemma \ref{lem:sub_gaussian_projection}.  Define $t=(1-C_2\psi)\Delta^2/(8\sigma^2)$ where $C_2=C_1 + 16$. 
Since $C_2$ does not depend on $C$, we can let $C>\max\{4C_2,128\}$ such that $1-C_2\psi >1/2$.
Then we have $k/t\leq 16k^2\sigma^2/\Delta^2\leq 16\psi_1^2$ where we use the fact that $\frac{\Delta}{k\sigma} >\psi_1^{-1}$  from (\ref{eqn:Delta_subg}) as $\beta \leq 1$. Then we have 
\begin{align*}
\sigma^2(\kr+2\sqrt{\kr t}+2t)&= 2\sigma^2t\br{\frac{1}{2}\frac{\kr}{t}+\sqrt{\frac{\kr}{t}}+1} \leq 2\sigma^2t \br{8\psi_1^2+4\psi_1+1} \leq 2\sigma^2t\br{1+8\psi_1}\\
&\leq (1-C_2\psi)\Delta^2/(8\sigma^2)\br{1+8\psi}\leq (1-C_1\psi)\Delta^2/(8\sigma^2),
\end{align*}
where we use  that $\psi_1<1/128$ and $\psi<1/64$ as we let $C>128$. Then from Lemma \ref{lem:sub_gaussian_projection}, we have
\begin{align*}
\E  \indic{\br{1-C_1\psi}\Delta \leq 2\norm{\hat U_{-i}\hat U_{-i}^T\epsilon_i}} \leq \ebr{-t} = \ebr{-(1-C_2\psi)\frac{\Delta^2}{8\sigma^2}}.
\end{align*}
\end{proof}

\section*{Acknowledgements} The authors are grateful to an anonymous Associate Editor and  anonymous referees for careful reading of
the manuscript and their valuable remarks and suggestions.

\begin{supplement}
\sname{Supplement A}\label{suppA}
\stitle{Supplement to ``Leave-one-out Singular Subspace Perturbation Analysis for Spectral Clustering''}
\slink[url]{url to be specified}
\sdescription{In the supplement \cite{supplement}, we first 
provide the proof of Theorem \ref{thm:perturbation_r} in Appendix \ref{sec:proof_perturbation_r}, followed by the proofs of 
results of Section \ref{sec:adaptive}
in Appendix \ref{sec:appendix_1}. The proof of Theorem \ref{thm:GMM} is given in Appendix \ref{sec:appendix_2}. 
The proofs of results of Section \ref{sec:lower} are given in Appendix \ref{sec:proof_lower}.
Auxiliary lemmas and propositions and their proofs are included in Appendix \ref{sec:auxiliary}.}
\end{supplement}

\bibliographystyle{imsart-number}
\bibliography{spectral}

\newpage
\thispagestyle{empty}
\setcounter{page}{1}
\begin{center}
\uppercase{\large Supplement to ``Leave-one-out Singular Subspace Perturbation Analysis for Spectral Clustering''}
\medskip

{BY Anderson Y. Zhang and Harrison H.~Zhou}
\medskip

{University of Pennsylvania and Yale University}
\end{center}

\appendix

\section{Proof of Theorem \ref{thm:perturbation_r}}\label{sec:proof_perturbation_r}
The proof idea is similar to that of Theorem \ref{thm:perturbation} but with  more involved calculation as $r$ is not necessarily $\kr$. Consider any $i\in[n]$. Define
\begin{align*}
\tilde\rho_{-i}:=\frac{\hat\lambda_{-i,r} - \hat\lambda_{-i,r+1}}{\norm{\br{I- \hat U_{-i,1:r} \hat U_{-i,1:r}^T} X_i}}.
\end{align*}
We need to verify $\tilde\rho_{-i}>2$ first in order to apply Theorem \ref{thm:general}. Recall the definition of $P_{-i}$ in (\ref{eqn:P_i_def}) and $E_{-i}$ in (\ref{eqn:E_i_def}). Let the SVD of $P_{-i}$ be
\begin{align*}
P_{-i}  =\sum_{j=1}^{p\wedge(n-1)} \lambda_{-i,j}u_{-i,j}v_{-i,j}^T,
\end{align*}
where
$\lambda_{-i,1}\geq \lambda_{-i,2}\geq \ldots \geq \lambda_{-i,p\wedge (n-1)}$. Denote $U_{-i,1:r}=(u_{-i,1},u_{-i,2},\ldots,u_{-i,r})\in\matho^{p\times r}$.
Then by Weyl's inequality, we have 
\begin{align}\label{eqn:hat_lambda_i_r_r_1}
|\hat\lambda_{-i,r} - \lambda_{-i,r}|, |\hat\lambda_{-i,r+1} - \lambda_{-i,r+1}|\leq \norm{E_{-i}}\leq \norm{E}.
\end{align}
Then the numerator
\begin{align}\label{eqn:hat_lambda_i_r_r_plus_def_1}
\hat\lambda_{-i,r} - \hat\lambda_{-i,r+1} \geq \lambda_{-i,r} - \lambda_{-i,r+1} -2\norm{E}.
\end{align}
In the following, we are going to connect $ \lambda_{-i,r} - \lambda_{-i,r+1}$ with $\lambda_{r}- \lambda_{r+1}$.

To bridge the gap between $\lambda_{-i,r}, \lambda_{-i,r+1}$ and $\lambda_{r}, \lambda_{r+1}$, define
\begin{align*}
\tilde P_{-i}:= (\theta^*_{z^*_1},\ldots,\theta^*_{z^*_{i-1}},U_{-i,1:r}U_{-i,1:r}^T \theta^*_{z^*_i},\theta^*_{z^*_{i+1}},\ldots, \theta^*_{z^*_n})\in\mathr^{p\times n}.
\end{align*}
Let $\tilde \lambda_{-i,1}\geq \tilde \lambda_{-i,2}\geq \ldots \geq \tilde \lambda_{-i,p\wedge n}$ be its singular values. Note that $U_{-i,1:r}U_{-i,1:r}^T\tilde P_{-i}$ is the best rank-$r$ approximation of $\tilde P_{-i}$. This is because for any rank-$r$ projection matrix $M\in\mathr^{p\times p}$ such that $M^2=M$, we have
\begin{align*}
\fnorm{\tilde P_{-i} - MM^T\tilde P_{-i}}^2 &=\fnorm{(I-MM^T) P_{-i}}^2 + \fnorm{(I-MM^T) U_{-i,1:r}U_{-i,1:r}^T\theta^*_{z^*_i}}^2\\
&\geq \fnorm{(I-U_{-i,1:r}U_{-i,1:r}^T) P_{-i}}^2  + 0\\
& = \fnorm{\tilde P_{-i} - U_{-i,1:r}U_{-i,1:r}^T\tilde P_{-i}}^2,
\end{align*} 
where we use the fact $U_{-i,1:r}U_{-i,1:r}^T P_{-i}$ is the best rank-$r$ approximation of $P_{-i}$. Hence, $\tspan(U_{-i,1:r})$ is exactly the leading $r$ left singular space of $\tilde P_{-i}$. It immediately implies:
\begin{itemize}
\item $\tilde \lambda_{-i,j}=\lambda_{-i,j}$ for any $j\geq r+1$, including
\begin{align}\label{eqn:proof_thm7_11}
\tilde \lambda_{-i,r+1}=\lambda_{-i,r+1}.
\end{align}
\item %
Since $U_{-i,1:r}U_{-i,1:r}^T\tilde P_{-i}$ and $U_{-i,1:r}U_{-i,1:r}^T P_{-i}$ only differ by one column where the latter one can be seen as the leave-one-out counterpart of the former one, using the same argument as in (\ref{eqn:thm_2_proof_1}), we have
\begin{align}\label{eqn:hat_lambda_i_r_r_2}
\lambda^2_{-i,r}\geq \br{1-\frac{k}{\beta n}}\tilde\lambda^2_{-i,r}.
\end{align}
\end{itemize}
Then from (\ref{eqn:hat_lambda_i_r_r_plus_def_1}), we have
\begin{align}\label{eqn:hat_lambda_i_r_r_plus_def_2}
\hat\lambda_{-i,r} - \hat\lambda_{-i,r+1} \geq \sqrt{1-\frac{k}{\beta n}}\tilde\lambda_{-i,r} - \tilde \lambda_{-i,r+1} -2\norm{E}.
\end{align}

For the difference between $\tilde\lambda_{-i,r},\tilde\lambda_{-i,r+1}$ and $\lambda_{r},\lambda_{r+1}$, we use the Weyl's inequality again:
\begin{align*}
\max_{j\in[k]}\abs{\tilde \lambda_{-i,j} - \lambda_j}\leq \norm{P- \tilde P_{-i}} = \norm{\theta^*_{z^*_i} - U_{-i,1:r}U_{-i,1:r}^T \theta^*_{z^*_i}}.
\end{align*}
In the proof of Theorem \ref{thm:perturbation}, we show $u_{-i,j}\in\text{span}(\{\theta^*_a\}_{a\in[k]})$ for each $j\in[\kr]$. Then
\begin{align*}
\norm{\theta^*_{z^*_i} - U_{-i,1:r}U_{-i,1:r}^T \theta^*_{z^*_i}} &= \norm{\br{u_{-i,r+1},\ldots,u_{-i,\kr}}\br{u_{-i,r+1},\ldots,u_{-i,\kr}}^T\theta^*_{z^*_i}} \\
&=\sqrt{\sum_{a\in[\kr]: a\geq r+1}\br{u_{-i,a}^T\theta^*_{z^*_i}}^2}.
\end{align*}
For any $a\in[\kr]$ such $a\geq r+1$, we have
\begin{align*}
\br{u_{-i,a}^T\theta^*_{z^*_i}}^2 &\leq \frac{1}{\abs{\cbr{j\in[n]:z^*_j=z^*_i}}-1}\sum_{j\in[n]:j\neq i,z^*_j=z^*_i}\br{u_{-i,a}^T\theta^*_{z^*_{j}}}^2 \leq \frac{1}{\frac{\beta n}{k}-1} (u_{-i,a}^TP_{-i})^2  \\
&\leq \frac{\lambda_{-i,a}^2}{\frac{\beta n}{k}-1} \leq  \frac{\lambda_{-i,r+1}^2}{\frac{\beta n}{k}-1}.
\end{align*}
Hence, we obtain $\normop{\theta^*_{z^*_i} - U_{-i,1:r}U_{-i,1:r}^T \theta^*_{z^*_i}}\leq \sqrt{\kr}\lambda_{-i,a}/\sqrt{\beta n/k-1}$ and consequently,
\begin{align}\label{eqn:proof_thm7_1}
\max_{j\in[k]}\abs{\tilde \lambda_{-i,j} - \lambda_j} \leq \frac{\sqrt{\kr} \lambda_{-i,r+1}}{\sqrt{\frac{\beta n}{k}-1}}.
\end{align}
Then together with (\ref{eqn:proof_thm7_11}), we have $|\lambda_{-i,r+1} -\lambda_{r+1}|\leq \sqrt{\kr}\lambda_{-i,r+1} /\sqrt{\beta n/k-1}$ and hence
\begin{align}\label{eqn:proof_thm7_3}
\lambda_{-i,r+1} \leq \frac{\lambda_{r+1}}{1-\frac{\sqrt{\kr}}{\sqrt{\frac{\beta n}{k}-1}}}.
\end{align}
Denote $d:=\beta n/k$. With (\ref{eqn:hat_lambda_i_r_r_plus_def_2}), we have
\begin{align}
\nonumber\hat\lambda_{-i,r} - \hat\lambda_{-i,r+1} &\geq \sqrt{\frac{d-1}{d}}\br{\lambda_r -  \frac{ \lambda_{-i,r+1}}{\sqrt{d-1}}} -\br{\lambda_{r+1} +  \frac{ \lambda_{-i,r+1}}{\sqrt{d-1}}} - 2\norm{E}\\
\nonumber&\geq \sqrt{\frac{d-1}{d}}\lambda_r - \lambda_{r+1}\br{1+\br{\frac{1}{\sqrt{d}} + \frac{1}{\sqrt{d-1}}}\frac{1}{1-\frac{\sqrt{\kr}}{\sqrt{d-1}}}} -2\norm{E}\\
\nonumber&\geq \sqrt{\frac{d-1}{d}} \br{\lambda_r - \lambda_{r+1}  -  \frac{4}{\sqrt{d}}  \lambda_{r+1}} -2\norm{E}\\
&\geq \frac{3}{4} \br{\lambda_r - \lambda_{r+1}  -  \frac{4}{\sqrt{d}}  \lambda_{r+1}} -2\norm{E},\label{eqn:hat_lambda_i_r_r_plus_def_3}
\end{align}
where in the last two inequalities we use the assumption that $d/k\geq 10$. 
As a consequence, we have
\begin{align*}
\tilde\rho_{-i}\geq \frac{\hat\lambda_{-i,r} - \hat\lambda_{-i,r+1}}{\norm{\br{I- \hat U_{-i,1:r} \hat U_{-i,1:r}^T} X_i}}\geq \frac{\frac{3}{4} \br{\lambda_r - \lambda_{r+1}  -  \frac{4}{\sqrt{d}}  \lambda_{r+1}} -2\norm{E}}{\norm{\br{I- \hat U_{-i,1:r} \hat U_{-i,1:r}^T} X_i}}.
\end{align*}

Next, we are going to simplify the denominator of the above display. Using the orthogonality of the singular vectors, we have
\begin{align*}
&\norm{\br{I - \hat U_{-i,1:r}\hat U_{-i,1:r}^T}\theta^*_{z^*_i}}\\
 & \leq \norm{\br{I - \hat U_{-i,1:\kr}\hat U_{-i,1:\kr}^T}\theta^*_{z^*_i}}  + \norm{\br{\hat u_{-i,r+1},\ldots, \hat u_{-i,\kr}}\br{\hat u_{-i,r+1},\ldots, \hat u_{-i,\kr}}^T\theta^*_{z^*_i}}\\
&=  \norm{\br{I - \hat U_{-i,1:\kr}\hat U_{-i,1:\kr}^T}\theta^*_{z^*_i}} + \sqrt{\sum_{j=r+1}^{\kr}\br{\hat u_{-i,j}^T\theta^*_{z^*_i}}^2}\\
&\leq \frac{3\sqrt{\kr}\norm{E}}{\sqrt{\frac{\beta n}{k}-1}} + \sqrt{\sum_{j=r+1}^{\kr}\br{\frac{\hat\lambda_{-i,j}}{\sqrt{\frac{\beta n}{k}-1}} + \frac{\norm{E}}{\sqrt{\frac{\beta n}{k}-1}}}^2}\\
&\leq \frac{3\sqrt{\kr}\norm{E}}{\sqrt{\frac{\beta n}{k}-1}} + \sqrt{\kr}\br{\frac{\hat\lambda_{-i,r+1}}{\sqrt{\frac{\beta n}{k}-1}} + \frac{\norm{E}}{\sqrt{\frac{\beta n}{k}-1}}},
\end{align*}
where the second to the  inequality is due to (\ref{eqn:perturbation_proof_2}) and (\ref{eqn:perturbation_proof_5}). By (\ref{eqn:proof_thm7_3}) and the Weyl's inequality, we have
\begin{align*}
\hat \lambda_{-i,r+1}\leq  \lambda_{-i,r+1} + \norm{E} \leq  \frac{1}{1-\frac{\sqrt{\kr}}{\sqrt{\frac{\beta n}{k}-1}} }\lambda_{r+1} + \norm{E}.
\end{align*}
Then, with the assumption $\beta n/k^2\geq 10$, we have
\begin{align*}
\norm{\br{I - \hat U_{-i,1:r}\hat U_{-i,1:r}^T}\theta^*_{z^*_i}} &\leq  \frac{3\sqrt{\kr}\norm{E}}{\sqrt{\frac{\beta n}{k}-1}} + \sqrt{\kr}\br{\frac{\lambda_{r+1}}{\sqrt{\frac{\beta n}{k}-1}-\sqrt{\kr}} + \frac{2\norm{E}}{\sqrt{\frac{\beta n}{k}-1}}}\\
&\leq \frac{\sqrt{k\kr}}{\sqrt{\beta n}}(6\norm{E}+2\lambda_{r+1}).
\end{align*}
Hence,
\begin{align*}
\norm{\br{I- \hat U_{-i,1:r} \hat U_{-i,1:r}^T} X_i}&\leq \norm{\br{I- \hat U_{-i,1:r} \hat U_{-i,1:r}^T} \theta^*_{z^*_i}} + \norm{\br{I- \hat U_{-i,1:r} \hat U_{-i,1:r}^T} \epsilon_i}\\
&\leq \frac{\sqrt{k\kr}}{\sqrt{\beta n}}(6\norm{E}+2\lambda_{r+1}) +\norm{E}.
\end{align*}
As a result,
\begin{align*}
\tilde\rho_{-i}\geq  \frac{\frac{3}{4} \br{\lambda_r - \lambda_{r+1}  -  \frac{4}{\sqrt{\beta n/k}}  \lambda_{r+1}} -2\norm{E}}{\frac{\sqrt{k\kr}}{\sqrt{\beta n}}(6\norm{E}+2\lambda_{r+1}) +\norm{E}} \geq \frac{\tilde\rho_0}{8}>2,
\end{align*}
under the assumption that $\beta n/(k^2)\geq 10$ and (\ref{eqn:perturbation_r}).

The remaining part of the proof is to study $\{\hat u_{-i,a}^T X_i\}_{a\in[\rr]}$ and then apply Theorem \ref{thm:general}. Following the exact argument as in the proof of Theorem \ref{thm:perturbation}, we have 
\begin{align*}
\sqrt{\sum_{a\in\rr} \br{\frac{\hat u_{-i,a}^TX_i}{\hat\lambda_{-i,a}}}^2} &\leq \frac{\sqrt{\rr}}{\sqrt{\frac{\beta n}{k}-1}} +  \frac{1}{\hat\lambda_{-i,\rr}}  \frac{\norm{E}\sqrt{\rr}}{\sqrt{\frac{\beta n}{k}-1}}  +\frac{1}{\hat\lambda_{-i,\rr}}\norm{\hat U_{-i,1:\rr}\hat U_{-i,1:\rr}^T\epsilon_i}.
\end{align*}
Under the assumption that $\beta n/(k^2)\geq 10$ and (\ref{eqn:perturbation_r}), (\ref{eqn:hat_lambda_i_r_r_plus_def_3}) is lower bounded by $\lambda_r/2$. This also implies $\hat\lambda_{-i,r}\geq \lambda_r/2$. Then a direct application of Theorem \ref{thm:general} leads to
\begin{align*}
\fnorm{{\hat U_{1:r} \hat U_{1:r}^T - \hat U_{-i,1:r}\hat U_{-i,1:r}^T}} &\leq \frac{4\sqrt{2}}{\tilde\rho_{-i}}\br{\frac{\sqrt{r}}{\sqrt{\beta n/k-1}}+\frac{1}{\hat \lambda_{-i,r}}\br{\frac{\sqrt{r}\norm{E}}{\sqrt{\beta n/k-1}}+\norm{\hat U_{-i,1:r}\hat U_{-i,1:r}^T\epsilon_i}}}\\
&\leq  \frac{128}{\tilde\rho_0}\br{\frac{\sqrt{kr}}{\sqrt{\beta n}} +\frac{\norm{\hat U_{-i,1:r}\hat U_{-i,1:r}^T\epsilon_i}}{\lambda_{r}}}.
\end{align*}

\section{Proofs of Results in Section \ref{sec:adaptive}}
\label{sec:appendix_1}

Before presenting the proof of Lemma \ref{lem:thresholding}, we first show $\hat r$ defined in (\ref{eqn:hat_r_def}) always exists. In addition, since $\hat r\in[k]$ is a random variable, we are going to associate it with some deterministic set in $[k]$.  Recall $\lambda_1\geq \lambda_2\geq \ldots\geq \lambda_{p\wedge n}$ are singular values of the signal matrix $P$ and $\kr$ is the its rank. Let its SVD be $P=\sum_{i\in[p\wedge n]}\lambda_i u_iv_i^T$ with $\{u_j\}_{j\in[p\wedge n]}\in\mathr^p$ being its left singular vectors.
\begin{lemma}\label{lem:hat_r_exist}
Under the same conditions as stated in  Lemma \ref{lem:thresholding}, $\hat r$ always exists. Furthermore, we have $\hat r\in \mathcal{R}$ where
 \begin{align}\label{eqn:proof_thresholding_1}
 \mathcal{R}:=\cbr{a\in[k]: \lambda_a -\lambda_{a+1} \geq  (\tilde \rho -2)\norm{E} \text{ and }\lambda_{a+1}\leq (k\tilde \rho+1)\norm{E}}.
 \end{align}
\end{lemma}
\begin{proof}
The existence of $\hat r$ can be proved by contradiction. If $\hat r$ does not exist, it means that $\{a\in[k]:\hat \lambda_a -\hat \lambda_{a+1}\geq T\}$ is empty, which implies $\hat\lambda_1 < \hat\lambda_{k+1} +k T =  \hat\lambda_{k+1} +k\tilde \rho \|E\|$. By Weyl's inequality, we have $|{\hat \lambda_a - \lambda_a}|\leq \norm{E}$ for all singular values of $X$ and $P$. Then  we have $ \lambda_1< (k\tilde \rho +1)\|E\| $. On the other hand,  we have 
\begin{align*}
\lambda_1^2 &= \max_{w\in\mathr^p:\|w\|=1}\norm{w^T P}^2 \geq \max_{a,b\in[k]:a\neq b}  \max_{w\in\mathr^p:\|w\|=1} \frac{\beta n}{k}\br{ \norm{w^T\theta^*_a}^2 + \norm{w^T\theta^*_b}^2}\\
& \geq  \max_{a,b\in[k]:a\neq b}  \max_{w\in\mathr^p:\|w\|=1}  \frac{\beta n}{2k} \norm{w^T\theta^*_a - w^T\theta^*_b}^2 = \frac{\beta n}{2k}\Delta^2,
\end{align*}
where the first inequality is due to  the mixture model structure in $P$ and the second inequality is due to $2(x_1+x_2)^2 \geq (x_1-x_2)^2$ for any two scalars $x_1,x_2$. Then we have $\lambda_1\geq \sqrt{\beta n/(2k)}\Delta =(\tilde\psi_0/\sqrt{2}) k^{1.5}\norm{E}$ by (\ref{eqn:thresholding_tilde_psi}). Since $\tilde \rho<\tilde \psi_0/64$ is assumed, we have   $(k\tilde \rho +1)\|E\|  < (\tilde\psi_0/\sqrt{2}) k^{1.5}\norm{E}$, which is a contradiction.

To prove the second statement, note that
we have $\hat \lambda_{\hat r} -\hat \lambda_{\hat r+1}\geq \tilde \rho \norm{E}$ and $\hat \lambda_{\hat r+1} \leq k\tilde \rho \norm{E}$.
 Since $|{\hat \lambda_a - \lambda_a}|\leq \norm{E}$ for all singular values of $X$ and $P$,  we have $\lambda_{\hat r} -\lambda_{\hat r + 1}\geq (\tilde \rho -2)\norm{E}$ and $\lambda_{\hat r+1}\leq (k\tilde \rho + 1)\norm{E}$. 
 Hence, $\hat r\in\mathcal{R}$.
\end{proof}

 \begin{proof}[Proof of Lemma \ref{lem:thresholding}]
 From Lemma \ref{lem:hat_r_exist}, we know $\hat r$ exists and $\hat r\in\mathcal{R}$.
Consider an arbitrary $r\in\mathcal{R}$ and define $\hat U_{1:r} := (\hat u_1,\ldots,\hat u_r)\in\mathr^{p\times r}$. %
Perform $k$-means on the columns of $\hat U_{1:r}\hat U_{1:r}^TX$ and let the output be
\begin{align*}
 \br{\check z(r),\cbr{\check \theta_j(r)}_{j=1}^k} = \argmin_{ z\in [k]^n , \cbr{ \theta_j}_{j=1}^{k} \in \mathr^{p}} \sum_{i\in[n]} \norm{\hat U_{1:r}\hat U_{1:r}^TX- \theta_{z_i}}^2.
\end{align*}
In the following, we are going to establish statistical properties for $\check z(r)$ and eventually obtain a desired upper bound for $\ell(\check z(r),z^*)$. Since performing $k$-means on the columns of $\hat U_{1:r}^TX$  is equivalent to $k$-means on the columns of $\hat U_{1:r}\hat U_{1:r}^TX$, and since $\hat r\in \mathcal{R}$, we have $\tilde z = \check z(\hat r)$ and thus  the desired upper bound also holds for $\ell(\tilde z,z^*)$.

In the rest of the proof we are going to analyze $\check z(r)$ for any $r\in\mathcal{R}$. For simplicity, we use the notation $\check z, \{\check \theta_j\}_{j\in[n]}$ instead of $\check z(r), \{\check \theta_j(r)\}_{j\in[n]}$. The remaining proof can be decomposed into several parts.

~\\
\emph{(Preliminary Results for $\check z, \{\check \theta_j\}_{j\in[n]}$).} We are going to use Proposition \ref{prop:poly} to have some preliminary results. Define $U_{1:r}:=(u_1,\ldots,u_r)$ and $U_{(r+1):k}:=(u_{r+1},\ldots, u_k)$. Instead of the decomposition (\ref{eqn:mixture_model}), we can write
\begin{align*}
X_i = U_{1:r}U_{1:r}^T \theta^*_{z^*_i} + U_{(r+1):k}U_{(r+1):k}^T\theta^*_{z^*_i}  + \epsilon_i =U_{1:r}U_{1:r}^T \theta^*_{z^*_i} + \check \epsilon_i,
\end{align*}
where $\check \epsilon_i := U_{(r+1):k}U_{(r+1):k}^T\theta^*_{z^*_i}  + \epsilon_i $. In this way, we have a new mixture model with the centers being $\{U_{1:r}U_{1:r}^T\theta^*_a\}_{a\in[k]}$ and the additive noises being $\{\check\epsilon_i\}$.
Define $\check E:= (\check \epsilon_1,\ldots,\check \epsilon_n )$. Then
\begin{align}
\nonumber\norm{\check E} &\leq \norm{E} + \norm{\br{U_{(r+1):k}U_{(r+1):k}^T\theta^*_{z^*_1}, \ldots, U_{(r+1):k}U_{(r+1):k}^T\theta^*_{z^*_n}}} \\
\nonumber& = \norm{E}  + \norm{U_{(r+1):k}U_{(r+1):k}^T P} = \norm{E} + \lambda_{r+1} \\
&\leq  (k\tilde \rho +2)\norm{E}.\label{eqn:proof_thresholding_4}
\end{align}

The separation among the new centers is no longer $\Delta$. Define $$\check\Delta := \min_{a,b\in[k]:a\neq b}\norm{U_{1:r}U_{1:r}^T\theta^*_a - U_{1:r}U_{1:r}^T\theta^*_b}.$$ For any $a,b\in[k]$, $U_{1:r}U_{1:r}^T\theta^*_a - U_{1:r}U_{1:r}^T\theta^*_b = (\theta^*_a -\theta^*_b) - U_{(r+1):k}U_{(r+1):k}^T \theta^*_a + U_{(r+1):k}U_{(r+1):k}^T \theta^*_b$. Also,
\begin{align}
\nonumber \max_{a\in[k]}\norm{U_{(r+1):k}U_{(r+1):k}^T \theta^*_a} &= \max_{a\in[k]}  \sqrt{\frac{\sum_{i\in[n]:z_i^*=a}\norm{U_{(r+1):k}U_{(r+1):k}^T \theta^*_a}^2}{\abs{\cbr{i\in[n]:z_i^*=a}}}} \leq  \frac{\fnorm{U_{(r+1):k}U_{(r+1):k}^T P}}{\sqrt{\beta n/k}}  \\
&\leq \frac{2\sqrt{k}\lambda_{r+1}}{\sqrt{\beta n/k}}\leq \frac{\sqrt{k}(k\tilde \rho +1)\norm{E}}{\sqrt{\beta n/k}}\label{eqn:proof_thresholding_3}.
\end{align}
Hence, we have 
\begin{align}\label{eqn:thresholding_check_Delta}
\check\Delta &\geq \min_{a,b\in[k]:a\neq b}\norm{\theta^*_a - \theta^*_b} - 2\max_{a\in[k]}\norm{U_{(r+1):k}U_{(r+1):k}^T \theta^*_a} 
 \geq   \Delta - \frac{2\sqrt{k}(k\tilde \rho +1)\norm{E}}{\sqrt{\beta n/k}}.
\end{align}
 Then from Proposition \ref{prop:poly}, as long as (which will be verified later)
 \begin{align}\label{eqn:proof_thresholding_2}
\check \psi_0:= \frac{\check\Delta}{\beta^{-0.5}kn^{-0.5} \norm{\check E}} \geq 16,
 \end{align}
 we have
 \begin{align*}
 \ell(\check z,z^*) = \frac{1}{n}\abs{i\in[n]:\check z_i\neq \phi(z^*_i)} \leq \frac{C_0k\norm{\check E}^2}{n\check\Delta^2} ,
 \end{align*}
 and
 \begin{align*}
 \max_{a\in[k]}\norm{\check \theta_{\phi(z)} - U_{1:r}U_{1:r}^T\theta^*_a}  \leq C_0\beta^{-0.5}kn^{-0.5}\norm{\check E} .
 \end{align*}
 where $C_0=128$.
 
 ~\\
\emph{(Entrywise Decomposition for $\check z$).} Next, we are going to have an entrywise decomposition for $\indic{\hat z_i \neq \phi(z^*_i)}$ that is analogous to that of Lemma \ref{lem:decomposition}. When (\ref{eqn:proof_thresholding_2}) is satisfied, from Lemma \ref{lem:decomposition_simple}, we have
\begin{align*}
\indic{\check z_i \neq \phi(z^*_i)} \leq \indic{\br{1-C_0\check\psi_0^{-1}}\check\Delta \leq 2\norm{\hat U_{1:r}\hat U_{1:r}^T\check\epsilon_i}}.
\end{align*}
By the definition of $\check\epsilon_i$ and (\ref{eqn:proof_thresholding_3}), we have
\begin{align*}
\norm{\hat U_{1:r}\hat U_{1:r}^T\check\epsilon_i} &\leq \norm{\hat U_{1:r}\hat U_{1:r}^T\epsilon_i} + \norm{\hat U_{1:r}\hat U_{1:r}^TU_{(r+1):k}U_{(r+1):k}^T\theta^*_{z^*_i}}  \\
&\leq  \norm{\hat U_{1:r}\hat U_{1:r}^T\epsilon_i} + \norm{U_{(r+1):k}U_{(r+1):k}^T\theta^*_{z^*_i}} \\
&\leq   \norm{\hat U_{1:r}\hat U_{1:r}^T\epsilon_i} + \frac{\sqrt{k}(k\tilde \rho +1)\norm{E}}{\sqrt{\beta n/k}}.
\end{align*}
Then, we have
\begin{align*}
\indic{\check z_i \neq \phi(z^*_i)} &\leq \indic{\br{1-C_0\check\psi_0^{-1}}\check\Delta \leq 2\br{ \norm{\hat U_{1:r}\hat U_{1:r}^T\epsilon_i} + \frac{\sqrt{k}(k\tilde \rho +1)\norm{E}}{\sqrt{\beta n/k}}}}\\
& =  \indic{\br{1-C_0\check\psi_0^{-1} -  \frac{2\sqrt{k}(k\tilde \rho +1)\norm{E}}{\sqrt{\beta n/k}\check\Delta}}\check\Delta \leq 2 \norm{\hat U_{1:r}\hat U_{1:r}^T\epsilon_i}}.
\end{align*}

From (\ref{eqn:proof_thresholding_1}), under the assumption that $\tilde \rho >4$ and $\beta n/k^4>400$, we have $\tilde\rho_0$ defined as in (\ref{eqn:perturbation_r}) to satisfy
\begin{align*}
\tilde\rho_0\geq \frac{(\tilde \rho -1)\norm{E}}{\max\cbr{\norm{E}, \sqrt{\frac{k^2}{\beta n}}(k\tilde \rho +1)\norm{E} }}\geq 2.
\end{align*}
Then  Theorem \ref{thm:perturbation_r} can be applied, with which we have
\begin{align*}
\fnorm{{\hat U_{1:r} \hat U_{1:r}^T - \hat U_{-i,1:r}\hat U_{-i,1:r}^T}}  \leq  \frac{256\sqrt{rk}}{\sqrt{n \beta }} +\frac{256\norm{\hat U_{-i,1:r}\hat U_{-i,1:r}^T\epsilon_i}}{\lambda_{r}}.
\end{align*}
 Then following the proof of Lemma \ref{lem:decomposition}, we have 
\begin{align*}
&\indic{\check z_i \neq \phi(z^*_i)} \\
&\leq \indic{\br{1-C_0\check\psi_0^{-1}-\frac{2\sqrt{k}(k\tilde \rho +1)\norm{E}}{\sqrt{\beta n/k}\check\Delta}}\check\Delta \leq 2\br{\norm{\hat U_{-i,1:r}\hat U_{-i,1:r}^T\epsilon_i} + \fnorm{{\hat U_{1:r} \hat U_{1:r}^T - \hat U_{-i,1:r}\hat U_{-i,1:r}^T}}   \norm{E}} } \\
&\leq \indic{\br{1-C_0\check\psi_0^{-1}-\frac{2\sqrt{k}(k\tilde \rho +1)\norm{E}}{\sqrt{\beta n/k}\check\Delta}}\check\Delta \leq 2\br{ \frac{256\sqrt{rk}\norm{E}}{\sqrt{n \beta }} + \br{1+\frac{256\norm{E}}{\lambda_r}}\norm{\hat U_{-i,1:r}\hat U_{-i,1:r}^T\epsilon_i} }} \\
&\leq  \indic{\br{1-C_0\check\psi_0^{-1}-\frac{2\sqrt{k}(k\tilde \rho +257)\norm{E}}{\sqrt{\beta n/k}\check\Delta}}\check\Delta \leq 2 \br{1+\frac{256\norm{E}}{\lambda_r}}\norm{\hat U_{-i,1:r}\hat U_{-i,1:r}^T\epsilon_i} } \\
&\leq  \indic{\br{1-C_0\check\psi_0^{-1}-\frac{2\sqrt{k}(k\tilde \rho +257)\norm{E}}{\sqrt{\beta n/k}\check\Delta}}\check\Delta \leq 2 \br{1+\frac{256}{\tilde \rho -2}}\norm{\hat U_{-i,1:r}\hat U_{-i,1:r}^T\epsilon_i} },
\end{align*}
where in the last inequality we use $\lambda_r \geq (\tilde \rho -2) \norm{E}>0$ (as long as $\tilde \rho>2$) from (\ref{eqn:proof_thresholding_1}). 

The last step of the proof is to simplify the above display using $\Delta$ instead of $\check\Delta$.
Then, under the assumption that $\tilde \rho > 256$, we have $(1+256/(\tilde\rho-2))^{-1}\leq (1-512/\tilde \rho)$. Recall the definition of $\tilde\psi_0$ in (\ref{eqn:thresholding_tilde_psi}). Under the assumption that %
$\tilde \rho \leq \tilde\psi_0/64$, we have 
\begin{align}\label{eqn:proof_thresholding_5}
\check\Delta\geq \Delta\br{1-\frac{4\beta^{-0.5}k^2n^{-0.5}\tilde \rho\norm{E}}{\Delta}} = \Delta\br{1-\frac{4\tilde \rho}{\tilde\psi_0}} \geq \frac{\Delta}{2},
\end{align}
according to (\ref{eqn:thresholding_check_Delta}). Then together with (\ref{eqn:proof_thresholding_4}), we can verify (\ref{eqn:proof_thresholding_2}) holds due to
\begin{align*}
\check\psi_0 \geq \frac{\Delta/2}{\beta^{-0.5}kn^{-0.5}(k\tilde \rho + 2)\norm{E}} \geq \frac{\Delta}{4\beta^{-0.5}k^2n^{-0.5}\tilde\rho\norm{E}} =\frac{\tilde\psi_0}{4\tilde \rho} \geq16.
\end{align*}
Rearranging all the terms with the help of (\ref{eqn:proof_thresholding_5}), we can simplify $\indic{\check z_i \neq \phi(z^*_i)}$ into
\begin{align*}
&\indic{\check z_i \neq \phi(z^*_i)} \\
&\leq \indic{\br{1-4C_0\tilde \rho \tilde\psi_0 -\frac{4\beta^{-0.5}k^{2}n^{-0.5}\tilde\rho \norm{E}}{\Delta/2}}\br{1-\frac{256}{\tilde\rho}}\br{1-\frac{4\tilde\rho}{\tilde\psi_0}}\Delta\leq 2 \norm{\hat U_{-i,1:r}\hat U_{-i,1:r}^T\epsilon_i}}\\
&\leq \indic{\br{1-5C_0 \tilde\rho\tilde \psi_0^{-1} -256\tilde\rho^{-1}}\Delta\leq 2 \norm{\hat U_{-i,1:r}\hat U_{-i,1:r}^T\epsilon_i}}.
\end{align*}
\end{proof}

\begin{proof}[Proof of Theorem \ref{thm:GMM_2}]
Recall the definition of $\mathcal{F}$ in (\ref{eqn:gmm_f}). Then if $\mathcal{F}$ holds, by appropriate choices of $C_1,C_2$, we can verify the assumptions needed in Lemma \ref{lem:thresholding} hold, which lead to
\begin{align*}
\indic{\tilde z_i\neq \phi(z^*_i)}\indic{\mathcal{F}} \leq \indic{\br{1-C''(\rho_2\psi_2^{-1}+\rho_2^{-1})}\Delta \leq 2 \norm{\hat U_{-i,1:\hat r}\hat U_{-i,1:\hat r}^T\epsilon_i}}\indic{\mathcal{F}},
\end{align*}
for some constant $C''>0$. Though $\hat r$ is random, the proof of Lemma \ref{lem:thresholding} shows that $\hat r\in\mathcal{R}\subset [k]$ where $\mathcal{R}$ is defined in (\ref{eqn:proof_thresholding_1}). Note that for any $r\in[k]$, we can follow the proof of Theorem \ref{thm:subg} to show
\begin{align*}
\E\indic{\br{1-C''(\rho_2\psi_2^{-1}+\rho_2^{-1})}\Delta \leq 2 \norm{\hat U_{-i,1: r}\hat U_{-i,1: r}^T\epsilon_i}}\leq \ebr{-(1-C'''(\rho_2\psi_2^{-1}+\rho_2^{-1}))\frac{\Delta^2}{8\sigma^2}},
\end{align*}
for some constant $C'''>0$. Hence, the same upper bound holds for $\E \mathbb{I}\{(1-C''(\rho_2\psi_2^{-1}+\rho_2^{-1}))\Delta \leq 2 \normop{\hat U_{-i,1:\hat r}\hat U_{-i,1:\hat r}^T\epsilon_i}\}$.
The rest of the proof follows that of Theorem \ref{thm:subg} and is omitted here.
\end{proof}

\section{Proof of Theorem \ref{thm:GMM}}\label{sec:appendix_2}

Define $\mathf = \cbr{\norm{E}\leq \sqrt{2}(\sqrt{n} + \sqrt{p})\sigma}$. Then by Lemma B.1 of \cite{loffler2019optimality}, we have $\pbr{\mathf}\geq 1-e^{-0.08n}$. Then under the event $\mathf$, the assumption (\ref{eqn:Delta_GMM}) implies (\ref{eqn:Delta_poly}) holds, and hence (\ref{eqn:l_poly}) and (\ref{eqn:hat_theta_difference}) hold. For simplicity, and without loss of generality, we can let $\phi$ in (\ref{eqn:l_poly})-(\ref{eqn:hat_theta_difference}) to be the identity, and we get
\begin{align*}
 \ell(\hat z,z^*) =\frac{1}{n}|\{i\in[n]:\hat z_i \neq z^*_i\}| \leq \frac{C_0k\br{1+\sqrt{\frac{p}{n}}}^2\sigma^2}{\Delta^2},
\end{align*}
and
\begin{align*}
\max_{a\in[k]} \norm{\hat\theta_{a} - \theta_a^*}\leq C_0\beta^{-0.5}k\br{1+\sqrt{\frac{p}{n}}}\sigma,
\end{align*}
where $C_0>0$ is some constant.

Denote $\hat P = \hat U_{1:k} \hat U_{1:k}^TX$ and let $\hat P_{\cdot,i}$ be its $i$th column so that $\hat P_{\cdot,i} = \hat U_{1:k}\hat U_{1:k}^TX_i$. We define $r\in[k]$ as (with $\lambda_{k+1}:=0$)
\begin{align}
r = \max\cbr{ j \in [k]: \lambda_j  - \lambda_{j+1} \geq \tau \sqrt{n+p}\sigma},\label{eqn:r_def}
\end{align}
for a sequence  $\tau \rightarrow \infty $ to be determined later.  
We note that if  $\Delta / (k^\frac{3}{2}\tau \beta^\frac{1}{2} \br{1+p/n}^\frac{1}{2}\sigma)\rightarrow\infty$, 
the set $\cbr{ j \in [k]: \lambda_j  - \lambda_{j+1} \geq \tau \sqrt{n+p}\sigma}$ is not empty. Otherwise, this would imply $\lambda_1 \leq k\tau \sqrt{n+p}\sigma$ which would contradict with the fact $\lambda_1\geq \sqrt{\beta n/k} \Delta/(2\sigma)$ (see Proposition A.1 of \cite{loffler2019optimality}). By the definition of $r$ in (\ref{eqn:r_def}), we immediately have
\begin{align}\label{eqn:proof_gmm_4}
&\lambda_r - \lambda_{r+1}\geq \tau\sqrt{n+p}\sigma,\\
\text{and }&\lambda_{r+1}\leq k\tau \sqrt{n+p}\sigma.\label{eqn:proof_gmm_5}
\end{align}

We split $\hat U_{1:k}$ into $(\hat U_{1:r} ,\hat U_{(r+1):k})$ where $\hat U_{1:r}:=\br{\hat u_1,\ldots, \hat u_{r}}$ and $\hat U_{(r+1):k}:=(\hat u_{r+1},\ldots, \hat u_k)$.  We decompose $\hat P_\cd{i} = \hat P_\cd{i}^\tpo + \hat P_\cd{i}^\tpt$, where $\hat P_\cd{i}^\tpo := {\hat U_{1:r}\hat U_{1:r}^T}\hat P_\cd{i} $ and $\hat P_\cd{i}^\tpt  := {\hat U_{\br{r+1}:k}\hat U_{\br{r+1}:k}^T}\hat P_\cd{i}$.  Similarly, for each $a\in[k]$, we decompose $\hat \theta_a = \hat \theta_a^\tpo + \hat \theta_a^\tpt$, where $\hat \theta_a^\tpo :=   {\hat U_{1:r}\hat U_{1:r}^T}\hat \theta_a$ and $\hat \theta_a^\tpt := {\hat U_{\br{r+1}:k}\hat U_{\br{r+1}:k}^T}\hat \theta_a$. Due to the orthogonality of $\cbr{\hat u_l}_{l\in[k]}$, we obtain that for any $i\in[n]$ and any $a\in[k]$ such that $a\neq z^*_i$,
\begin{align*}
\indic{\hat z_i =a}  &\leq  \indic{\norm{\hat P_\cd{i}^\tpo + \hat P_\cd{i}^\tpt - \hat \theta_a^\tpo -  \hat \theta_a^\tpt}^2 \leq \norm{\hat P_\cd{i}^\tpo + \hat P_\cd{i}^\tpt - \hat \theta_{z^*_i}^\tpo - \hat \theta_{z^*_i}^\tpt }^2}  \\
& = \indic{ 2\iprod{\hat P_\cd{i}^\tpo - \hat \theta^\tpo_{z^*_i}}{\hat \theta^\tpo_{z^*_i} - \hat \theta^\tpo_{a}} + \norm{\hat \theta^\tpo_{z^*_i} - \hat \theta^\tpo_{a}}^2 \leq  2\iprod{\hat P_\cd{i}^\tpt }{ \hat \theta_a^\tpt - \hat \theta^\tpt_{z^*_i}}  - \norm{\hat \theta_a^\tpt}^2 + \norm{\hat \theta^\tpt_{z^*_i}}^2}
\end{align*}

We denote  $\tau'' = o(1)$ to be another sequence which we will specify later.  Then the above display can be decomposed and upper bounded  by
\begin{align*}
\indic{\hat z_i =a}   \leq  &   \indic{  \norm{\hat \theta^\tpo_{z^*_i} - \hat \theta^\tpo_{a}} -\frac{\tau''\Delta^2  + \norm{\hat \theta^\tpt_{z^*_i}}^2}{\norm{\hat \theta^\tpo_{z^*_i} - \hat \theta^\tpo_{a}}} \leq  2\norm{\hat P_\cd{i}^\tpo - \hat \theta^\tpo_{z^*_i}} } \\ &  + \indic{\tau''\Delta^2 \leq 2\iprod{\hat P_\cd{i}^\tpt }{ \hat \theta_a^\tpt - \hat \theta^\tpt_{z^*_i}}  }  =:A_{i,a}+B_{i,a}.
\end{align*}
Then
\begin{align}
\nonumber\E \ell(\hat z,z^*) &\leq \frac{1}{n}\sum_{i\in[n]}\sum_{a\in[k]:a\neq z^*_i} \E \indic{\hat z_i =a} \\
&\leq \pbr{\mathf^\complement} +  \frac{1}{n}\sum_{i\in[n]}\sum_{a\in[k]:a\neq z^*_i} \E A_{i,a} \indic{\mathf} +  \frac{1}{n}\sum_{i\in[n]}\sum_{a\in[k]:a\neq z^*_i} \E B_{i,a} \indic{\mathf}.\label{eqn:proof_gmm_6}
\end{align}
We are going to establish upper bounds first for  $n^{-1}\sum_{i\in[n]}\sum_{a\in[k]:a\neq z^*_i} \E B_{i,a} \indic{\mathf} $ and then for $n^{-1}\sum_{i\in[n]}\sum_{a\in[k]:a\neq z^*_i} \E A_{i,a} \indic{\mathf} $.

~\\
\emph{(Analysis on $n^{-1}\sum_{i\in[n]}\sum_{a\neq z^*_i} \E B_{i,a} \indic{\mathf} $).} For  $\sum_{i\in[n]}\sum_{a\neq z^*_i} \E B_{i,a} \indic{\mathf} $, we can directly use upper bounds established in Section 4.4.3 of \cite{loffler2019optimality}\footnote{The model in \cite{loffler2019optimality} assumes $\{\epsilon_j\}\iid\mathn(0,I)$ while in this paper we assume $\{\epsilon_j\}\iid\mathn(0,\sigma^2I)$. To directly use results from \cite{loffler2019optimality}, we can re-scale our data to have $X'_j = X_j/\sigma$ for all $j\in[n]$. Then $\{X'_j\}$ has $\mathn(0,I)$ noise and the separation between their centers becomes $\Delta/\sigma$. Then all the results from \cite{loffler2019optimality} can be used here with $\Delta$ replaced by $\Delta/\sigma$.}. It proves that for any $i\in[n]$, 
\begin{align*}
\sum_{a\in[k]:a\neq z^*_i} B_{i,a}\indic{\mathf\cap \mathcal{T}}\leq 2 \ebr{-\frac{1}{2} \left ( c_4 \frac{\tau''\Delta}{k^\frac{7}{2}\tau^2 \beta^{-\frac{1}{2}} (1+\frac{p}{n})\sigma}  \sqrt{\frac{n-k}{3n}} \right )^2 \frac{\Delta^2}{\sigma^2}} ,
\end{align*}
where $c_4>0$ is some constant, and $\mathcal{T}$ is some high-probability event in the sense that
\begin{align*}
\pbr{\mathcal{T}}\geq 1-nk\exp \left ( - \frac{(n-k)}{9}\right ).
\end{align*}
Hence,
\begin{align*}
\frac{1}{n}\sum_{i\in[n]}\sum_{a\in[k]:a\neq z^*_i} \E B_{i,a} \indic{\mathf} &\leq \frac{1}{n}\sum_{i\in[n]}\sum_{a\in[k]:a\neq z^*_i} \E B_{i,a} \indic{\mathf\cap\mathcal{T}} + \pbr{\mathcal{T}^\complement}\\
 &\leq 2 \ebr{-\frac{1}{2} \left ( c_4 \frac{\tau''\Delta}{k^\frac{7}{2}\tau^2 \beta^{-\frac{1}{2}} (1+\frac{p}{n})\sigma}  \sqrt{\frac{n-k}{3n}} \right )^2 \frac{\Delta^2}{\sigma^2}}  + nk\exp \left ( - \frac{(n-k)}{9}\right ).
\end{align*}

~\\
\emph{(Analysis on $n^{-1}\sum_{i\in[n]}\sum_{a\neq z^*_i} \E A_{i,a} \indic{\mathf} $).} We first follow some algebra as in Section 4.4.2 of \cite{loffler2019optimality} to simplify $ A_{i,a}\indic{\mathcal{F}}$. For any $i\in[n]$ and $a\neq z^*_i$, it proves
\begin{align}\label{eqn:proof_gmm_1}
 &  A_{i,a}\indic{\mathcal{F}}  \leq \indic{\br{1-c_1 \tau'' - \frac{c_1k^2\tau\beta^{-\frac{1}{2}} \sqrt{1+\frac{p}{n}} \sigma}{\Delta}} \Delta \leq 2\norm{\hat P_\cd{i}^\tpo - \hat \theta^\tpo_{z^*_i}} }\indic{\mathcal{F}},
 \end{align} 
 for some constant $c_1>0$. Still working on the event $\mathcal{F}$, it also proves
 \begin{align}\label{eqn:proof_gmm_2}
 \norm{\hat P_\cd{i}^\tpo - \hat \theta^\tpo_{z^*_i}} \leq  \norm{\hat P_\cd{i}^\tpo - \hat U_{1:r}\hat U_{1:r}^T \theta^*_{z^*_i}}  + 8\sqrt{2} \sqrt{\beta^{-1} k^2 \br{1+\frac{p}{n}}}\sigma.
 \end{align}
 
 Our following analysis on $A_{i,a}\indic{\mathcal{F}}$ is different from the rest proof in Section 4.4.2 of \cite{loffler2019optimality}. Note that $\hat P_\cd{i}^\tpo - \hat U_{1:r}\hat U_{1:r}^T \theta^*_{z^*_i} 
 = \hat U_{1:r}\hat U_{1:r}^TX_i -  \hat U_{1:r}\hat U_{1:r}^T \theta^*_{z^*_i} =\hat U_{1:r}\hat U_{1:r}^T \epsilon_i  $. Then (\ref{eqn:proof_gmm_1}) and (\ref{eqn:proof_gmm_2}) give
 \begin{align}\label{eqn:proof_gmm_3}
 A_{i,a}\indic{\mathcal{F}}  \leq \indic{\br{1-c_2 \tau'' - \frac{c_2k^2\tau\beta^{-\frac{1}{2}} \br{1+\sqrt{\frac{p}{n}}} \sigma}{\Delta}} \Delta \leq 2\norm{\hat U_{1:r}\hat U_{1:r}^T \epsilon_i } }\indic{\mathcal{F}},
 \end{align}
 where we use $\tau\rightarrow\infty$ and the fact that $1+\sqrt{p/n},\sqrt{1+p/n}$ are of the same order. 

 Recall the definition of $X_{-i}$ in (\ref{eqn:X_minus_i}) and  $\hat U_{-i,1:r}\hat U_{-i,1:r}^T$ is the leave-one-out counterpart of $\hat U_{1:r}\hat U_{1:r}^T$.
 For (\ref{eqn:proof_gmm_3}), we can decompose $\normop{\hat U_{1:r}\hat U_{1:r}^T \epsilon_i } $ into
\begin{align*}
\norm{\hat U_{1:r}\hat U_{1:r}^T \epsilon_i } \leq \norm{\hat U_{-i,1:r}\hat U_{-i,1:r}^T \epsilon_i }  + \fnorm{\hat U_{1:r}\hat U_{1:r}^T-\hat U_{-i,1:r}\hat U_{-i,1:r}^T}\norm{\epsilon_i}.
\end{align*}
To upper bound $\normf{\hat U_{1:r}\hat U_{1:r}^T-\hat U_{-i,1:r}\hat U_{-i,1:r}^T}$, we are going to use Theorem \ref{thm:perturbation_r}. Since (\ref{eqn:proof_gmm_4})-(\ref{eqn:proof_gmm_5}) hold, under the assumption $\beta n/k^4\geq 100$, we have
\begin{align*}
\frac{\lambda_{r}-\lambda_{r+1}}{\max\cbr{\norm{E},\sqrt{\frac{k^2}{n\beta}}\lambda_{r+1}}}\geq \frac{\tau}{2}.
\end{align*}
Applying Theorem \ref{thm:perturbation_r}, we have
\begin{align*}
\fnorm{\hat U_{1:r}\hat U_{1:r}^T-\hat U_{-i,1:r}\hat U_{-i,1:r}^T} &\leq  \frac{256\sqrt{rk}}{\sqrt{n\beta}} + \frac{256\norm{\hat U_{-i,1:r}\hat U_{-i.1:r}^T\epsilon_i}}{\lambda_r}.
\end{align*}
Hence,
\begin{align*}
\norm{\hat U_{1:r}\hat U_{1:r}^T \epsilon_i } &\leq \norm{\hat U_{-i,1:r}\hat U_{-i,1:r}^T \epsilon_i }  + \br{\frac{256\sqrt{rk}}{\sqrt{n\beta}} + \frac{256\norm{\hat U_{-i,1:r}\hat U_{-i.1:r}^T\epsilon_i}}{\lambda_r}}\norm{E} \\
&= \frac{256 k\norm{E}}{\sqrt{n\beta }} + \br{1+ \frac{256\norm{E}}{\lambda_r}}\norm{\hat U_{-i,1:r}\hat U_{-i.1:r}^T\epsilon_i} \\
&\leq \frac{256\sqrt{2} k(\sqrt{n}+\sqrt{p})\sigma}{\sqrt{n\beta }}+ \br{1+ \frac{256\sqrt{2}(\sqrt{n}+\sqrt{p})\sigma}{\tau \sqrt{n+p}\sigma}}\norm{\hat U_{-i,1:r}\hat U_{-i.1:r}^T\epsilon_i} \\
&\leq 512  k\beta^{-0.5}\br{1+\sqrt{\frac{p}{n}}}\sigma + \br{1+512\tau^{-1}}\norm{\hat U_{-i,1:r}\hat U_{-i.1:r}^T\epsilon_i},
\end{align*}
where in the second to the last inequality, we use (\ref{eqn:proof_gmm_4}) for $\lambda_r$ and the event $\mathf$ for $\norm{E}$. Then (\ref{eqn:proof_gmm_3}) leads to
\begin{align*}
 A_{i,a}\indic{\mathcal{F}}  &\leq \indic{\br{1-c_3 \tau'' - \frac{c_3k^2\tau\beta^{-\frac{1}{2}} \br{1+\sqrt{\frac{p}{n}}} \sigma}{\Delta}} \Delta \leq 2\br{1+512\tau^{-1}}\norm{\hat U_{-i,1:r}\hat U_{-i.1:r}^T\epsilon_i}}\indic{\mathcal{F}}\\
 &\leq \indic{\br{1-c_4\br{\frac{k^2\tau\beta^{-\frac{1}{2}} \br{1+\sqrt{\frac{p}{n}}} \sigma}{\Delta} + \tau^{-1}} }\Delta \leq 2\norm{\hat U_{-i,1:r}\hat U_{-i,1:r}^T\epsilon_i}},
\end{align*}
where $c_3,c_4>0$ are some constants. As long as $1-c_4(k^2\tau\beta^{-0.5}(1+\sqrt{p/n})\sigma/\Delta + \tau^{-1})>1/2$, we can use Lemma \ref{lem:sub_gaussian_projection} to calculate the tail probability of $\normop{\hat U_{-i,1:r}\hat U_{-i,1:r}^T\epsilon_i}$. Following the proof of Theorem \ref{thm:subg}, we have 
\begin{align*}
\E  A_{i,a}\indic{\mathcal{F}} \leq \ebr{-\br{1-c_5\br{\frac{k^2\tau\beta^{-\frac{1}{2}} \br{1+\sqrt{\frac{p}{n}}} \sigma}{\Delta} + \tau^{-1}}}\frac{\Delta^2}{8\sigma^2}},
\end{align*}
for some constant $c_5>0$.
Then we have,
\begin{align*}
n^{-1}\sum_{i\in[n]}\sum_{a\in[k]:a\neq z^*_i} \E A_{i,a} \indic{\mathf} &\leq k \ebr{-\br{1-c_5\br{\frac{k^2\tau\beta^{-\frac{1}{2}} \br{1+\sqrt{\frac{p}{n}}} \sigma}{\Delta} + \tau^{-1}}}\frac{\Delta^2}{8\sigma^2}}.
\end{align*}

~\\
\emph{(Obtaining the Final Result.)} From (\ref{eqn:proof_gmm_6}) and the  above upper bounds  on  $n^{-1}\sum_{i\in[n]}\sum_{a\in[k]:a\neq z^*_i} \E B_{i,a} \indic{\mathf} $  and $n^{-1}\sum_{i\in[n]}\sum_{a\in[k]:a\neq z^*_i} \E A_{i,a} \indic{\mathf} $, we have
\begin{align*}
\E \ell(\hat z,z^*)&\leq e^{-0.08n}+2 \ebr{-\frac{1}{2} \left ( c_4 \frac{\tau''\Delta}{k^\frac{7}{2}\tau^2 \beta^{-\frac{1}{2}} (1+\frac{p}{n})\sigma}  \sqrt{\frac{n-k}{3n}} \right )^2 \frac{\Delta^2}{\sigma^2}}  + nk\exp \left ( - \frac{(n-k)}{9}\right ) \\
&\quad + k \ebr{-\br{1-c_5\br{\frac{k^2\tau\beta^{-\frac{1}{2}} \br{1+\sqrt{\frac{p}{n}}} \sigma}{\Delta} + \tau^{-1}}}\frac{\Delta^2}{8\sigma^2}}.
\end{align*}
Since we assume $\beta n/k^4\geq 100$, we have $(n-k)/n>0.99$. Hence, under the assumption that $\Delta/(k^{3.5}\beta^{-0.5}(1+\frac{p}{n})\sigma) \rightarrow\infty$, we can take $\tau,\tau''$ to be
\begin{align*}
\tau = \tau''^{-1} := \br{\frac{\Delta}{k^{3.5}\beta^{-0.5}\br{1+\frac{p}{n}}\sigma}}^{0.25}
\end{align*}
such that $\tau\rightarrow\infty$ and $\tau''=o(1)$. Then for some constant $c_6>0$, we have
\begin{align*}
\E \ell(\hat z,z^*) &\leq e^{-0.08n} + 2 \ebr{-\frac{c_4^2}{12}  \br{\frac{\Delta}{k^{3.5}\beta^{-0.5}\br{1+\frac{p}{n}}\sigma}}^{0.5}\frac{\Delta^2}{\sigma^2}} + nke^{-0.1n}\\
&\quad + k \ebr{-\br{1-2c_5\br{\frac{\Delta}{k^{3.5}\beta^{-0.5}\br{1+\frac{p}{n}}\sigma}}^{-0.25}}\frac{\Delta^2}{8\sigma^2}}\\
&\leq \ebr{-\br{1-c_6\br{\frac{\Delta}{k^{3.5}\beta^{-0.5}\br{1+\frac{p}{n}}\sigma}}^{-0.25}}\frac{\Delta^2}{8\sigma^2}} + 2e^{-0.08n}.
\end{align*}

\section{Proofs of Results in Section \ref{sec:lower}}\label{sec:proof_lower}
\subsection{Proof of Theorem \ref{thm:lower_bound_spectral}}

The proof of Theorem \ref{thm:lower_bound_spectral} relies on the following entrywise decomposition that is analogous to Lemma \ref{lem:decomposition} but in an opposite direction. Note the the singular vectors $\hat u_1$, and $\{\hat u_{1,-i}\}_{i\in[n]}$ are all identifiable up to sign. Without loss of generality, we assume $\iprod{\hat u_1}{u_1}\geq 0$ and $\iprod{\hat u_{1,-i}}{u_1}\geq 0$ for all $i\in[n]$.

\begin{lemma}\label{lem:lower_bound_enrywise}
Consider the model (\ref{eqn:simple}). 
Let $\phi\in\Phi$ be the permutation such that $\ell(\check z,z^*)=\frac{1}{n}|\{i\in[n]:\check z_i \neq \phi(z^*_i)\}| $. Then there exists a constants $C,C_1>0$ such that if
\begin{align}\label{eqn:lower_bound_entrywise_condition}
\frac{\Delta}{\beta^{-0.5} n^{-0.5}\norm{E}} \geq C,
\end{align}
then for any $i\in[n]$,
\begin{align}\label{eqn:decomposition_lower}
\indic{\check z_i \neq \phi(z_i^*)}  \geq \indic{\br{1+\frac{C_1\beta^{-0.5} n^{-0.5}\norm{E}}{\Delta}}\Delta \leq -2(\hat u_{1,-i}^T\epsilon_i ) \text{sign}( u_1^T  \theta_{ \phi(z^*_i) }  )  }.
\end{align}
\end{lemma}
\begin{proof}
The proof mainly follows the proofs of Lemma \ref{lem:decomposition_simple} and Lemma \ref{lem:decomposition} with some modifications such as adding a negative term instead of a positive term in order to obtain a lower bound. 

We first write $\check z$  equivalently  as
\begin{align*}
 \br{\check z,\cbr{\check \theta_j}_{j=1}^2} = \argmin_{ z\in [2]^n , \cbr{ \theta_j}_{j=1}^{2} \in \mathr^{p}} \sum_{i\in[n]} \norm{\hat u_1\hat u_1^{T} X_i - \theta_{z_i}}^2,
\end{align*}
where $\check\theta_a = \hat u_1 \check c_a$ for each $a\in[2]$. 
Note that $k=2$. From Proposition \ref{prop:poly}, we have
\begin{align*}
\frac{1}{n}\abs{\cbr{i\in[n]:\check z_i \neq \phi(z^*_i)}} \leq \frac{C_0k\norm{E}^2}{n\Delta^2},
\end{align*}
and
\begin{align}\label{eqn:rank_one_proof_1}
\max_{a\in[2]} \norm{\check \theta_{\phi(a)}-\theta^*_a}\leq C_0\beta^{-0.5} kn^{-0.5}\norm{E},
\end{align}
for some permutation $\phi:[2]\rightarrow[2]$ and some constant $C_0>0$. Without loss of generality, assume $\phi=\text{Id}$.

Recall that $\theta_1^* = -\theta^*_2 =\delta \one_p$, $u_1 = 1/\sqrt{p}\one_p$, $\lambda_1 = \delta \sqrt{np}= \frac{\Delta \sqrt{n}}{2}$, and $|u_1^T (\theta^*_{ z^*_i }  - (-\theta^*_{ z^*_i }))| = 2\delta\sqrt{p} =\Delta$.  By  Davis-Kahan Theorem, we have
\begin{align*}
\min_{s\in\pm 1}\norm{\hat u_1 - su_1}\leq \frac{\norm{E}}{\lambda_1}  =\frac{2\norm{E}}{\sqrt{n}\Delta} \leq 1/16,
\end{align*}
where the last inequality is due to the assumption (\ref{eqn:Delta_poly}). 
Since we  assume $\iprod{\hat u_1}{u_1}\geq 0$, we have $\norm{\hat u_1 - su_1} = \min_{s\in\pm 1}\norm{\hat u_1 - su_1}$.

Consider any $i\in[n]$ and any $a\in[2]$ such that $a\neq z^*_i$. Note that for any scalars $x,y,w$, if $\abs{x-y}\leq\abs{x-w}$, we have equivalently $  \text{sign}(w-y) (y+w)/2\geq \text{sign}(w-y)x$. Since $(y+w)/2 =(y-w)/2 + w$, a sufficient condition is $  \abs{w-y}/2 + \abs{w} \leq (-\text{sign}(w-y))x$. Hence, we have
\begin{align*}
&\indic{ \norm{\hat u_1\hat u_1^TX_i - \check \theta_{ a }} \leq \norm{\hat u_1\hat u_1^TX_i - \check \theta_{ z^*_i }}} \\
& =\indic{ \abs{\hat u_1^TX_i - \hat u_1^T\check \theta_{ a }} \leq \abs{\hat u_1^TX_i - \hat u_1^T\check \theta_{ z^*_i }}} \\
&=\indic{ \abs{\hat u_1^T\epsilon_i - \hat u_1^T\br{\check \theta_{ a } - \theta^*_{z^*_i}}} \leq \abs{\hat u_1^T\epsilon_i -  \hat u_1^T\br{\check \theta_{ z^*_i }  - \theta^*_{z^*_i}}}} \\
&\geq \indic{ \frac{1}{2} \abs{ \hat u_1^T(\check \theta_{ z^*_i }  - \check \theta_{ a })} + \abs{\hat u_1^T\br{\check \theta_{ z^*_i } - \theta^*_{z^*_i}}} \leq   -(\hat u_1^T\epsilon_i) \text{sign}( \hat u_1^T(\check \theta_{ z^*_i }  - \check \theta_{ a }))  }\\
&\geq \indic{     \norm{ \check \theta_{ z^*_i }  - \check \theta_{ a }} + 2\norm{\check \theta_{ z^*_i } - \theta^*_{z^*_i}} \leq -2(\hat u_1^T\epsilon_i) \text{sign}( \hat u_1^T(\check \theta_{ z^*_i }  - \check \theta_{ a })) }.
\end{align*}

We are going to show $\text{sign}(\hat u_1^T (\check \theta_{ z^*_i }  -\check \theta_{ a })) = \text{sign}( u_1^T ( \theta^*_{ z^*_i }  - \theta^*_{ a }))$.
By (\ref{eqn:rank_one_proof_1}), we have
\begin{align*}
\iprod{\check \theta_{ z^*_i }  -\check \theta_{ a }}{ \theta^*_{ z^*_i }  - \theta^*_{ a }} &= \norm{ \theta^*_{ z^*_i }  - \theta^*_{ a }}^2 + \iprod{\check \theta_{ z^*_i } - \theta^*_{ z^*_i }}{ \theta^*_{ z^*_i }  - \theta^*_{ a }} + \iprod{\check \theta_a - \theta^*_a}{ \theta^*_{ z^*_i }  - \theta^*_{ a }} \\
&\geq \Delta^2\br{1-\frac{2C_0k\beta^{-0.5} n^{-0.5}\norm{E}}{\Delta}}\\
&>0,
\end{align*}
where the last inequality holds as long as $\Delta >2C_0\beta^{-0.5} kn^{-0.5}\norm{E}$. Due to the fact $\theta^*_{ z^*_i }  - \theta^*_{ a }\in\text{span}(u_1)$, $\check \theta_{ z^*_i }  - \check\theta^*_{ a }\in\text{span}(\hat u_1)$, and $\iprod{\hat u_1}{u_1}\geq 0$, if $u_1,\theta^*_{ z^*_i }  - \theta^*_{ a }$ are in the same direction, then $\hat u_1,\check \theta_{ z^*_i }  - \check\theta^*_{ a }$ must also be in the same direction, and vice versa. Hence, we have $\text{sign}(\hat u_1^T (\check \theta_{ z^*_i }  -\check \theta_{ a })) = \text{sign}( u_1^T ( \theta^*_{ z^*_i }  - \theta^*_{ a }))$. Thus,
\begin{align*}
&\indic{ \norm{\hat u_1\hat u_1^TX_i - \check \theta_{ a }} \leq \norm{\hat u_1\hat u_1^TX_i - \check \theta_{ z^*_i }}} \\
&\geq \indic{     \norm{ \check \theta_{ z^*_i }  - \check \theta_{ a }} + 2\norm{\check \theta_{ z^*_i } - \theta^*_{z^*_i}} \leq -2(\hat u_1^T\epsilon_i) \text{sign}(  u_1^T( \theta^*_{ z^*_i }  -  \theta^*_{ a })) }.
\end{align*}
Following the same analysis as in the proof of Lemma \ref{lem:decomposition_simple}, we can get the following result that is analogous to (\ref{eqn:lemma_1_proof_2}):
\begin{align*}
&\indic{ \norm{\hat u_1\hat u_1^TX_i - \check \theta_{ a }}\leq \norm{\hat u_1\hat u_1^TX_i - \check \theta_{ z^*_i }}} \\
&\geq  \indic{\br{1+\frac{4C_0\beta^{-0.5}kn^{-0.5}\norm{E}}{\Delta}}\Delta \leq -2(\hat u_1^T\epsilon_i ) \text{sign}( u_1^T ( \theta^*_{ z^*_i }  - \theta^*_{ a }))   }.
\end{align*}

Next, we are going to decompose $\hat u_1^T\epsilon_i$ following the proof of Lemma \ref{lem:decomposition}. Denote $\hat u_{1,-i}$ be the leave-one-out counterpart of $\hat u_1$, i.e., $\hat u_{1,-i}$ is the leading left singular vector of $X_{-i}$. 
Since we  assume $\iprod{\hat u_{1,-i}}{u_1}\geq 0$, we have $\norm{\hat u_{1,-i} -u_1}\leq 2\norm{E}/(\sqrt{n-1}\Delta)$. As a result, we have $\norm{\hat u_{1,-i} -\hat u_1}\leq 4\norm{E}/(\sqrt{n-1}\Delta)$ which leads to
\begin{align}\label{eqn:lower_bound_proof_1}
\iprod{\hat u_{1,-i}}{\hat u_1}\geq 1-4\norm{E}/(\sqrt{n-1}\Delta) >0.
\end{align}

We have the following decomposition: 
\begin{align*}
&(\hat u_1^T\epsilon_i) \text{sign}( u_1^T ( \theta^*_{ z^*_i }  - \theta^*_{ a })) \\
  &= \iprod{\hat u_1}{\hat u_1\hat u_1^T\epsilon_i} \text{sign}( u_1^T ( \theta^*_{ z^*_i }  - \theta^*_{ a }))  \\
& = \iprod{\hat u_1}{(\hat u_{1,-i}\hat u_{1,-i}^T)\epsilon_i} \text{sign}( u_1^T ( \theta^*_{ z^*_i }  - \theta^*_{ a }))  + \iprod{\hat u_1}{(\hat u_1 \hat u_1^T-\hat u_{1,-i}\hat u_{1,-i}^T)\epsilon_i}\text{sign}( u_1^T ( \theta^*_{ z^*_i }  - \theta^*_{ a }))  \\
&=\iprod{\hat u_1}{\hat u_{1,-i}} (\hat u_{1,-i}^T\epsilon_i) \text{sign}( u_1^T ( \theta^*_{ z^*_i }  - \theta^*_{ a }))  + \iprod{\hat u_1}{(\hat u_1 \hat u_1^T-\hat u_{1,-i}\hat u_{1,-i}^T)\epsilon_i}\text{sign}( u_1^T ( \theta^*_{ z^*_i }  - \theta^*_{ a }))  \\
&\leq   \iprod{\hat u_1}{\hat u_{1,-i}}(\hat u_{1,-i}^T\epsilon_i) \text{sign}( u_1^T ( \theta^*_{ z^*_i }  - \theta^*_{ a }))   +\norm{\hat u_1 \hat u_1^T-\hat u_{1,-i}\hat u_{1,-i}^T}\norm{\epsilon_i}.
\end{align*}
Note that $\lambda_1/\norm{E}=\Delta \sqrt{n}/(2\norm{E})$ is greater than 16 under the assumption (\ref{eqn:lower_bound_entrywise_condition}) holds for a large constant $C$.  From Theorem \ref{thm:perturbation} we have
\begin{align*}
\norm{\hat u_1 \hat u_1^T-\hat u_{1,-i}\hat u_{1,-i}^T}&\leq \frac{128}{\lambda_1/\norm{E}}\br{\frac{k}{\sqrt{\beta n}} +\frac{\norm{\hat u_{1,-i}\hat u_{1,-i}^T\epsilon_i}}{\lambda_1}}.
\end{align*}
Then,
\begin{align*}
&(\hat u_1^T\epsilon_i) \text{sign}( u_1^T ( \theta^*_{ z^*_i }  - \theta^*_{ a })) \\
&\leq  \iprod{\hat u_1}{\hat u_{1,-i}}(\hat u_{1,-i}^T\epsilon_i )\text{sign}( u_1^T ( \theta^*_{ z^*_i }  - \theta^*_{ a }))  +\br{ \frac{128k}{\sqrt{n\beta} (\lambda_1/\norm{E})} +\frac{128\norm{\hat u_{1,-i}\hat u_{1,-i}^T\epsilon_i}}{\lambda_1^2/\norm{E}}}\norm{E} \\
& = \iprod{\hat u_1}{\hat u_{1,-i}} (\hat u_{1,-i}^T\epsilon_i )\text{sign}( u_1^T ( \theta^*_{ z^*_i }  - \theta^*_{ a }))   + \frac{256 n^{-0.5}k\beta^{-0.5} \norm{E}^2}{\Delta} + \frac{512\abs{\hat u_{1,-i}^T\epsilon_i} n^{-1}\norm{E}^2}{\Delta^2}.
\end{align*}

So far we have obtained
\begin{align*}
&\indic{ \norm{\hat u_1\hat u_1^TX_i - \check \theta_{ a }} \leq \norm{\hat u_1\hat u_1^TX_i - \check \theta_{ z^*_i }}} \\
&\geq   \mathbb{I}\Bigg\{\br{1+\frac{4C_0\beta^{-0.5}kn^{-0.5}\norm{E}}{\Delta}}\Delta \leq -2 \iprod{\hat u_1}{\hat u_{1,-i}} (\hat u_{1,-i}^T\epsilon_i )\text{sign}( u_1^T ( \theta^*_{ z^*_i }  - \theta^*_{ a })) \\
&\quad -\frac{256 n^{-0.5}k\beta^{-0.5} \norm{E}^2}{\Delta}   -\frac{512\abs{\hat u_{1,-i}^T\epsilon_i} n^{-1}\norm{E}^2}{\Delta^2} \Bigg\}\\
&=  \mathbb{I}\Bigg\{\br{1+\frac{4C_0\beta^{-0.5}kn^{-0.5}\norm{E}}{\Delta} + \frac{256 n^{-0.5}k\beta^{-0.5} \norm{E}^2}{\Delta^2} }\Delta \\
&\quad  \leq  -2 \iprod{\hat u_1}{\hat u_{1,-i}} (\hat u_{1,-i}^T\epsilon_i )\text{sign}( u_1^T ( \theta^*_{ z^*_i }  - \theta^*_{ a })) -\frac{512\abs{\hat u_{1,-i}^T\epsilon_i} n^{-1}\norm{E}^2}{\Delta^2}  \Bigg\}.
\end{align*}
From (\ref{eqn:lower_bound_proof_1}) we have
\begin{align*}
\iprod{\hat u_{1,-i}}{\hat u_1} - \frac{512 n^{-1}\norm{E}^2}{\Delta^2} & \geq 1-4\frac{\norm{E}(n-1)^{-0.5}}{\Delta} - \frac{512n^{-1}\norm{E}^2}{\Delta^2}\\
&\geq 1-\frac{16n^{-0.5}\norm{E}}{\Delta}\geq \frac{1}{2},
\end{align*}
assuming $\frac{\Delta}{n^{-0.5}\norm{E}}\geq 64$.
For any $x,y,z,w\in\mathr$ such that $x\geq 0$, $1\geq z\geq 0$, and $z\abs{y}>w\geq 0$, we have $ \indic{x\leq zy-w}\geq \indic{x \leq \br{z-w/|y|}y}$. We then have,
\begin{align*}
&\indic{ \norm{\hat u_1\hat u_1^TX_i - \check \theta_{ a }}\leq \norm{\hat u_1\hat u_1^TX_i - \check \theta_{ z^*_i }}} \\
&\geq  \mathbb{I}\Bigg(\br{1+\frac{4C_0\beta^{-0.5}kn^{-0.5}\norm{E}}{\Delta} + \frac{256 n^{-0.5}k\beta^{-0.5} \norm{E}^2}{\Delta^2} }\Delta \\
&\quad\quad \leq -2\br{1-\frac{16n^{-0.5}\norm{E}}{\Delta}}(\hat u_{1,-i}^T\epsilon_i ) \text{sign}( u_1^T ( \theta^*_{ z^*_i }  - \theta^*_{ a }))   \Bigg)\\
&\geq \indic{\br{1+\frac{C_1\beta^{-0.5} n^{-0.5}\norm{E}}{\Delta}}\Delta \leq -2(\hat u_{1,-i}^T\epsilon_i ) \text{sign}( u_1^T ( \theta^*_{ z^*_i }  - \theta^*_{ a }))  }.
\end{align*}
Since $\theta^*_{ a } =- \theta^*_{ z^*_i }$, we have $\text{sign}( u_1^T ( \theta^*_{ z^*_i }  - \theta^*_{ a }))   = \text{sign}( u_1^T  \theta^*_{ z^*_i }  )  $. The proof is complete.
\end{proof}

\begin{proof}[Proof of Theorem \ref{thm:lower_bound_spectral}]
Recall that $\lambda_1 = \Delta\sqrt{n}/2$. Same as the proof of Theorem \ref{thm:subg}, we work on the high-probability event  (\ref{eqn:gmm_f}).

For the upper bound, from Lemma \ref{lem:decomposition},  there exists some $\phi\in\Phi$ such that for any $i\in[n]$, 
\begin{align*}
\indic{\hat z_i \neq \phi(z_i^*)} &\leq  \indic{\br{1-C_1\psi_3^{-1}}\Delta \leq 2\norm{\hat u_{1,-i}\hat u_{-i}^T\epsilon_i}}  =  \indic{\br{1-C_1\psi_3^{-1}}\Delta \leq 2\abs{\hat u_{1,-i}^T\epsilon_i}},
\end{align*}
for some $C_1>0$, where the last inequality is due to that $\psi_3$ is large. By Davis-Kahan Theorem, we know there exists some $s_i\in\{-1,1\}$ such that $\norm{\hat u_{1,-i}-s_i u_1}\leq 2\norm{E}/(\sqrt{n-1}\Delta) \leq 4\psi_3^{-1}$. Since $\iprod{\hat u_{1,-i}}{u_1}\geq 0$ is assumed, we have $s_i=1$ for all $i\in[n]$. Then
\begin{align*}
\indic{\hat z_i \neq \phi(z_i^*)}  & \leq  \indic{\br{1-C_1\psi_3^{-1}}\Delta \leq 2\abs{ u_{1}^T\epsilon_i}+ 2\abs{\br{\hat u_{1,-i} - s_iu_1}^T\epsilon_i}}\\
&\leq  \indic{\br{1-(C_1+C_2)\psi_3^{-1}}\Delta \leq 2\abs{ u_{1}^T\epsilon_i} } +  \indic{C_2 \psi^{-1}_3\Delta  \leq 2\abs{\br{\hat u_{1,-i} - s_iu_1}^T\epsilon_i}},
\end{align*}
where $C_2>0$ is a constant whose value will be determined later. Due to the independence of $\hat u_{1,-i} - s_iu_1$ and $\epsilon_i$, we have  $\br{\hat u_{1,-i} - s_iu_1}^T\epsilon_i\sim\text{SG}(16\psi_3^{-2}\sigma^2)$ and then
\begin{align*}
\E \indic{C_2\Delta  \leq 2\abs{\br{\hat u_{1,-i} - s_iu_1}^T\epsilon_i}}\leq 2\ebr{-\frac{C_2^2\Delta^2}{128\sigma^2}}.
\end{align*}
On the other hand, $u_{1}^T\epsilon_i = p^{-\frac{1}{2}}\sum_{j=1}^p \epsilon_{i,j}$ where $\{\epsilon_{i,j}\}_{j\in[p]}$ are i.i.d. with variance $\bar \sigma^2$, which can be approximated by a normal distribution. Since the distribution $F$ is sub-Gaussian, its moment generating function exists. Then we can use the following  KMT quantile inequality (see Proposition [KMT] of \cite{mason2012quantile}). Let $Y\stackrel{d}{=}\bar\sigma^{-1}p^{-\frac{1}{2}}\sum_{j=1}^p \epsilon_{i,j}$. There exist some constants $D,\eta>0$ and $Z\sim\mathn(0,1)$, such that whenever $\abs{Y}\leq \eta\sqrt{p}$, we have
\begin{align*}
\abs{Y-Z}\leq \frac{DY^2}{\sqrt{p}} + \frac{D}{\sqrt{p}}.
\end{align*}
Then,
\begin{align*}
&\E \indic{\br{1-(C_1+C_2)\psi_3^{-1}}\Delta \leq 2\abs{ u_{1}^T\epsilon_i} } \\
& = \E \indic{\br{1-(C_1+C_2)\psi_3^{-1}}\frac{\Delta}{\bar\sigma} \leq 2\abs{Y} } \\
&\leq  \E \indic{\br{1-(C_1+C_2)\psi_3^{-1}}\frac{\Delta}{\bar\sigma} \leq 2\abs{Z} +  \frac{2DY^2}{\sqrt{p}} + \frac{2D}{\sqrt{p}}} + \E \indic{\abs{Y}>\eta \sqrt{p}}\\
&\leq  \E \indic{\br{1-(C_1+C_2+C_3 +2D)\psi_3^{-1}}\frac{\Delta}{\bar\sigma} \leq 2\abs{Z}} + \E \indic{ \frac{2DY^2}{\sqrt{p}}\geq C_3}  + \E \indic{\abs{Y}>\eta \sqrt{p}},
\end{align*}
where $C_3>0$ is a constant. Using the fact that $Y\sim\text{SG}(1)$ with zero mean, we have
\begin{align*}
&\E \indic{\br{1-(C_1+C_2)\psi_3^{-1}}\Delta \leq 2\abs{ u_{1}^T\epsilon_i} } \\
&\leq  2\ebr{-\frac{\br{1-(C_1+C_2+C_3+2D)\psi_3^{-1}}^2\Delta^2}{8\bar \sigma^2}} + 2\ebr{-\frac{C_3\sqrt{p}}{4D}} + 2\ebr{-\frac{\eta^2p}{2}}.
\end{align*}
Then we have 
\begin{align*}
&\E \ell(\check z,z^*)\\
&\leq \frac{1}{n}\sum_{i=1}^n \E \indic{\br{1-(C_1+C_2)\psi_3^{-1}}\Delta \leq 2\abs{ u_{1}^T\epsilon_i} } + \frac{1}{n}\sum_{i=1}^n \E\indic{C_2\Delta  \leq 2\abs{\br{\hat u_{1,-i} - s_iu_1}^T\epsilon_i}} + e^{-0.5n}\\
&\leq 2\ebr{-\frac{\br{1-(C_1+C_2+C_3+2D)\psi_3^{-1}}^2\Delta^2}{8\bar \sigma^2}} \\
&\quad + 2\ebr{-\frac{C_2^2\Delta^2}{128\sigma^2}}+ 2\ebr{-\frac{C_3\sqrt{p}}{4D}} + 2\ebr{-\frac{\eta^2p}{2}}+ e^{-0.5n},
\end{align*}
where $e^{-0.5n}$ is the probability that (\ref{eqn:gmm_f}) does not hold.  Since $\sigma\leq C\bar \sigma$, when $C_2$ is chosen to satisfy $C_2^2/(128C^2)\geq 16$, we have
\begin{align*}
\E \ell(\check z,z^*) \leq  2\ebr{-\frac{\br{1-C''\psi_3^{-1}}^2\Delta^2}{8\bar \sigma^2}}+ \ebr{-C''\sqrt{p}} + e^{-0.5n},
\end{align*}
for some constant $C''>0$.

For the lower bound, from (\ref{eqn:decomposition_lower}) we know 
\begin{align*}
\indic{\check z_i \neq \phi(z_i^*)}  \geq \indic{\br{1+C_4\psi_3^{-1}}\Delta \leq -2(\hat u_{1,-i}^T\epsilon_i ) \text{sign}( u_1^T ( \theta_{ \phi(z^*_i) }  - \theta_{ 3-\phi(z^*_i) }))  },
\end{align*}
for some constant $C_4>0$ assuming $\psi_3$ is large. Using the same argument as in the upper bound, we are going to decompose $\hat u_{1,-i}^T\epsilon_i$ into $ u_{1}^T\epsilon_i$ and $(\hat u_{1,-i}-y_1)^T\epsilon_i$. Hence,
\begin{align*}
\indic{\check z_i \neq \phi(z_i^*)}  &\geq \indic{\br{1+C_4\psi_3^{-1}}\Delta \leq -2( u_{1}^T\epsilon_i ) \text{sign}( u_1^T ( \theta_{ \phi(z^*_i) }  - \theta_{ 3-\phi(z^*_i) }))  - 2\abs{(\hat u_{1,-i}-s_iu_1)^T\epsilon_i} }\\
&\geq \indic{\br{1+\br{C_4+C_5}\psi_3^{-1}}\Delta \leq -2( u_{1}^T\epsilon_i ) \text{sign}( u_1^T ( \theta_{ \phi(z^*_i) }  - \theta_{ 3-\phi(z^*_i) }))  } \\
&\quad -  \indic{C_5\psi_3^{-1}\Delta \leq 2\abs{(\hat u_{1,-i}-s_iu_1)^T\epsilon_i}   },
\end{align*}
for some constant $C_5>0$ whose value to be chosen. Let $$Y'\stackrel{d}{=}\bar\sigma^{-1}( u_{1}^T\epsilon_i ) \text{sign}( u_1^T ( \theta_{ \phi(z^*_i) }  - \theta_{ 3-\phi(z^*_i) }))=\text{sign}( u_1^T ( \theta_{ \phi(z^*_i) }  - \theta_{ 3-\phi(z^*_i) }))\bar\sigma^{-1}p^{-\frac{1}{2}}\sum_{j=1}^p \epsilon_{i,j}.$$ Then using the same argument above, there exists some $Z'\sim \mathn(0,1)$ such that whenever $Y'\leq \eta'\sqrt{p}$, we have $\abs{Y'-Z'}\leq \frac{D'Y'^2}{\sqrt{p}} + \frac{D'}{\sqrt{p}}$ where $D',\eta'>0$ are constants. Then
\begin{align*}
&\E  \indic{\br{1+\br{C_4+C_5}\psi_3^{-1}}\Delta \leq -2( u_{1}^T\epsilon_i ) \text{sign}( u_1^T ( \theta_{ \phi(z^*_i) }  - \theta_{ 3-\phi(z^*_i) }))  }\\
&= \E \indic{\br{1+\br{C_4+C_5}\psi_3^{-1}}\frac{\Delta}{\bar \sigma} \leq -2Y' }\\
&\geq  \E \indic{\br{1+\br{C_4+C_5}\psi_3^{-1}}\frac{\Delta}{\bar \sigma} \leq -2Z' - \frac{2DY'^2}{\sqrt{p}} - \frac{2d}{\sqrt{p}} }\indic{Y'\leq \eta' \sqrt{p}} \\
&\geq  \E \indic{\br{1+\br{C_4+C_5+2D+C_6}\psi_3^{-1}}\frac{\Delta}{\bar \sigma} \leq -2Z'} - \E\indic{ \frac{2DY'^2}{\sqrt{p}} \geq C_6 } -\E\indic{Y'> \eta' \sqrt{p}},
\end{align*}
where $C_6>0$ is a constant. Then following the proof of the upper bound, and by a proper choice of $C_5$, we have
\begin{align*}
\E \ell(\check z,z^*) \geq  2\ebr{-\frac{\br{1+C'''\psi_3^{-1}}^2\Delta^2}{8\bar \sigma^2}}- \ebr{-C'''\sqrt{p}}- e^{-0.5n},
\end{align*}
for some constant $C'''>0$.
\end{proof}

\subsection{Proofs of Lemma \ref{lem:minimax_spectral} and Theorem \ref{thm:optimal}}
\begin{proof}[Proof of Lemma \ref{lem:minimax_spectral}]
For the upper bound, we consider the following likelihood ratio test. For any $x\in\mathr^p$, define 
the two log-likelihood functions as
\begin{align*}
l_1(x) &= \sum_{j=1}^{p} \log f(x_j-\delta)  ,\text{ and } l_2(x) = \sum_{j=1}^{p} \log f(x_j+\delta).
\end{align*}
Then for each $i\in[n]$, define the likelihood ratio test as
\begin{align*}
\hat z_i^\lrt = \begin{cases}
1,\text{ if }l_1(X_i)\geq l_2(X_i),\\
2,\text{ otherwise.}
\end{cases}
\end{align*}
Then for any $i\in[n]$ such that $z^*_i=1$, we have
\begin{align*}
\E \indic{\hat z_i^\lrt =2}& = \pbr{l_2(X_i)>l_1(X_i)} = \pbr{\sum_{j=1}^p \log \frac{f(2\delta + \epsilon_{i,j})}{f(\epsilon_{i,j})} >0} = \pbr{\sum_{j=1}^p \log  \frac{f_{\frac{\Delta}{\sqrt{p}}}(\epsilon_{i,j})}{f_0(\epsilon_{i,j})}>0},
\end{align*}
where we use the fact $2\delta = \frac{\Delta}{\sqrt{p}}$.
Since $\Delta$ is a constant, by local asymptotic normality (c.f., Chapter 7, \cite{van2000asymptotic}), we have
\begin{align*}
\sum_{j=1}^p \log  \frac{f_{\frac{\Delta}{\sqrt{p}}}(\epsilon_{i,j})}{f_0(\epsilon_{i,j})}\stackrel{d}{\rightarrow}\mathn\br{-\frac{ \fc \Delta^2}{2},  \fc \Delta^2}.
\end{align*}
Then, $\lim_{p\rightarrow\infty }\E \indic{\hat z_i^\lrt =2} \leq C_1\ebr{- \fc \Delta^2/8}$ for some constant $C_1>0$. We have the same upper bound if $z^*_i=2$ instead. Hence,
\begin{align*}
 \lim_{p\rightarrow\infty}\inf_{z}\sup_{z^*\in[2]^n}\E \ell(z,z^*) \leq  \lim_{p\rightarrow\infty}\sup_{z^*\in[2]^n}\E \ell(\hat z^\lrt,z^*) \leq \ebr{-\frac{ \fc \Delta^2}{8}}.
\end{align*}

For the lower bound, instead of allowing $z^*\in[2]^n$, we consider a slightly smaller parameter space. Define $\mathcal{Z}=\cbr{z\in[2]^n: z_i=1,\forall 1\leq i\leq n/3, z_i=2,\forall n/3+1\leq i\leq 2n/3}$. Then for any $z,z'\in \mathcal{Z}$ we have $\ell(z,z') =n^{-1}\sum_{i=1}^n\indic{z_i\neq z'_i}\leq 1/3$ due to the fact $n^{-1}\sum_{i=1}^n\indic{\phi(z_i)\neq z'_i}\geq 1/3$ if $\phi\neq \text{Id}$. Hence,
\begin{align*}
\inf_{z}\sup_{z^*\in[2]^n}\E \ell(z,z^*)  &\geq \inf_{z}\sup_{z^*\in\mathcal{Z}}\E \ell(z,z^*)  \geq n^{-1} \inf_{z}\sup_{z^*\in\mathcal{Z}}\E \sum_{i\in[n]}\indic{z_i\neq z^*_i} \\
&\geq n^{-1}  \sum_{i> 2n/3} \inf_{z_i}\sup_{z^*_i\in[2]} \E \indic{z_i\neq z^*_i} = \frac{1}{3}\inf_{z_n}\sup_{z^*_n\in[2]} \E \indic{z_n\neq z^*_n},
\end{align*}
where it is reduced into a testing problem on whether $X_n$ has mean $\theta_1^*$ or $\theta_2^*$. According to the Neyman-Pearson Lemma, the optimal procedure is the likelihood ratio test $\hat z_n^\lrt$ defined above. By the same argument, we have
\begin{align*}
\lim_{p\rightarrow} \inf_{z}\sup_{z^*\in[2]^n}\E \ell(z,z^*)   \geq \frac{1}{3} \lim_{p\rightarrow} \inf_{z_n}\sup_{z^*_n\in[2]} \E \indic{z_n\neq z^*_n} \geq C_2 \ebr{-\frac{ \fc \Delta^2}{8}},
\end{align*}
for some constant $C_2>0$.
\end{proof}

\begin{proof}[Proof of Theorem \ref{thm:optimal}]
First, we have the following connection between the Fisher information $\fc$ and the variance $\bar \sigma^2$:
\begin{align*}
\fc\bar\sigma^2 &= \br{\int \br{\frac{f'}{f}}^2f\diff x }\br{\int x^2 f\diff x}\geq \br{\int \frac{f'}{f} xf\diff x}^2 = \br{\int x f'\diff x}^2=1,
\end{align*}
where we use  Cauchy-Schwarz inequality and the integral by part $\int x f'\diff x = \int xf\diff x - \int f\diff x = 0-1=-1$. The equation holds if and only if $f'/f\propto x$, which is equivalent to $F$ being normally distributed. 
\end{proof}

\section{Auxiliary Lemmas and Propositions and Their Proofs}\label{sec:auxiliary}

\begin{proposition}\label{prop:wedin}
For $Y$ and $\hat Y$ defined in (\ref{eqn:Y_hat_Y}), we have (\ref{eqn:wedin}) holds assuming  $\sigma_\rr-\sigma_{\rr + 1}>2\norm{(I-U_\rr U_\rr ^T)y_n}$.
\end{proposition}
\begin{proof}
Recall the augmented matrix $Y'$ is defined as $(Y,U_\rr U_\rr ^T y_n)$.
 Note that $U_\rr U_\rr^TY$ is the best rank-$\rr$ approximation of $Y$. Since
\begin{align*}
\fnorm{\br{I- U_\rr U_\rr^T}Y'} = \fnorm{\br{\br{I- U_\rr U_\rr^T}Y,0}} =  \fnorm{\br{I- U_\rr U_\rr^T}Y},
\end{align*}
we have $U_\rr U_\rr^TY'$ also being the best rank-$\rr$ approximation of $Y'$. This proves that $\tspan(U_\rr)$ and $U_\rr U_\rr^T$ are also the leading $r$ left singular subspace and projection matrix of $Y'$.
Then $\hat U_\rr  \hat U_\rr ^T -   U_\rr U_\rr ^T$ is about the perturbation between $\hat Y$ and $Y'$.

Let $\sigma'_\rr , \sigma'_{\rr+1}$ be the $\rr $th and $(\rr+1)$th largest singular values of $Y'$, respectively. 
By Wedin's Thereom (see Section 2.3 of \cite{cai2018rate}), if
$\sigma'_\rr -\hat \sigma_{\rr +1}>0$, then we have
\begin{align}\label{eqn:appendix_d_1}
\normf{\text{sin}\; \Theta(\hat U_\rr ,U_\rr )}\leq  \frac{\fnorm{\hat Y-Y'}}{\sigma'_\rr -\hat \sigma_{\rr +1}}= \frac{\norm{(I-U_\rr U_\rr ^T)y_n}}{\sigma'_\rr -\hat \sigma_{\rr +1}}.
\end{align}

Regarding the values of $\sigma'_\rr $ and $\sigma'_{\rr+1}$, first we have
 $\sigma'_\rr  \geq \sigma_\rr$. This is because
\begin{align*}
\sigma'_\rr =\inf_{x\in\tspan(U_\rr)} \norm{x^TY'} = \inf_{x\in\tspan(U_\rr)} \norm{\br{x^TY,x^Ty_n}} \geq \inf_{x\in\tspan(U_\rr)} \norm{x^TY}\geq \sigma_\rr.
\end{align*}
In addition, we have $\sigma'_{\rr + 1}=\sigma_{\rr + 1}$, due to the fact that $(I - U_\rr U_\rr ^T)Y'= ((I - U_\rr U_\rr ^T)Y,0)$. By Weyl's inequality, we have $$|\hat \sigma_{\rr+1}-\sigma'_{\rr+1}|\leq \norm{Y-Y'}=\norm{(I-U_\rr U_\rr ^T)y_n}.$$
Hence, if $\sigma_\rr-\sigma_{\rr + 1}>2\norm{(I-U_\rr U_\rr ^T)y_n}$ is further assumed, we have 
\begin{align}\label{eqn:appendix_d_2}
\sigma'_\rr -\hat \sigma_{\rr +1} &\geq \sigma_\rr - \sigma_{\rr +1} - \norm{(I-U_\rr U_\rr ^T)y_n}\geq \frac{1}{2}\br{ \sigma_\rr - \sigma_{\rr +1}}.
\end{align}
With (\ref{eqn:appendix_d_1}), (\ref{eqn:appendix_d_2}), and the fact $\normf{\hat U_\rr  \hat U_\rr ^T -   U_\rr U_\rr ^T}  = \sqrt{2}\normf{\text{sin}\; \Theta(\hat U_\rr ,U_\rr )}$ (see Lemma 1 of \cite{cai2018rate}), the proof is complete.
\end{proof}

\begin{lemma}\label{lem:sub_gaussian_operator}
Let $E=(\epsilon_1,\ldots,\epsilon_n)\in\mathr^{p\times n}$ be a random matrix with each column $\epsilon_i\sim \text{SG}_p(\sigma^2),\forall i\in[n]$ independently. Then
\begin{align*}
\pbr{\norm{E}\geq 4t\sigma(\sqrt{n} + \sqrt{p})} \leq \ebr{-\frac{(t^2-3)n}{2}},
\end{align*}
for any $t\geq 2$.
\end{lemma}
\begin{proof}
We follow a standard $\epsilon$-net argument. Let $\mathcal{U}$ and $\mathcal{V}$ be a 1/4 covering set of the unit sphere in $\mathr^p$ and  in $\mathr^n$, respectively. That is, for any $u\in \mathr^p$ such that $\norm{u}=1$, there exists a $u'\in\mathcal{U}$ such that $\norm{u'}=1$ and $\norm{u-u'}\leq 1/4$. Similarly, for any $v\in \mathr^n$ such that $\norm{v}=1$, there exists a $v'\in\mathcal{V}$ such that $\norm{v'}=1$ and $\norm{v-v'}\leq 1/4$. Then
\begin{align*}
\abs{u^TEv} &= \abs{u^{'T}Ev' + u^{'T}E(v-v') + (u-u')^{T}Ev' + (u-u')^{T}E(v-v')}\\
&\leq  \abs{u^{'T}Ev'} + \abs{u^{'T}E(v-v')} + \abs{ (u-u')^{T}Ev'} + \abs{(u-u')^{T}E(v-v')}.
\end{align*}
Maximizing over $u,v$ on both sides, we have
\begin{align*}
\norm{E} = \max_{u\in\mathr^p,v\in\mathr^n:\norm{u}=\norm{v}=1}\abs{u^TEv} \leq \max_{u'\in\mathcal{U},v'\in\mathcal{V}} \abs{u^{'T}Ev'}  + \frac{1}{4}\norm{E} + \frac{1}{4}\norm{E} + \frac{1}{16}\norm{E}.
\end{align*}
Hence,
\begin{align*}
\norm{E} \leq 4\max_{u'\in\mathcal{U},v'\in\mathcal{V}} \abs{u^{'T}Ev'}.
\end{align*}
For any $u'\in\mathcal{U},v'\in\mathcal{V}$, we have each $u'^{T}\epsilon_i$ being an independent  $\text{SG}(\sigma^2)$  and then $u'^{T}Ev'\sim \text{SG}(\sigma^2)$. Note $\abs{U}\leq 9^p\leq e^{3p}$ and similarly $\abs{V}\leq e^{3n}$.
Then by the tail probability of sub-Gaussian random variable and by the union bound, we have
\begin{align*}
\pbr{\norm{E}\leq 4t \sigma(\sqrt{n} + \sqrt{p})} &\leq \pbr{\max_{u'\in\mathcal{U},v'\in\mathcal{V}} \abs{u'^{T}Ev'}\leq t \sigma(\sqrt{n} + \sqrt{p})} \\
&\leq \abs{U}\abs{V}\ebr{-\frac{t^2\br{\sqrt{n}+\sqrt{p}}^2}{2}}\\
&\leq \ebr{-\frac{(t^2-3)n}{2}},
\end{align*}
for any $t\geq 2$.
\end{proof}

\begin{lemma}\label{lem:sub_gaussian_projection}
Let $X\sim \text{SG}_d(\sigma^2)$. Consider any $k\leq d$. For any matrix $U=(u_1,\ldots,u_k)\in\mathr^{d\times k}$ that is independent of $X$ and is with orthogonal columns $\{u_i\}_{i\in[k]}$. We have
\begin{align*}
\pbr{\norm{UU^TX}^2 \geq \sigma^2(k+2\sqrt{kt} + 2t)} \leq e^{-t}.
\end{align*}
\end{lemma}
\begin{proof}
Note that $\text{tr}(UU^T) =\text{tr}((UU^T)^2)=k $ and $\norm{UU^T}=1$. This is a direct consequence of Theorem 1 in \cite{hsu2012tail} for  concentration of quadratic forms of sub-Gaussian random vectors.
\end{proof}

\begin{proof}[Proof of Proposition \ref{prop:poly}]
Define $\hat P= \sum_{i\in[r]}\hat \lambda_i\hat u_i\hat v_i^T$.
Due to the fact that $\hat P$  is the best rank-$r$ approximation of $X$ in spectral norm and $P$ is rank-$\kr$, under the assumption that $\kr\leq r$, we have that 
\begin{align*}
\norm{\hat P - X} \leq  \norm{P-X}=\|E\|.
\end{align*}
Since $r\leq k$ is assumed, the rank of $\hat P - P$ his at most $2k$, and we have
\begin{align}
\fnorm{\hat P - P} &\leq \sqrt{2k}\norm{\hat P-P } \leq \sqrt{2k}\br{\norm{\hat P - X} +  \|P-X\|}
\leq 2\sqrt{2k} \norm{E}\label{eqn:2.1_E_op_norm}
\end{align}
Now, denote  $\hat \Theta:=(\hat \theta_{\hat z_1},\hat \theta_{\hat z_2},\ldots, \hat \theta_{\hat z_n})$. 
Since $\hat \Theta$ is the solution to the $k$-means objective  (\ref{eqn:spectral}), we have that
\begin{align*}
\fnorm{\hat \Theta - \hat P} \leq \fnorm{P-\hat  P}.
\end{align*}
Hence, by the triangle inequality, we obtain that 
\begin{align*}
\fnorm{\hat \Theta - P}\leq 2 \fnorm{\hat  P -P} \leq 4\sqrt{2k} \norm{E}.
\end{align*}
Now, define the set $S$ as
\begin{align*}
S = \cbr{i\in[n]: \norm{\hat \theta_{\hat z_i} - \theta^*_{z^*_i}} > \frac{\Delta }{2}}.
\end{align*}
Since $\cbr{\hat \theta_{\hat z_i} - \theta^*_{z^*_i}}_{i\in[n]}$ are exactly the  columns of $\hat \Theta - P$, we have that
\begin{align*}
\abs{S} \leq \frac{\fnorm{\hat \Theta - P}^2}{\br{\Delta /2}^2} \leq \frac{128k\norm{E}^2}{\Delta^2}. 
\end{align*}
Under the assumption (\ref{eqn:Delta_poly}) we have 
\begin{align*}
\frac{\beta \Delta^2 n}{ k^2\norm{E}^2}\geq 256,
\end{align*}
which implies
\begin{align*}
\abs{S} \leq \frac{\beta n}{2k}.
\end{align*}
We now  show that all the data points in $S^C$ are correctly clustered. We define
\begin{align*}
C_j=\left\{i\in[n]:z^*_i=j,i\in S^C\right\},~ j\in[k].
\end{align*}
The following holds: 
\begin{itemize}
\item For each $j\in[k]$, $C_j$ cannot be empty, as $|C_j|\geq |\{i:z^*_i=j\}| - |S|>0$.
\item For each pair $j,l\in[k],j\neq l$, there cannot exist some $i\in C_j,i'\in C_l$ such that $\hat z_i=\hat z_{i'}$. Otherwise $\hat \theta_{\hat z_i} = \hat \theta_{\hat z_{i'}}$ which would imply
\begin{align*}
\norm{\theta^*_j - \theta^*_l}  & = \norm{\theta^*_{z^*_{i}} - \theta^*_{z^*_{i'}}} \\ & \leq \norm{\theta^*_{z^*_{i}} -\hat \theta_{\hat z_i} } + \norm{\hat \theta_{\hat z_i} -\hat \theta_{\hat z_{i'}}} +  \norm{\hat \theta_{\hat z_{i'}}- \theta^*_{z^*_{i'}}}  <  \Delta,
\end{align*}
contradicting  with the definition of $\Delta$.
\end{itemize}
Since $\hat z_i$ can only take values in $[k]$, we conclude that the sets $\{\hat z_i:i\in C_j\}$ are disjoint for all  $j\in[k]$. That is, there exists a permutation $\phi\in\Phi$, such that
\begin{align*}
\hat z_i = \phi(j),~ i\in C_j,~ j\in[k].
\end{align*}
This implies that $\sum_{i\in S^C}\mathbb{I}\{\hat z_i \neq \phi (z^*_i)\}=0$. Hence, we obtain that 
\begin{align*}
|\{i\in[n]:\hat z_i \neq \phi(z^*_i)\}| \leq \abs{S}\leq \frac{128k\norm{E}^2}{\Delta^2}.
\end{align*}
Since $\abs{S}\leq \frac{\beta n}{2k}$ (which means $\ell(\hat z,z^*)\leq \frac{\beta n}{2k}$ from the above display), for any $\psi\in \Phi$ such that $\psi\neq \phi$, we have $|\{i\in[n]:\hat z_i \neq \psi(z^*_i)\}| \geq 2\beta n/k-\abs{S}\geq \beta n/k$. As a result, we have
\begin{align*}
\ell(\hat z,z^*)= \frac{1}{n}|\{i\in[n]:\hat z_i \neq \phi(z^*_i)\}| \leq   \frac{128k\norm{E}^2}{n\Delta^2}.
\end{align*}
Moreover, for each $a\in[k]$,  we have
\begin{align*}
\norm{\hat \theta_{\phi(a)} - \theta^*_a}^2 \leq \frac{\fnorm{\hat \Theta - P}^2}{\abs{\{i\in[n]:\hat z_i = \phi(a), z^*_i=a\}}} \leq  \frac{\fnorm{\hat \Theta - P}^2}{\frac{\beta n}{k} -\abs{S}} \leq  \frac{64k^2\norm{E}^2}{\beta n} %
\end{align*}
\end{proof}

\end{document}